\documentclass[11pt,letterpaper,reqno]{amsart}
\usepackage{amssymb,amsmath,amsfonts,amsthm,enumerate,stmaryrd}
\usepackage{mathtools}
\usepackage{mathrsfs}
\usepackage{empheq}
\usepackage[left=.75 in, right=.75 in,top=.75 in, bottom=.75 in]{geometry}
\usepackage{stackengine}
\usepackage{nameref,hyperref,cleveref}
\usepackage{soul}

\newtheorem{theorem}{Theorem}[section]
\newtheorem{lemma}[theorem]{Lemma}
\newtheorem{proposition}[theorem]{Proposition}

\numberwithin{equation}{section}

\newcommand{\abs}[1]{\left\lvert#1\right\rvert}

\newcommand{\norm}[1]{\left\lVert#1\right\rVert}

\newcommand{\re}{\mbox{Re}}
\newcommand{\im}{\mbox{Im}}
\newcommand{\supp}{\mbox{supp}}
\stackMath
\newcommand\tsup[2][2]{%
 \def\useanchorwidth{T}%
  \ifnum#1>1%
    \stackon[-.5pt]{\tsup[\numexpr#1-1\relax]{#2}}{\scriptscriptstyle\sim}%
  \else%
    \stackon[.5pt]{#2}{\scriptscriptstyle\sim}%
  \fi%
}

\title[Traveling wave solutions to general N-S-F]{Traveling wave solutions to a general incompressible Navier-Stokes-Fourier system with free boundary}

\author{Jae Ho Choi}
\address{
Department of Mathematical Sciences\\
Carnegie Mellon University\\
Pittsburgh, PA 15213, USA
}
\email[J.H. Choi]{jaechoi@andrew.cmu.edu}

\author{Ian Tice}
\address{
Department of Mathematical Sciences\\
Carnegie Mellon University\\
Pittsburgh, PA 15213, USA
}
\email[I. Tice]{iantice@andrew.cmu.edu}
\thanks{I. Tice was supported by an NSF Grant (DMS \#2508400)}



\begin{document}

\begin{abstract} 
We study traveling wave solutions to the free boundary problem associated to a generalized Navier-Stokes Fourier system, which models a viscous, incompressible, heat-conducting fluid.  The fluid is assumed to occupy a horizontally infinite strip-like domain with flat rigid bottom and moving upper surface.  The fluid is acted upon by gravity as well as external sources of bulk force and boundary stress and an external heat source.  Additionally, we allow for temperature-dependent viscosity and capillary coefficients, the latter of which gives rise to  Marangoni stresses on the free surface.  We develop a small data well-posedness theory in Sobolev spaces that shows that if the sources of force, stress, and heat are small, then there exists a unique solution depending continuously on these data.  
\end{abstract}

\maketitle

\section{Introduction}

\subsection{Problem formulation}
In this paper, we consider the dynamics of a body of fluid occupying a ``layer-like" domain, an infinite expanse of finite depth that is bounded below by a rigid level surface and above by a free surface bordering the atmosphere. To describe the setting more precisely, we first make a few notational conventions. Throughout the paper, we assume that \(2 \leq n \in \mathbb{N}\) and represent a point \(x \in \mathbb{R}^{n}\) by \(x=(x',x_{n}) \in \mathbb{R}^{n-1} \times \mathbb{R}\). For any continuous function \(\mathrm{k}: \mathbb{R}^{n-1} \to (0,\infty)\), we write
\begin{align}
    \Omega_{\mathrm{k}} =\{x \in \mathbb{R}^{n} \mid 0<x_{n}< \mathrm{k}(x')\} \subseteq \mathbb{R}^{n} \text{ and }
    \Sigma_{\mathrm{k}} =\{x \in \mathbb{R}^{n} \mid x_n = \mathrm{k}(x') \mbox{ for all \(x' \in \mathbb{R}^{n-1}\)}\} \subseteq \mathbb{R}^{n}.
\end{align}
The layer-like domain occupied by the fluid at time \(t \geq 0\) is then \(\Omega_{b+\zeta(\cdot, t)} \subseteq \mathbb{R}^{n}\) for a given equilibrium depth parameter  \(b>0\) and a continuous ``free surface'' function \(\zeta: \mathbb{R}^{n-1} \times [0,\infty) \to (-b,0)\). At this time, the fluid is bounded below by the fixed surface \(\Sigma_{0}\) and above by the moving surface \(\Sigma_{b+\zeta(\cdot,t)}\).

For our system, we consider a heat-conducting, viscous, incompressible fluid of constant density (normalized to unity without loss of generality) obeying the system 
\begin{equation}\label{NSF_general}
\begin{cases}
    \partial_{t}v+v \cdot \nabla v + \nabla\cdot \mathrm{S}=\mathrm{F}  \\
    \nabla \cdot v = 0 \\ 
    \partial_{t}\theta+ v \cdot \nabla \theta + \nabla \cdot \Psi  = 0
\end{cases}
\end{equation}
at each time  \(t \geq 0\) in the fluid domain $\Omega_{b+\zeta(\cdot,t)}$, where \(v\) is the fluid velocity field,  \(\mathrm{S}\) is the stress tensor, \(\mathrm{F}\) is the external bulk force, $\theta$ is the temperature, and $\Psi$ is the heat flux.   The symmetric stress tensor \(\mathrm{S}\) is given by $\mathrm{S}=P I- \Pi$, where \(P\) is the scalar pressure, \(I \in \mathbb{R}^{n \times n}\) is the \(n \times n\) identity matrix, and \(\Pi\) is the viscous stress tensor.  To close the description of the bulk dynamics of the fluid, we now specify how \(\Pi\) and $\Psi$ depend on the unknowns of the system.

Denoting the set of real symmetric $n \times n$ matrices by \(\mathbb{R}_{sym}^{n \times n}\), we posit that $\Pi = \Pi(\theta,\mathbb{D} v)$ for a given smooth function \(\Pi: \mathbb{R} \times \mathbb{R}_{sym}^{n \times n} \to \mathbb{R}_{sym}^{n \times n}\), where \(\mathbb{D}v = (\nabla v)+(\nabla v)^{t} \in \mathbb{R}_{sym}^{n \times n}\) is the symmetrized gradient of $v$. In addition, we assume that there exists a reference temperature $\theta_0 \in \mathbb{R}$ at which the linearized identity $D \Pi(\theta_0,0)(\theta,M) = \mu M$ holds for a given viscosity coefficient $\mu >0$ and $\Pi(\cdot,0) =0$. The simplest example of such a $\Pi$ is $\Pi_{NS}(\theta,M) = \mu M$, which corresponds to a standard (Newtonian) Navier-Stokes system in which the viscosity does not depend on the temperature. The simplest model that allows for temperature dependence is perhaps $\Pi_{TD}(\theta,M)= m(\theta) M$ for $m : \mathbb{R} \to \mathbb{R}^+$ a given smooth function with $\mu = m(\theta_0)$.

For the heat flux, we posit that $\Psi = \Psi(\theta,\nabla \theta)$ for a given smooth function $\Psi : \mathbb{R} \times \mathbb{R}^n \to \mathbb{R}^n$. In addition, we assume that the linearized identity $D \Psi(\theta_0,0)(\theta,z) = -\kappa z$ holds for a given thermal conductivity coefficient $\kappa >0$ and $\Psi(0,\cdot)=0$. The simplest example is given by Fourier's law, $\Psi_F(\theta,z) = -\kappa z$.

When $\Pi = \Pi_{NS}$ and $\Psi = \Psi_F$, the system \eqref{NSF_general} is often referred to as the Navier-Stokes-Fourier (NSF) system. In summary, we consider the general viscous stress and heat flux for which the NSF system is the leading-order part of our system, perturbatively speaking. Although general continuum mechanics arguments (see, for instance, \cite{Gurtin_Fried_Anand_2010}) place various additional constraints on $\Pi$ and $\Psi$, we can avoid them due to the linearized identities and the NSF perturbation framework.

Having specified the bulk equations, we now turn to a discussion of the precise forms of the bulk force and the free boundary stress. We assume that the bulk force decomposes as $\mathrm{F} = -\mathfrak{g} e_n + \mathrm{f}_{r}$, where the first term corresponds to a constant gravitational field of strength \(\mathfrak{g}>0\) pointing downward and the second term is a residual bulk force to be specified later.  On the free surface, we assume that the fluid is subjected to a constant pressure due to the atmosphere, written \(P_{ext} \in \mathbb{R}\), and an external stress generated by the stress tensor $\mathrm{t}_{ext} : \Sigma_{b + \zeta(\cdot,t)} \to \mathbb{R}^{n \times n}_{sym}$.  We also  allow for an external source of heat, modeled by a term $\mathrm{h}_{ext} : \Sigma_{b + \zeta(\cdot,t)} \to \mathbb{R}$.  Lastly, the free surface is under the effect of surface tension, with the capillary coefficient given by $\varsigma(\theta)$ for $\varsigma : \mathbb{R} \to \mathbb{R}^+$ a given smooth function.

Combining these modeling assumptions with \eqref{NSF_general}, we arrive at the full system studied in this paper:
\begin{align} \label{eqns}
    \begin{cases}
        \partial_{t}v + v \cdot \nabla v + \nabla \cdot (PI - \Pi(\theta,\mathbb{D} v))= - \mathfrak{g} e_{n}+\mathrm{f}_{r} & \mbox{in \(\Omega_{b+\zeta(\cdot,t)}\)} \\
        \nabla \cdot v = 0 &\mbox{in \(\Omega_{b+\zeta(\cdot,t)}\)} \\
        \partial_{t}\theta+ v \cdot \nabla \theta + \nabla \cdot \Psi(\theta,\nabla\theta) = 0 &\mbox{in \(\Omega_{b+\zeta(\cdot,t)}\)} \\
        -(PI - \Pi(\theta,\mathbb{D} v)) \nu(\zeta) = \varsigma(\theta)\mathcal{H}(\zeta)\nu(\zeta) + \nabla_{\Sigma}[\varsigma(\theta)]-P_{ext}\nu(\zeta)+\mathrm{t}_{ext} \nu(\zeta) &\mbox{on \(\Sigma_{b+\zeta(\cdot,t)}\)} \\
        \partial_{t}\zeta= v \cdot (-\nabla'\zeta,1) &\mbox{on \(\Sigma_{b+\zeta(\cdot,t)}\)} \\
        \Psi(\theta,\nabla\theta) \cdot \nu(\zeta) = -\mathrm{h}_{ext} &\mbox{on \(\Sigma_{b+\zeta(\cdot,t)}\)} \\
        v=0 &\mbox{on \(\Sigma_{0}\)} \\
        \theta=\theta_0 &\mbox{on \(\Sigma_{0}\)}.
    \end{cases}
\end{align}
Here, \(\nu(\zeta)=(-\nabla' \zeta,1)/\sqrt{1+\lvert\nabla' \zeta\rvert^{2}}\) is the outward-pointing unit normal vector on \(\Sigma_{b+\zeta(\cdot,t)}\), \(\mathcal{H}\) is the mean curvature operator which acts on \(\zeta\) via $\mathcal{H}(\zeta)=\nabla' \cdot [ (1+\lvert\nabla' \zeta\rvert^{2})^{-1/2}\nabla' \zeta ]$, and \(\nabla_{\Sigma}\) is the surface gradient on \(\Sigma=\Sigma_{b+\zeta(\cdot,t)}\) which acts on a function \(\mathrm{k}: \Omega_{b+\zeta(\cdot,t)} \to \mathbb{R}\) via \(\nabla_{\Sigma}\mathrm{k}=(I-\nu \otimes \nu)\nabla \mathrm{k} \mid_{\Sigma}\). The term $\nabla_{\Sigma}[\varsigma(\theta)]$ is known as the Marangoni stress. The precise meaning of the reference temperature $\theta_0 \in \mathbb{R}$ introduced in the assumptions about $\Pi$ and $\Psi$ arises from the final equation, which states that the fluid's temperature match $\theta_0$ on the bottom surface. That is, $\theta_0$ is the fixed temperature of the rigid, level solid underneath the fluid.

\subsection{Traveling waves}\label{sec_intro_tw}

Our main goal in this paper is to construct traveling wave solutions to \eqref{eqns}, i.e., solutions that are stationary when viewed in a coordinate system moving parallel to $\Sigma_0$ with velocity $\gamma e_1$, where $\gamma \in \mathbb{R} \backslash \{0\}$. Here, the choice of direction $e_1 \in \mathbb{R}^n$ does not engender a loss of generality due to the invariance of \eqref{eqns} under horizontal rotations. For such solutions to exist, the residual bulk force \(\mathrm{f}_{r}\), external free boundary stress \(\mathrm{t}_{ext}\), and external heat source \(\mathrm{h}_{ext}\) must be stationary in the traveling frame, i.e.,
\begin{equation}
    \mathrm{f}_{r}(x,t) = \mathrm{f}(x- \gamma t e_1),  \; \mathrm{t}_{ext}(x,t)  = \mathrm{t}(x- \gamma t e_1), \text{ and }\mathrm{h}_{ext}(x,t) = \mathrm{h}(x- \gamma t e_1)
\end{equation}
for given $\mathrm{f}$, $\mathrm{t}$, and $\mathrm{h}$. We impose a traveling wave ansatz on the unknowns and remove the constant terms $-\mathfrak{g} e_n$, $P_{ext}$, and $\theta_0$ from the system by writing 
\begin{align}
 &\zeta(x,t) = \eta(x-\gamma t e_1), \quad v(x,t) = w(x-\gamma t e_1), \quad \theta(x,t) = \theta_0 + \phi(x-\gamma t e_1), \nonumber \\
 &\text{and }P(x,t) = q(x-\gamma t e_1) +P_{ext} - \mathfrak{g}(x_n-b-\eta(x-\gamma t e_1)).
\end{align}
Here, the new unknowns \(w\), \(q\), and \(\phi\) are defined on the stationary yet unknown set $\Omega_{b+\eta}$. For notational brevity, we shift the surface tension, heat flux, and viscous stress functions via \(\sigma(r)=\varsigma(r+\theta_0)\), \(\Phi(r,z)=\Psi(r+\theta_0,z)\), and \(\Gamma(r,M)=\Pi(r+\theta_0,M)\). In particular, \(\Phi\) and \(\Gamma\) satisfy  $D \Phi(0,0)(\theta,z) = -\kappa z$ and $\quad D \Gamma(0,0)(\theta,M) = \mu M$. For the new unknowns, the problem \eqref{eqns} then becomes
\begin{align} \label{eqnsmf}
    \begin{cases}
        -\gamma\partial_{1}w+w \cdot \nabla w + \nabla \cdot (qI - \Gamma(\phi,\mathbb{D}w)) +\mathfrak{g}(\nabla'\eta,0) = \mathrm{f}  &\mbox{ in \(\Omega_{b+\eta}\)} \\
        \nabla \cdot w = 0 &\mbox{ in \(\Omega_{b+\eta}\)} \\
        -\gamma\partial_{1} \phi + w \cdot \nabla \phi +\nabla \cdot \Phi(\phi,\nabla \phi) = 0 &\mbox{ in \(\Omega_{b+\eta}\)} \\
        -(qI - \Gamma(\phi,\mathbb{D}w))\nu(\eta) = \sigma(\phi)\mathcal{H}(\eta)\nu(\eta) + \nabla_{\Sigma}[\sigma(\phi)]+ \mathrm{t} \nu(\eta) &\mbox{ on \(\Sigma_{b+\eta}\)} \\
        -\gamma\partial_{1}\eta = w \cdot (-\nabla ' \eta,1) &\mbox{ on \(\Sigma_{b+\eta}\)} \\
        \Phi(\phi,\nabla\phi)\cdot \nu(\eta) = -\mathrm{h} &\mbox{ on \(\Sigma_{b+\eta}\)} \\
        w = 0 &\mbox{ on \(\Sigma_{0}\)} \\
        \phi = 0 &\mbox{ on \(\Sigma_{0}\)}.
    \end{cases}
\end{align}

Much is known about the construction of traveling wave solutions to the incompressible Euler system, a classic topic in mathematical fluid mechanics. We refer to the survey articles \cite{Toland_1996,Groves_2004,Strauss_2010,HHSTWWW_2022} for a thorough treatment of this literature and directions of modern research.  In contrast, the mathematical analysis of traveling wave solutions with viscosity is relatively new.  The first well-posedness theory for the non-heat conducting version of \eqref{eqnsmf} was developed in \cite{leoni2023traveling}. Since then, similar results have been established for multi-layer fluids \cite{MR4337506}, tilted and periodic fluids \cite{MR4609068,nguyen_Banihashemi_2025}, fluids with Navier-slip boundary conditions \cite{MR4785303}, stationary fluids \cite{MR4787851}, 
and compressible fluids \cite{stevenson2023wellposedness}. There are also well-posedness theories for traveling wave solutions for the closely related Darcy flow problem \cite{MR4690615,Nguyen_2026,MR4797733,nguyen_stevenson_2025} as well as the viscous Saint-Venant problem \cite{stevenson2023shallow,stev_tice_big_waves} and the bore-wave problem \cite{stev_tice_bore_waves}. Outside the free boundary context, the well-posedness of the full dynamic problem \eqref{eqns} and related models  has been studied extensively: see, for instance,  \cite{MR1373218,MR2410236,MR3280824,MR3680944,MR3828345,MR3794117,MR3850090,MR3904564,MR4172578,MR4423144,MR4357119,MR4896668}.  To the best of our knowledge, though, there are no results on the general dynamics of the free boundary problem \eqref{eqns}.

The highlight of this paper is that we show that a viscous traveling wave can emerge due to thermal effects. Prior works have shown that the bulk force $\mathrm{f}$ and external free boundary stress $\mathrm{t}$ can serve as mechanisms by which traveling waves are generated. In our model, we show that they can be neglected, as long as the heat source term $\mathrm{h}$ is not.  In other words, a traveling heat source can drive a viscous incompressible fluid into a traveling wave state by itself.

\subsection{Reformulation in a fixed coordinate system}

It is convenient to recast the problem \eqref{eqnsmf} in a known fixed coordinate system. To do this, we use the unknown free surface function $\eta$ to construct a diffeomorphism that maps into the flattened domain $\Omega = \Omega_b = \mathbb{R}^{n-1} \times (0,b)$. For any given  $\eta : \mathbb{R}^{n-1} \to (-b,\infty)$, we define the map \(\mathfrak{F}_{\eta}: cl(\Omega) \to cl(\Omega_{b+\eta})\) via $ \mathfrak{F}_{\eta}(x) = (x',x_n(1+ \eta(x')/b))$.
This map is a bijection with inverse $\mathfrak{F}_{\eta}^{-1}(y)=(y',y_{n}b/(b+\eta(y')))$. It is a homeomorphism when $\eta$ is continuous and a $C^k$ diffeomorphism when $\eta$ is $C^k$.  If \(\eta\) is differentiable, then
\begin{align}
    &\nabla \mathfrak{F}_{\eta}(x) =
    \begin{pmatrix}
        I_{(n-1) \times (n-1)} & 0_{(n-1) \times 1} \\
        x_{n}\frac{\nabla' \eta(x')}{b} & 1+\frac{\eta(x')}{b}
    \end{pmatrix}
\end{align}
and in this case, we define
\begin{equation} \label{AJ}
    J(x)= \det\nabla \mathfrak{F}_{\eta}(x) = 1+\frac{\eta(x')}{b} \text{ and } \mathcal{A}(x) = (\nabla \mathfrak{F}_{\eta}(x))^{-t} =
    \begin{pmatrix}
        I_{(n-1) \times (n-1)} & -x_{n}\frac{\nabla' \eta(x')}{b+\eta(x')} \\
        0_{1 \times (n-1)} & \frac{b}{b+\eta(x')}
    \end{pmatrix}
    .
\end{equation}
If \(\mathrm{k}: cl(\Omega_{b+\eta}) \to \mathbb{R}\) is a differentiable function, then for each \(i \in \{1, \dots, n\}\), \(\partial_{i}\mathrm{k} \circ \mathfrak{F}_\eta = \partial_{i}^{\mathcal{A}}(\mathrm{k} \circ \mathfrak{F}_{\eta})\), where \(\partial_{i}^{\mathcal{A}}=\mathcal{A}_{ij}\partial_{j}\) in Einstein notation.

Having defined the map $\mathfrak{F}_{\eta}$, we can now reformulate the problem \eqref{eqnsmf} in $\Omega$ by setting \(u=w \circ \mathfrak{F}_{\eta}\), \(p=q \circ \mathfrak{F}_{\eta}\) and \(\psi=\phi\circ\mathfrak{F}_{\eta}\). Provided that $\eta$ is at least $C^2$, we arrive at the equivalent reformulation
\begin{align} \label{fprsmf}
    \begin{cases}
        -\gamma \partial_{1}^{\mathcal{A}}u + u \cdot \nabla_{\mathcal{A}}u+\mathfrak{g}(\nabla'\eta,0) + \nabla_{\mathcal{A}}p-\nabla_{\mathcal{A}}\cdot \Gamma(\psi,\mathbb{D}_{\mathcal{A}}u)=\mathrm{f}\circ \mathfrak{F}_{\eta}  &\mbox{in \(\Omega\)}\\
        \nabla_{\mathcal{A}} \cdot u = 0 &\mbox{in \(\Omega\)} \\
        -\gamma \partial_{1}^{\mathcal{A}}\psi + u \cdot \nabla_{\mathcal{A}}\psi +\nabla_{\mathcal{A}} \cdot \Phi(\psi,\nabla_{\mathcal{A}}\psi) = 0 &\mbox{in \(\Omega\)} \\
        (-pI+\Gamma(\psi,\mathbb{D}_{\mathcal{A}}u))\nu(\eta) = \sigma(\psi)\mathcal{H}(\eta)\nu(\eta)+\nabla_{\Sigma}^{\mathcal{A}}[\sigma(\psi)]+(\mathrm{t} \circ \mathfrak{F}_{\eta})\nu(\eta) &\mbox{on \(\Sigma_{b}\)} \\
        -\gamma \partial_{1}\eta =u \cdot (-\nabla'\eta,1) &\mbox{on \(\Sigma_{b}\)} \\
        \Phi(\psi,\nabla_{\mathcal{A}}\psi)\cdot\nu(\eta) = -\mathrm{h} \circ \mathfrak{F}_{\eta}  &\mbox{on \(\Sigma_{b}\)} \\
        u = 0 &\mbox{on \(\Sigma_{0}\)} \\
        \psi = 0 &\mbox{on \(\Sigma_{0}\)},
    \end{cases}
\end{align}
where \(\nabla_{\mathcal{A}}=(\partial_{1}^{\mathcal{A}}, \dots, \partial_{n}^{\mathcal{A}})^{t}\); \(\mathbb{D}_{\mathcal{A}}=\nabla_{\mathcal{A}}+(\nabla_{\mathcal{A}})^{t}\), and \(\nabla_{\Sigma}^{\mathcal{A}}=(I-\nu \otimes \nu)\nabla_{\mathcal{A}}\mid_{\Sigma}\).

\subsection{Main theorems}\label{mtps}

In order to state the two main theorems of the paper, we need to establish four notational conventions. Firstly, solutions to \eqref{eqnsmf} are constructed using the Hilbert space 
\begin{align} \label{Xs}
    \mathcal{X}^{s}=\{(u,\psi,p,\eta) \in \prescript{}{0}{H}^{s+2}(\Omega;\mathbb{R}^{n}) \times \prescript{}{0}{H}^{s+2}(\Omega;\mathbb{R})\times H^{s+1}(\Omega;\mathbb{R}) \times X^{s+5/2}(\mathbb{R}^{n-1};\mathbb{R})\}, \quad s \geq 0,
\end{align}
where \(H^{t}(\Omega;\mathbb{R}^{m})\) is the standard \(L^{2}\)-based Sobolev space of order \(t \geq 0\) and \(\prescript{}{0}{H}^{t}(\Omega;\mathbb{R}^{m})\) is the subspace whose members have the boundary value of zero on \(\Sigma_{0}\). The space \(X^{s+5/2}(\mathbb{R}^{n-1};\mathbb{R})\), defined in \eqref{xxs}, is a specialized container space for the free surface function \(\eta\). Secondly, given $U \in \{\Omega,\mathbb{R}^{n-1}\}$ and a finite-dimensional normed vector space $V$, we write $C^k_{b}(U;V)$ for the space of $k-$times continuously differentiable maps from $U$ to $V$, whose derivatives are all bounded. We write $C^k_0(U;V)$ for the subspace whose members' derivatives all decay to zero as $|x| \to \infty$. Thirdly, we reference the norms of certain bilinear forms \(Q_{1}\) and \(Q_{2}\) defined in \eqref{Q1} and \eqref{Q2} to define a nonempty open set to which the various physical parameters must belong. These norms are constants that depend on the background dimension \(n \geq 2\) and the equilibrium depth parameter \(b>0\). To motivate the last notational convention, we note that under certain modeling contexts, it is reasonable to consider the residual bulk force \(\mathrm{f}\), free boundary stress \(\mathrm{t}\), and heat source \(\mathrm{h}\) that do not depend on the vertical spatial variable. In such cases, the regularity demands can be relaxed. To accommodate for these scenarios, we decompose these data into a part that depends on all spatial variables and a part that depends only on the horizontal ones. To execute these distinctions, we define and use the following maps. For any function \(\mathrm{k}:\mathbb{R}^{n-1} \to (0,\infty)\), \(L_{\mathrm{k}}:H^{s}(\mathbb{R}^{n-1};\mathbb{R}^{n}) \to H^{s}(\Omega_{\mathrm{k}};\mathbb{R}^{n})\) is defined via \((L_{\mathrm{k}}f)(x)=f(x')\) and \(S_{\mathrm{k}}: H^{s+1/2}(\mathbb{R}^{n-1};\mathbb{R}^{d}) \to H^{s+1/2}(\Sigma_{\mathrm{k}};\mathbb{R}^{d})\) is given by \((S_{\mathrm{k}}T)(x)=T(x')\). Having established these notational conventions, we are now ready to state the main theorems.

Our first result establishes the well-posedness of \eqref{fprsmf} for small data.

\begin{theorem} \label{1.1analog}
    Suppose that \(n \geq 2\), \(\mathbb{N} \ni s \geq 1+\lfloor n/2 \rfloor\), and \(\mathcal{X}^{s}\) is given by \eqref{Xs}. If \((\mu,\kappa,\sigma'(0)) \in \mathbb{R}^{+} \times \mathbb{R}^{+} \times \mathbb{R}\) satisfy
    \begin{align}
        \max\left\{\frac{\norm{Q_{1}}^{2}}{4},\frac{\norm{Q_{2}}^{2}}{4}\right\}\abs{\sigma'(0)}^{2} < 2\mu\kappa,
    \end{align}
    where \(Q_{1}\) and \(Q_{2}\) are bilinear forms defined in \eqref{Q1} and \eqref{Q2}, then there exist open sets
    \begin{align}
        \mathcal{U}^{s} \subseteq & \mathbb{R} \setminus \{0\} \times H^{s+1}(\mathbb{R}^{n};\mathbb{R}^{n}) \times H^{s}(\mathbb{R}^{n-1};\mathbb{R}^{n}) \times H^{s+2}(\mathbb{R}^{n};\mathbb{R}_{sym}^{n \times n}) \nonumber\\
        &\times H^{s+1/2}(\mathbb{R}^{n-1};\mathbb{R}_{sym}^{n \times n}) \times H^{s+2}(\mathbb{R}^{n};\mathbb{R}) \times H^{s+1/2}(\mathbb{R}^{n-1};\mathbb{R})
    \end{align}
    and \(\mathcal{O}^{s}\subseteq \mathcal{X}^{s}\) such that the following hold:
    \begin{enumerate}
        \item We have $(0,0,0,0) \in \mathcal{O}^{s}$, as well as the inclusion
        \begin{equation}
            \mathcal{O}^{s} \subseteq  C_{0}^{s+1-\lfloor n/2\rfloor}(\Omega;\mathbb{R}^{n}) \times C_{0}^{s+1-\lfloor n/2\rfloor}(\Omega;\mathbb{R}) \times C_{0}^{s-\lfloor n/2 \rfloor}(\Omega;\mathbb{R}) \times  C_{0}^{s+1-\lfloor (n-1)/2\rfloor}(\mathbb{R}^{n-1};\mathbb{R}). 
        \end{equation}
        Moreover, if \((u,\psi,p,\eta) \in \mathcal{O}^{s}\), then   
        \(\max_{\mathbb{R}^{n-1}}\abs{\eta} \leq b/2\)  and \(\mathfrak{F}_{\eta}: cl(\Omega) \to cl({\Omega}_{b+\eta})\) is a bi-Lipschitz homeomorphism and a \(C^{3+\lfloor s-n/2\rfloor}\) diffeomorphism.
        
        \item \(\mathbb{R} \setminus \{0\} \times \{0\} \times \{0\} \times \{0\} \times \{0\} \times \{0\}\times \{0\} \subseteq \mathcal{U}^{s}\).
        \item For each \((\gamma,\mathfrak{f},f, \mathcal{T},T,\mathfrak{H},H) \in \mathcal{U}^{s}\), there exists a unique \((u,\psi,p,\eta) \in \mathcal{O}^{s}\) classically solving
        \begin{align} \label{sysflattened}
            \begin{cases}
                -\gamma \partial_{1}^{\mathcal{A}}u + u \cdot \nabla_{\mathcal{A}}u+\mathfrak{g}(\nabla'\eta,0) + \nabla_{\mathcal{A}}p-\nabla_{\mathcal{A}}\cdot \Gamma(\psi,\mathbb{D}_{\mathcal{A}}u) = \mathfrak{f} \circ \mathfrak{F}_{\eta}+L_{b}f &\mbox{ in \(\Omega\)} \\
                \nabla_{\mathcal{A}}\cdot u=0 &\mbox{ in \(\Omega\)} \\
                -\gamma \partial_{1}^{\mathcal{A}}\psi + u \cdot \nabla_{\mathcal{A}}\psi +\nabla_{\mathcal{A}} \cdot \Phi(\psi,\nabla_{\mathcal{A}}\psi)=0 &\mbox{ in \(\Omega\)} \\
                -p\mathcal{N}(\eta)+\Gamma(\psi,\mathbb{D}_{\mathcal{A}}u)\mathcal{N}(\eta)-\sigma(\psi)\mathcal{H}(\eta)\mathcal{N}(\eta)-\nabla_{\Sigma}^{\mathcal{A}}[\sigma(\psi)]\abs{\mathcal{N}(\eta)} = (\mathcal{T} \circ \mathfrak{F}_{\eta} \mid_{\Sigma_{b}}+S_{b}T)\mathcal{N}(\eta) &\mbox{ on \(\Sigma_{b}\)} \\
                u \cdot \mathcal{N}(\eta)+\gamma \partial_{1}\eta=0 &\mbox{ on \(\Sigma_{b}\)} \\
                \Phi(\psi,\nabla_{\mathcal{A}}\psi)\cdot\nu(\eta)=-(\mathfrak{H} \circ \mathfrak{F}_{\eta}\mid_{\Sigma_{b}}+S_{b}H )&\mbox{ on \(\Sigma_{b}\)} \\
                u=0 &\mbox{ on \(\Sigma_{0}\)} \\
                \psi=0 &\mbox{ on \(\Sigma_{0}\)},
            \end{cases}
        \end{align}
        where \(\mathcal{N}(\eta)=(-\nabla' \eta,1)\).
        \item The map \((\gamma,\mathfrak{f},f, \mathcal{T},T,\mathfrak{H},H) \in \mathcal{U}^{s} \mapsto (u,\psi,p,\eta) \in \mathcal{O}^{s}\) is \(C^{1}\) and locally Lipschitz.
    \end{enumerate}
\end{theorem}

In the absence of forcing, Theorem \ref{1.1analog} remains silent on whether there exist large nontrivial solutions in \(\mathcal{X}^{s}\) outside of the open set $\mathcal{O}^{s}$, though we don't expect them to exist.  We note in particular that the theorem produces nontrivial solutions when $\mathfrak{f}$, $f$, $\mathcal{T}$, and $T$ all vanish but at least one of $\mathfrak{H},H$ does not.  This justifies our prior claim that a traveling heat source alone can induce nontrival traveling wave states in this model.

Our solutions in the flattened formulation are sufficiently regular to be transformed back to solutions in the original, Eulerian formulation.  We state this now.

\begin{theorem} \label{1.3analog}
    Suppose that \(n \geq 2\) and \(\mathbb{N} \ni s \geq 1+\lfloor n/2 \rfloor\). Let
    \begin{align}
        \mathcal{U}^{s} \subseteq &\mathbb{R} \setminus \{0\} \times H^{s+1}(\mathbb{R}^{n};\mathbb{R}^{n}) \times H^{s}(\mathbb{R}^{n-1};\mathbb{R}^{n}) \times H^{s+2}(\mathbb{R}^{n};\mathbb{R}_{sym}^{n \times n}) \nonumber \\
        &\times H^{s+1/2}(\mathbb{R}^{n-1};\mathbb{R}_{sym}^{n \times n}) \times H^{s+2}(\mathbb{R}^{n};\mathbb{R}) \times H^{s+1/2}(\mathbb{R}^{n-1};\mathbb{R})
    \end{align}
    and \(\mathcal{O}^{s}\subseteq \mathcal{X}^{s}\) be the open sets from Theorem \ref{1.1analog}. Then for each \((\gamma,\mathfrak{f},f, \mathcal{T},T,\mathfrak{H},H) \in \mathcal{U}^{s}\), there exist
    \begin{enumerate}
        \item a free surface function \(\eta \in X^{s+5/2}(\mathbb{R}^{n-1};\mathbb{R}) \cap C_{0}^{s+1-\lfloor (n-1)/2\rfloor}(\mathbb{R}^{n-1};\mathbb{R})\) such that \(\max_{\mathbb{R}^{n-1}}\abs{\eta}<b/2\) and \(\mathfrak{F}_{\eta}\) is a bi-Lipschitz homeomorphism and a \(C^{s+1-\lfloor (n-1)/2\rfloor}\) diffeomorphism,
        \item a velocity field \(w \in \prescript{}{0}{H}^{s+2}(\Omega_{b+\eta};\mathbb{R}^{n}) \cap C_{b}^{s+1-\lfloor n/2\rfloor}(\Omega_{b+\eta};\mathbb{R}^{n})\),
        \item a temperature \(\phi \in \prescript{}{0}{H}^{s+2}(\Omega_{b+\eta};\mathbb{R}) \cap C_{b}^{s+1-\lfloor n/2\rfloor}(\Omega_{b+\eta};\mathbb{R})\),
        \item a pressure \(q \in H^{s+1}(\Omega_{b+\eta};\mathbb{R}) \cap C_{b}^{s-\lfloor n/2 \rfloor}(\Omega_{b+\eta};\mathbb{R})\), and
        \item constants \(C,R>0\)
    \end{enumerate}
    such that
    \begin{enumerate}
        \item \((w,\phi,q,\eta)\) classically solve
        \begin{align} \label{sysori}
            \begin{cases}
                -\gamma\partial_{1}w+w \cdot \nabla w + \nabla \cdot(qI-\Gamma(\phi,\mathbb{D}w))+\mathfrak{g}(\nabla'\eta,0) = \mathfrak{f}+L_{b+\eta}f &\mbox{ in \(\Omega_{b+\eta}\)} \\
                \nabla \cdot w = 0 &\mbox{ in \(\Omega_{b+\eta}\)} \\
                -\gamma\partial_{1}\phi+ w \cdot \nabla \phi +\nabla \cdot \Phi(\phi,\nabla \phi) = 0 &\mbox{ in \(\Omega_{b+\eta}\)} \\
                -(qI-\Gamma(\phi,\mathbb{D}w))\nu(\eta) = \sigma(\phi)\mathcal{H}(\eta)\nu(\eta) + \nabla_{\Sigma}[\sigma(\phi)]+(\mathcal{T}+S_{b+\eta}T)\nu(\eta) &\mbox{ on \(\Sigma_{b+\eta}\)} \\
                -\gamma\partial_{1}\eta = w \cdot (-\nabla ' \eta,1) &\mbox{ on \(\Sigma_{b+\eta}\)} \\
                \Phi(\phi,\nabla\phi)\cdot \nu(\eta) = -(\mathfrak{H}+F_{b+\eta}H) &\mbox{ on \(\Sigma_{b+\eta}\)} \\
                w = 0 &\mbox{ on \(\Sigma_{0}\)} \\
                \phi = 0 &\mbox{ on \(\Sigma_{0}\)}.
            \end{cases}
        \end{align}
        \item \((w \circ \mathfrak{F}_{\eta}, \phi \circ \mathfrak{F}_{\eta}, q \circ \mathfrak{F}_{\eta}, \eta)\in \mathcal{O}^{s} \subseteq \mathcal{X}^{s}\).
        \item If \((\gamma^{*},\mathfrak{f}^{*},f^{*}, \mathcal{T}^{*},T^{*},\mathfrak{H}^{*},H^{*}) \in \mathcal{U}^{s}\) and
        \begin{align}
            &\abs{\gamma-\gamma^{*}}+\norm{\mathfrak{f}-\mathfrak{f}^{*}}_{H^{s+1}(\mathbb{R}^{n};\mathbb{R}^{n})}+\norm{f-f^{*}}_{H^{s}(\mathbb{R}^{n-1};\mathbb{R}^{n})}+\norm{\mathcal{T}-\mathcal{T}^{*}}_{H^{s+2}(\mathbb{R}^{n};\mathbb{R}_{sym}^{n\times n})} \nonumber \\
            &+\norm{T-T^{*}}_{H^{s+1/2}(\mathbb{R}^{n-1};\mathbb{R}_{sym}^{n \times n})}+\norm{\mathfrak{H}-\mathfrak{H}^{*}}_{H^{s+2}(\mathbb{R}^{n};\mathbb{R})}+\norm{H-H^{*}}_{H^{s+1/2}(\mathbb{R}^{n-1};\mathbb{R})}<R,
        \end{align}
        then we have the local Lipschitz estimate
        \begin{align}
            &\norm{(w \circ \mathfrak{F}_{\eta}, \phi \circ \mathfrak{F}_{\eta}, q \circ \mathfrak{F}_{\eta},\eta)-(w^{*} \circ \mathfrak{F}_{\eta^{*}}, \phi^{*} \circ \mathfrak{F}_{\eta^{*}}, q^{*} \circ \mathfrak{F}_{\eta^{*}},\eta^{*})}_{\mathcal{X}^{s}} \nonumber \\
            \leq &C(\abs{\gamma-\gamma^{*}}+\norm{\mathfrak{f}-\mathfrak{f}^{*}}_{H^{s+1}(\mathbb{R}^{n};\mathbb{R}^{n})}+\norm{f-f^{*}}_{H^{s}(\mathbb{R}^{n-1};\mathbb{R}^{n})}+\norm{\mathcal{T}-\mathcal{T}^{*}}_{H^{s+2}(\mathbb{R}^{n};\mathbb{R}_{sym}^{n\times n})} \nonumber \\
            &+\norm{T-T^{*}}_{H^{s+1/2}(\mathbb{R}^{n-1};\mathbb{R}_{sym}^{n \times n})}+\norm{\mathfrak{H}-\mathfrak{H}^{*}}_{H^{s+2}(\mathbb{R}^{n};\mathbb{R})}+\norm{H-H^{*}}_{H^{s+1/2}(\mathbb{R}^{n-1};\mathbb{R})}).
        \end{align}
    \end{enumerate}
\end{theorem}

\subsection{Proof techniques and challenges} \label{ptc} 
As Theorem \ref{1.3analog} is a corollary of Theorem \ref{1.1analog}, we dedicate much of our paper to proving the latter. To prove Theorem \ref{1.1analog}, we view the system \eqref{fprsmf} through the lens of an implicit function theorem.  In this context, the problem presents all of the same challenges present in the earlier work on viscous traveling waves mentioned at the end of Section \ref{sec_intro_tw}, as well as some interesting new ones caused by the heat coupling.  In particular, the analysis of the linearized system and the detailed asymptotic development of a certain crucial Fourier multiplier are much more involved than in previous work.  We now aim to discuss this strategy and highlight new difficulties and techniques.

In the implicit function theorem context, the dependent variables are the velocity, pressure, and temperature of the fluid as well as the free boundary map \(\eta: \mathbb{R}^{n-1} \to (-b, \infty)\). The data on which these variables depend are \(\gamma\), \(\mathfrak{f}\), \(f\), \(\mathcal{T}\), \(T\), \(\mathfrak{H}\), and \(H\). In the absence of any forcing, \((u, \psi, p, \eta)=(0,0,0,0)\) is a solution to \eqref{fprsmf}. Linearizing the system with respect to the dependent variables around \((\gamma,\mathfrak{f},f, \mathcal{T},T,\mathfrak{H},H, u, \psi,p,\eta)=(\gamma,0,0,0,0,0,0,0,0,0,0)\), we obtain
\begin{align} \label{ls}
    \begin{cases}
    -\gamma \partial_{1}u + \nabla p - \mu \Delta u = -\mathfrak{g}(\nabla'\eta,0) &\mbox{in \(\Omega\)} \\
    \nabla \cdot u = 0 &\mbox{in \(\Omega\)} \\
    -\gamma \partial_{1}\psi - \kappa \Delta \psi = 0 &\mbox{in \(\Omega\)} \\
    (-pI+\mu\mathbb{D}u)e_{n} = (\sigma'(0)\nabla'\psi,0)+(0,\sigma(0)\Delta'\eta) &\mbox{on \(\Sigma_{b}\)} \\
    -\gamma\partial_{1}\eta-u_{n}=0 &\mbox{on \(\Sigma_{b}\)} \\
    -\kappa\nabla\psi \cdot e_{n} = 0 &\mbox{on \(\Sigma_{b}\)} \\
    u = 0 &\mbox{on \(\Sigma_{0}\)} \\
    \psi = 0 &\mbox{on \(\Sigma_{0}\)}.
    \end{cases}
\end{align}
We show that this corresponds to an isomorphism between the Banach spaces  \(\mathcal{X}^{s}\) and \(\mathcal{Y}^{s}\) defined in \eqref{Xs} and \eqref{Ys}, respectively.

To show that \eqref{ls} resolves into an isomorphism of Banach spaces between \(\mathcal{X}^{s}\) and \(\mathcal{Y}^{s}\), we first study the solvability of a system without $\eta$ but in a more general form in terms of the velocity, pressure, and temperature that encodes two distinct sub-problems that are essential for our analysis:
\begin{align}
    \begin{cases} \label{mothereqn}
        \widetilde{\gamma} \partial_{1}w + \nabla \cdot (rI-\mu\mathbb{D}w) = f &\mbox{in \(\Omega\)} \\
        \nabla \cdot w = g &\mbox{in \(\Omega\)} \\
        \widetilde{\gamma} \partial_{1}\theta -\kappa\Delta\theta = l &\mbox{in \(\Omega\)} \\
        (rI-\mu\mathbb{D}w)e_{n}+\alpha_{1}(\nabla'\theta,0) =k &\mbox{on \(\Sigma_{b}\)} \\
        \kappa\partial_{n}\theta +\alpha_{2}\nabla' \cdot w'= m &\mbox{on \(\Sigma_{b}\)} \\
        w = 0 &\mbox{on \(\Sigma_{0}\)} \\
        \theta = 0 &\mbox{on \(\Sigma_{0}\)}.
    \end{cases}
\end{align}
In Section \ref{gstokesadj}, we combine Lax-Milgram with regularity theory for elliptic systems to prove the well-posedness of \eqref{mothereqn} in an \(L^{2}\)-based Sobolev framework.  More precisely, we prove the following. 

\begin{proposition} \label{msys}
    Suppose that \(\widetilde{\gamma} \in \mathbb{R}\). For every \((\mu,\kappa,\alpha_{1}, \alpha_{2}) \in O_{n,b}\) in Proposition \ref{paramsp}, the bounded linear operator
    \begin{align}
        \Phi_{\widetilde{\gamma},\mu,\kappa,\alpha_{1},\alpha_{2}}: &\prescript{}{0}{H}^{s+2}(\Omega;\mathbb{R}^{n}) \times \prescript{}{0}{H}^{s+2}(\Omega;\mathbb{R}) \times H^{s+1}(\Omega;\mathbb{R}) \nonumber \\
        \to& H^{s}(\Omega;\mathbb{R}^{n}) \times H^{s+1}(\Omega;\mathbb{R}) \times H^{s}(\Omega;\mathbb{R})\times H^{s+1/2}(\Sigma_{b};\mathbb{R}^{n}) \times H^{s+1/2}(\Sigma_{b};\mathbb{R})
    \end{align}
    given by
    \begin{align}
        \Phi_{\widetilde{\gamma},\mu,\kappa,\alpha_{1},\alpha_{2}}(w,\theta,r) = (&\widetilde{\gamma}\partial_{1}w + \nabla \cdot (rI-\mu\mathbb{D}w), \nabla \cdot w, \widetilde{\gamma}\partial_{1}\theta-\kappa\Delta\theta, \nonumber\\
        &(rI-\mu\mathbb{D}w)e_{n}+\alpha_{1}(\nabla'\theta,0)\mid_{\Sigma_{b}},\kappa \partial_{n}\theta+\alpha_{2}(\nabla' \cdot w')\mid_{\Sigma_{b}})
    \end{align}
    is an isomorphism of Banach spaces.
\end{proposition}

To parlay this well-posedness result for \eqref{mothereqn} into a well-posedness result for \eqref{ls}, we first ignore the free surface and view the resulting problem as an overdetermined one for the other unknowns:
\begin{align} \label{odp}
    \begin{cases}
        -\gamma \partial_{1}u + \nabla \cdot (pI-\mu\mathbb{D}u) = f &\mbox{in \(\Omega\)} \\
        \nabla \cdot u = g &\mbox{in \(\Omega\)} \\
        -\gamma \partial_{1}\psi -\kappa\Delta\psi = l &\mbox{in \(\Omega\)} \\
        (pI-\mu\mathbb{D}u)e_{n} +\sigma'(0)(\nabla'\psi,0)= k &\mbox{on \(\Sigma_{b}\)} \\
        u_{n}=h &\mbox{on \(\Sigma_{b}\)} \\
        \kappa\partial_{n}\psi =m &\mbox{on \(\Sigma_{b}\)} \\
        u = 0 &\mbox{on \(\Sigma_{0}\)} \\
        \psi = 0 &\mbox{on \(\Sigma_{0}\)}.
    \end{cases}
\end{align}
In general, there is no solution to \eqref{odp} because we can resolve \eqref{mothereqn} into \eqref{odp} minus the condition \(u_{n}=h\) on \(\Sigma_{b}\) by setting \((\widetilde{\gamma},\alpha_{1},\alpha_{2})=(-\gamma,\sigma'(0),0)\). To characterize the solvability of \eqref{odp}, we opt to characterize the range of the forward map as the annihilator of the kernel of the formal adjoint.  To identify this adjoint, we prove that if  \(u,v \in \prescript{}{0}{H}^{2}(\Omega;\mathbb{R}^{n})\), \(\psi,\phi \in \prescript{}{0}{H}^{2}(\Omega;\mathbb{R})\), \(p,q \in H^{1}(\Omega;\mathbb{R})\), and \((u,p,\psi)\) satisfy \eqref{odp}, then 
\begin{align}
    &\int_{\Omega}f \cdot v-g q + l\phi \nonumber\\
    = &\int_{\Omega}u \cdot (\gamma \partial_{1}v+\nabla \cdot (qI-\mu\mathbb{D}v)) -p\nabla \cdot v + \psi(\gamma \partial_{1}\phi -\kappa\Delta\phi) \nonumber\\
    &+\int_{\Sigma_{b}} k \cdot v -m \phi -h(qI-\mu\mathbb{D}v)e_{n}\cdot e_{n}-((qI-\mu\mathbb{D}v)e_{n})'\cdot u'  + \psi(\sigma'(0)\nabla' \cdot v' + \kappa \partial_{n}\phi).
\end{align}
Based on this identity, we can read off the formal adjoint of \eqref{odp}:
\begin{align} \label{ap}
    \begin{cases}
        \gamma \partial_{1}v + \nabla \cdot (qI-\mu\mathbb{D}v) = f &\mbox{in \(\Omega\)} \\
        \nabla \cdot v = g &\mbox{in \(\Omega\)} \\
        \gamma \partial_{1}\phi -\kappa\Delta\phi = l &\mbox{in \(\Omega\)} \\
        ((qI-\mu\mathbb{D}v)e_{n})' =k' &\mbox{on \(\Sigma_{b}\)} \\
        \kappa\partial_{n}\phi +\sigma'(0)\nabla' \cdot v'= m &\mbox{on \(\Sigma_{b}\)} \\
        v = 0 &\mbox{on \(\Sigma_{0}\)} \\
        \phi = 0 &\mbox{on \(\Sigma_{0}\)}.
    \end{cases}
\end{align}
This is an under-determined problem because \eqref{mothereqn} resolves into \eqref{ap} plus a normal stress condition on \(\Sigma_{b}\) when \((\widetilde{\gamma},\alpha_{1},\alpha_{2})=(\gamma,0,\sigma'(0))\). Proposition \ref{msys} ensures that there exists a unique solution to the formal adjoint \eqref{ap} if it is augmented with a normal stress \(\chi e_{n}\) on \(\Sigma_{b}\) for any \(\chi \in H^{s+1/2}(\Sigma_{b};\mathbb{R})\). We summarize this in the following.

\begin{proposition} \label{compcond}
    Let \(s \geq 0\). Suppose that \((u,p,\psi) \in \prescript{}{0}{H}^{s+2}(\Omega;\mathbb{R}^{n}) \times H^{s+1}(\Omega;\mathbb{R}) \times \prescript{}{0}{H}^{s+2}(\Omega;\mathbb{R})\) solve the overdetermined system (\ref{odp}). Let \((v,q,\phi) \in \prescript{}{0}{H}^{s+2}(\Omega;\mathbb{R}^{n}) \times H^{s+1}(\Omega;\mathbb{R}) \times \prescript{}{0}{H}^{s+2}(\Omega;\mathbb{R})\) be the unique solution to the adjoint system (\ref{ap}) with \((f,g,l,k',m)=(0,0,0,0,0)\) and the normal stress condition \((qI-\mu\mathbb{D}v)e_{n} \cdot e_{n}=\chi\) on \(\Sigma_{b}\), where \(\chi \in H^{s+1/2}(\Sigma_{b};\mathbb{R})\). Then
    \begin{align} \label{cc}
        \int_{\Omega}f \cdot v - gq+l\phi = \int_{\Sigma_{b}}k \cdot v-m\phi-h\chi.
    \end{align}
\end{proposition}
Remarkably, the conditions that the data \((f,g,l,k,h,m)\) satisfy \eqref{cc} for every \(\chi \in H^{s+1/2}(\Sigma_{b};\mathbb{R})\) and
\begin{align}
    h-\int_{0}^{b}g(\cdot,x_{n})dx_{n}\in \dot{H}^{-1}(\mathbb{R}^{n-1};\mathbb{R})
\end{align}
characterize the solvability of \eqref{odp}. For ease of reference, we call the former condition the adjoint kernel compatibility condition and the latter the divergence-trace compatibility condition. The divergence-trace compatibility condition is initially obtained as a necessary condition satisfied by a solution to \eqref{odp}. The solvability of \eqref{odp} is summarized as follows, recalling that \(\mathcal{X}^{s}\) and \(\mathcal{Y}^{s}\) are defined in \eqref{Xs} and \eqref{Ys}.

\begin{proposition} \label{Zsiso}
    Suppose that \(\gamma\in \mathbb{R}\) and \((\mu,\kappa,\sigma'(0)) \in \mathbb{R}^{+} \times \mathbb{R}^{+} \times \mathbb{R}\) satisfy
    \begin{align}
        \max\left\{\frac{\norm{Q_{1}}^{2}}{4},\frac{\norm{Q_{2}}^{2}}{4}\right\} \abs{\sigma'(0)}^{2}< 2\mu\kappa,
    \end{align}
    where \(Q_{1}\) and \(Q_{2}\) are bilinear forms defined in \eqref{Q1} and \eqref{Q2}. If
    \begin{align} \label{Zs}
        \mathcal{Z}^{s}=\{(f,g,l,k,h,m) \in \mathcal{Y}^{s}\mid \mbox{\((f,g,l,k,h,m)\) satisfies \eqref{cc} for every \(\chi \in H^{s+1/2}(\Sigma_{b};\mathbb{R})\)}\},\text{ } s \geq 0
    \end{align}
    then the bounded linear map \(T: \prescript{}{0}{H}^{s+2}(\Omega;\mathbb{R}^{n}) \times \prescript{}{0}{H}^{s+2}(\Omega;\mathbb{R})\times H^{s+1}(\Omega;\mathbb{R}) \to \mathcal{Z}^{s}\) given by
    \begin{align}
        T(u,\psi,p)=(&-\gamma \partial_{1}u + \nabla \cdot (pI-\mu\mathbb{D}u), \nabla \cdot u, -\gamma \partial_{1}\psi -\kappa\Delta\psi, \nonumber \\
        &(pI-\mu\mathbb{D}u) \mid_{\Sigma_{b}} e_{n} +\sigma'(0)(\nabla'\psi \mid_{\Sigma_{b}},0), u_{n} \mid_{\Sigma_{b}}, \kappa\partial_{n}\psi \mid_{\Sigma_{b}})
    \end{align}
    is an isomorphism of Banach spaces.
\end{proposition}

To address the question of solvability of \eqref{ls}, we view it as an instance of \eqref{odp} in which all the terms containing \(\eta\) are part of the data. This allows us to take advantage of Proposition \ref{Zsiso}. To study the adjoint kernel compatibility condition which now contains \(\eta\), we take its Fourier transform with respect to the horizontal variable \(x' \in \mathbb{R}^{n-1}\). In Section \ref{fst}, we apply a general theory of Fourier symbols for translation commuting bounded linear maps to instantiate the Fourier symbols associated with the maps that send \(\chi \in H^{s+1/2}(\Sigma_{b};\mathbb{R})\) to the \(v\), \(q\), and \(\phi\) in Proposition \ref{compcond} and the traces \(v \mid_{\Sigma_{b}}\), \(\phi\mid_{\Sigma_{b}}\), and \(v_{n}\mid_{\Sigma_{b}}\) and separate off \(\chi\) from the compatibility condition. These Fourier symbols are denoted by \(\omega_{v}\), \(\omega_{q}\), \(\omega_{\phi}\), \(\omega_{v\mid_{\Sigma_{b}}}\), \(\omega_{\phi\mid_{\Sigma_{b}}}\), and \(\omega_{v_{n}\mid_{\Sigma_{b}}}\), respectively. Once the appearance of \(\chi\) is removed, the compatibility condition becomes a Fourier multiplier equation:
\begin{align} \label{wfer}
    \mbox{\(W(\xi)=\mathcal{F}(\eta)(\xi)\rho(\xi)\) for a.e. \(\xi \in \mathbb{R}^{n-1}\)},
\end{align}
where
\begin{align}
    W(\xi)=&\int_{0}^{b}\mathcal{F}(f)(\xi,x_{n})\cdot \overline{\omega_{v}(\xi,x_{n})}-\mathcal{F}(g)(\xi,x_{n})\overline{\omega_{q}(\xi,x_{n})}+\mathcal{F}(l)(\xi,x_{n})\overline{\omega_{\phi}(\xi,x_{n})}dx_{n} \nonumber\\
    &-\mathcal{F}(k)(\xi)\cdot\overline{\omega_{v\mid_{\Sigma_{b}}}(\xi)}+\mathcal{F}(m)(\xi)\overline{\omega_{\phi\mid_{\Sigma_{b}}}(\xi)}+\mathcal{F}(h)(\xi)
\end{align}
and
\begin{align}
    \rho(\xi)=(\sigma(0)4\pi^{2}\abs{\xi}^{2}+\mathfrak{g})\overline{\omega_{v_{n}\mid_{\Sigma_{b}}}(\xi)}+\gamma 2\pi i \xi_{1}.
\end{align}
The function space to which $\eta$ will belong is thus entirely determined by the symbol $\rho$ and the form of $W$, and as such we need a relatively precise understanding of behavior of $\omega_{v_{n}\mid_{\Sigma_{b}}}$.  This turns out to be much trickier than the corresponding problems for the model without heat conduction.  We study the asymptotics of  \(W(\xi)\) and \(\rho(\xi)\) in Sections \ref{instfs} and \ref{odetheory} by using a delicate synthesis of Fourier multiplier techniques and explicit ODE solution computations.  Then in Section \ref{asymptotics}, we study the asymptotics of \(\rho\) itself. Based on this knowledge, we conclude that \(X^{s+5/2}(\mathbb{R}^{n-1};\mathbb{R})\) is the container space for \(\eta\), where for $t \ge 0$ the Hilbert space
\begin{align} \label{xxs}
    X^{t}(\mathbb{R}^{d};\mathbb{R})=\{f \in \mathscr{S}'(\mathbb{R}^{d};\mathbb{R}) \mid \mbox{\(f=\overline{f}\) and \(\hat{f} \in L_{loc}^{1}(\mathbb{R}^{d};\mathbb{R})\)}\} 
\end{align}
is defined via the inner-product
\begin{align}\label{df}
    (f,g)_{X_{t}}=\int_{\mathbb{R}^{d}}\omega_{t}(\xi)\hat{f}(\xi)\overline{\hat{g}(\xi)}d\xi
\text{ for }  
\omega_{t}(\xi)=\frac{\xi_{1}^{2}+\abs{\xi}^{4}}{\abs{\xi}^{2}}\chi_{B(0,1)}(\xi)+(1+\abs{\xi}^{2})^{t}\chi_{B(0,1)^{c}}(\xi).
\end{align}
This space was first identified in \cite{leoni2023traveling} and has played the role of the container space for $\eta$ in most subsequent work.  In Section \ref{welldef}, we establish that \eqref{ls} does indeed resolve into an isomorphism of Banach spaces between \(\mathcal{X}^{s}\) and \(\mathcal{Y}^{s}\).

Lastly, we check that the map that sends \((\gamma,\mathfrak{f},f, \mathcal{T},T,\mathfrak{H},H, u, \psi,p,\eta)\) to \eqref{fprsmf} is \(C^{1}\).  This is accomplished by using a variety of nonlinear functional analytic results (products, compositions, etc) valid in our Sobolev framework.  With this in hand, we then prove Theorem \ref{1.1analog} by way of the implicit function theorem, as discussed above.

To conclude this section we establish some additional notational conventions and make existing ones more precise. Throughout the paper, we assume that \(1 \leq d \in \mathbb{N}\). If \(\zeta: \mathbb{R}^{n-1} \to \mathbb{R}\) is a Lipschitz function such that \(\inf \zeta >0\), then for \(s>1/2\), we let \(\prescript{}{0}{H}^{s}(\Omega_{\zeta};\mathbb{R}^{n})=\{u \in H^{s}(\Omega_{\zeta};\mathbb{R}^{n}) \mid u=0 \mbox{ on \(\Sigma_{0}\)}\}\), which is well-defined due to trace theory. We endow \(\prescript{}{0}{H}^{1}(\Omega;\mathbb{R}^{n})\) with the inner product
\begin{align}
    (u,v)_{\prescript{}{0}{H}^{1}(\Omega;\mathbb{R}^{n})}=\frac{1}{2}\int_{\Omega}\mathbb{D}u:\mathbb{D}v
\end{align}
where \(\mathbb{D}\mathrm{k}=\nabla \mathrm{k}+(\nabla \mathrm{k})^{t}\) is the symmetrized gradient of \(\mathrm{k} \in H^{1}(\Omega;\mathbb{R}^{n})\). This is indeed an inner product on \(\prescript{}{0}{H}^{1}(\Omega;\mathbb{R}^{n})\) due to Korn's inequality and generates the same topology as the standard \(L^{2}\) inner product. We denote by \(\prescript{}{0}{H}^{1}_{\sigma}(\Omega;\mathbb{R}^{n})=\{u \in \prescript{}{0}{H}^{1}(\Omega; \mathbb{R}^{n}) \mid \nabla \cdot u =0\}\) the closed subspace of solenoidal vector fields. For \(s \geq 0\), we equip the Hilbert space
\begin{align}
    \mathcal{X}^{s}=\{(u,\psi,p,\eta) \in \prescript{}{0}{H}^{s+2}(\Omega;\mathbb{R}^{n}) \times \prescript{}{0}{H}^{s+2}(\Omega;\mathbb{R})\times H^{s+1}(\Omega;\mathbb{R}) \times X^{s+5/2}(\mathbb{R}^{n-1};\mathbb{R})\}
\end{align}
with the norm
\begin{align}
    \norm{(u,\psi,p,\eta)}_{\mathcal{X}^{s}}= \left(\norm{u}_{\prescript{}{0}{H}^{s+2}(\Omega;\mathbb{R}^{n})}^2 + \norm{\psi}_{\prescript{}{0}{H}^{s+2}(\Omega;\mathbb{R})}^2 + \norm{p}_{H^{s+1}(\Omega;\mathbb{R})}^2 + \norm{\eta}_{X^{s+5/2}(\mathbb{R}^{n-1};\mathbb{R})}^2 \right)^{1/2}.
\end{align}

For \(r<0\),
\begin{align} \label{hdspace}
    \dot{H}^{r}(\mathbb{R}^{d};\mathbb{R}) = \{f \in \mathscr{S}(\mathbb{R}^{d};\mathbb{R}) \mid \mbox{\(f=\overline{f}\), \(\hat{f} \in L^{1}_{loc}(\mathbb{R}^{d};\mathbb{R})\), and \([f]_{\dot{H}^{r}}<\infty\)}\}
\end{align}
where
\begin{align}
    [f]^{2}_{\dot{H}^{r}}=\int_{\mathbb{R}^{d}}\abs{\xi}^{2r}\abs{\hat{f}(\xi)}^{2}d\xi
\end{align}
is the homogeneous \(L^{2}\)-based Sobolev space of real-valued functions of order \(r\).
For \(s \geq 0\), we equip the Hilbert space
\begin{align} \label{Ys}
    \mathcal{Y}^{s}=\biggl\{&(f,g,l,k,h,m) \in H^{s}(\Omega;\mathbb{R}^{n}) \times H^{s+1}(\Omega;\mathbb{R}) \times H^{s}(\Omega;\mathbb{R}) \times H^{s+1/2}(\Sigma_{b};\mathbb{R}^{n}) \nonumber \\
    &\times H^{s+3/2}(\Sigma_{b};\mathbb{R}) \times H^{s+1/2}(\Sigma_{b};\mathbb{R}) \mid h-\int_{0}^{b}g(\cdot,x_{n})dx_{n}\in \dot{H}^{-1}(\mathbb{R}^{n-1};\mathbb{R})\biggr\}
\end{align}
with the norm
\begin{align}
    &\norm{(f,g,l,k,h,m)}_{\mathcal{Y}^{s}} \nonumber \\
    =&\biggl(\norm{f}^{2}_{H^{s}(\Omega;\mathbb{R}^{n})}+\norm{g}^{2}_{H^{s+1}(\Omega;\mathbb{R})}+\norm{l}^{2}_{H^{s}(\Omega;\mathbb{R})}+\norm{k}^{2}_{H^{s+1/2}(\Sigma_{b};\mathbb{R}^{n})}+\norm{h}^{2}_{H^{s+3/2}(\Sigma_{b};\mathbb{R})} \nonumber \\
    &+\norm{m}^{2}_{H^{s+1/2}(\Sigma_{b};\mathbb{R})}+\left[h-\int_{0}^{b}g(\cdot,x_{n})dx_{n}\right]_{\dot{H}^{-1}(\mathbb{R}^{n-1};\mathbb{R})}^{2}\biggr)^{1/2}.
\end{align}

Given a finite-dimensional inner product space \(W\), we identify \(H^{s}(\Sigma_{b};W)\) with \(H^{s}(\mathbb{R}^{n-1};W)\) for each \(s \in \mathbb{R}\). This is justified by the fact that \(\Sigma_{b}\) is canonically diffeomorphic to \(\mathbb{R}^{n-1}\) via the map \((x',b) \in \Sigma_{b} \mapsto x' \in \mathbb{R}^{n-1}\). We write \(\mathcal{F}(\mathrm{k})\) and \(\mathcal{F}^{-1}(\mathrm{k})\) for the Fourier transform and its inverse acting on \(\mathrm{k} \in H^{s}(\Sigma_{b};W)\), respectively.

\section{Solvability of the \(\widetilde{\gamma}\)-Stokes System \eqref{mothereqn}} \label{gstokesadj}
\subsection{The bilinear form} \label{blform}

To establish the weak solvability of \eqref{mothereqn}, we first construct the associated bilinear form. For a sufficiently regular solution triple \((w, \theta, r)\) of \eqref{mothereqn} and sufficiently regular test functions \(v\) and \(\psi\),
\begin{align}
    \int_{\Omega}v\cdot f-\int_{\Sigma_{b}} v \cdot k &= \int_{\Omega}-\widetilde{\gamma} \partial_{1}v \cdot w +\frac{\mu}{2}\mathbb{D}v:\mathbb{D}w-\alpha_{1}\int_{\Sigma_{b}}\nabla' \theta \cdot v' \label{wf1} \\
    \int_{\Omega}\psi l +\int_{\Sigma_{b}}\psi m &=\int_{\Omega}-\widetilde{\gamma}\partial_{1}\psi \theta+\kappa \nabla\psi \cdot \nabla \theta +\alpha_{2}\int_{\Sigma_{b}} (\nabla' \cdot w') \psi. \label{wf2}
\end{align}
Unfortunately, \eqref{wf1} and \eqref{wf2} are ill-defined for \(w,v \in \prescript{}{0}{H}^{1}_{\sigma}(\Omega;\mathbb{R}^{n})\), \(\theta,\psi \in \prescript{}{0}{H}^{1}(\Omega;\mathbb{R})\), and \(r \in H^{1}(\Omega;\mathbb{R})\), because neither \(\theta\) nor \(w\) is regular enough to have a trace on \(\Sigma_{b}\) to ensure that the right-most integrals in \eqref{wf1} and \eqref{wf2} are well-defined. To address this issue, we observe that for \(\mathrm{f}\in\mathscr{S}(\mathbb{R}^{n-1};\mathbb{R})\) and \(\mathrm{g}\in\mathscr{S}(\mathbb{R}^{n-1};\mathbb{R})\),
\begin{align}
    \int_{\Sigma_{b}}\partial_{i}\mathrm{f}\mathrm{g}=\int_{\Sigma_{b}}2\pi i \xi_{i}\hat{\mathrm{f}}(\xi)\overline{\hat{\mathrm{g}}(\xi)} \mbox{\quad for any \(i \in \{1, \dots, n-1\}\)}.
\end{align}
The bilinear form \((\mathrm{f},\mathrm{g}) \in \mathscr{S}(\mathbb{R}^{n-1};\mathbb{R}) \times \mathscr{S}(\mathbb{R}^{n-1};\mathbb{R}) \mapsto \int_{\Sigma_{b}}2\pi i \xi_{i}\hat{\mathrm{f}}(\xi)\overline{\hat{\mathrm{g}}(\xi)} \in \mathbb{R}\) then extends continuously to \(H^{1/2}(\Sigma_{b};\mathbb{R}) \times H^{1/2}(\Sigma_{b};\mathbb{R})\). Now, let \(Q_{1}: H^{1}(\Omega;\mathbb{R}^{n}) \times H^{1}(\Omega;\mathbb{R}) \to \mathbb{R}\) and \(Q_{2}: H^{1}(\Omega;\mathbb{R}^{n}) \times H^{1}(\Omega;\mathbb{R}) \to \mathbb{R}\) be bilinear forms given by 
\begin{align}
    &Q_{1}(v,\theta)=\int_{\Sigma_{b}}2\pi i\xi_{j}\mathcal{F}(\theta \mid_{\Sigma_{b}})(\xi)\overline{\mathcal{F}(v_{j}\mid_{\Sigma_{b}})(\xi)} d \xi \label{Q1} \\
    &Q_{2}(w,\psi)=\int_{\Sigma_{b}}2\pi i \xi_{j}\mathcal{F}(w_{j}\mid_{\Sigma_{b}})(\xi)\overline{\mathcal{F}(\psi\mid_{\Sigma_{b}})(\xi)}  d\xi \label{Q2}
\end{align}
in Einstein notation. The norms of \(Q_{1}\) and \(Q_{2}\) are denoted by \(\norm{Q_{1}}=\norm{Q_{1}}_{B(H^{1}(\Omega;\mathbb{R}^{n}),H^{1}(\Omega;\mathbb{R}))}\) and \(\norm{Q_{2}}=\norm{Q_{2}}_{B(H^{1}(\Omega;\mathbb{R}^{n}),H^{1}(\Omega;\mathbb{R}))}\), respectively. In terms of \(Q_{1}\) and \(Q_{2}\), we can define a bilinear form \(B_{\widetilde{\gamma},\mu,\kappa,\alpha_{1},\alpha_{2}}: (\prescript{}{0}{H}^{1}(\Omega;\mathbb{R}^{n}) \times \prescript{}{0}{H}^{1}(\Omega;\mathbb{R}))^{2} \to \mathbb{R}\) associated with \eqref{mothereqn} by
\begin{align}
    B_{\widetilde{\gamma},\mu,\kappa,\alpha_{1},\alpha_{2}}((w,\theta),(v,\psi))=&\int_{\Omega}-\widetilde{\gamma}\partial_{1}v \cdot w+\frac{\mu}{2}\mathbb{D}v:\mathbb{D}w-\widetilde{\gamma}\partial_{1}\psi \theta + \kappa \nabla\psi \cdot \nabla\theta \nonumber \\
    &-\alpha_{1}Q_{1}(v ,\theta)+\alpha_{2}Q_{2}(w,\psi). \label{bform}
\end{align}

\begin{proposition} \label{paramsp}
    If \(\widetilde{\gamma} \in \mathbb{R}\) and \((\mu,\kappa,\alpha_{1},\alpha_{2}) \in O_{n,b}\), where
    \begin{align}
        O_{n,b} = \left\{(\mu',\kappa',\alpha_{1}',\alpha_{2}') \in \mathbb{R}^{+} \times \mathbb{R}^{+} \times \mathbb{R} \times \mathbb{R} \mid \left(\frac{\abs{\alpha_{1}'}\norm{Q_{1}}+\abs{\alpha_{2}'}\norm{Q_{2}}}{2}\right)^{2} < 2\mu'\kappa'\right\},
    \end{align}
    then the map \(B_{\widetilde{\gamma},\mu,\kappa,\alpha_{1},\alpha_{2}}\) defined in \eqref{bform} is a coercive bounded bilinear form.
\end{proposition}
\begin{proof}
    Let \(\widetilde{\gamma} \in \mathbb{R}\) and \((\mu,\kappa,\alpha_{1},\alpha_{2}) \in O_{n,b}\). It is clear that \(B_{\widetilde{\gamma},\mu,\kappa,\alpha_{1},\alpha_{2}}\) is bilinear over \(\mathbb{R}\). Moreover, \(B_{\widetilde{\gamma},\mu,\kappa,\alpha_{1},\alpha_{2}}\) is bounded since
    \begin{align}
        \abs{B_{\widetilde{\gamma},\mu,\kappa,\alpha_{1},\alpha_{2}}((w,\theta),(v,\psi))} \leq& \abs{\widetilde{\gamma}}\abs{\int_{\Omega}\partial_{1}v \cdot w}+\frac{\mu}{2}\abs{\int_{\Omega}\mathbb{D}v:\mathbb{D}w}+\abs{\widetilde{\gamma}}\abs{\int_{\Omega}\partial_{1}\psi\theta}+\kappa\abs{\int_{\Omega}\nabla\psi \cdot \nabla\theta} \nonumber\\
        &+\abs{\alpha_{1}}\abs{Q_{1}(v,\theta)}+\abs{\alpha_{2}}\abs{Q_{2}(w,\psi)} \nonumber\\
        \lesssim&_{\widetilde{\gamma},\mu,\kappa,\alpha_{1},\alpha_{2},n,b}(\norm{w}_{H^{1}(\Omega;\mathbb{R}^{n})}+\norm{\theta}_{H^{1}(\Omega;\mathbb{R})})(\norm{v}_{H^{1}(\Omega;\mathbb{R}^{n})}+\norm{\psi}_{H^{1}(\Omega;\mathbb{R})}).
\end{align}
It remains to show the coercivity. For any \(w \in \prescript{}{0}{H}^{1}(\Omega;\mathbb{R}^{n})\) and \(\theta \in \prescript{}{0}{H}^{1}(\Omega;\mathbb{R})\),
\begin{align}
    B_{\widetilde{\gamma},\mu,\kappa,\alpha_{1},\alpha_{2}}((w,\theta),(w,\theta))=\mu\norm{w}_{\prescript{}{0}{H}^{1}(\Omega;\mathbb{R}^{n})}^{2}+2\kappa\norm{\theta}_{\prescript{}{0}{H}^{1}(\Omega;\mathbb{R})}^{2}+Q(w,\theta),
\end{align}
where
\begin{align} \label{q}
    Q(w,\theta)=-\alpha_{1}Q_{1}(w,\theta)+\alpha_{2}Q_{2}(w,\theta).
\end{align}
We have
\begin{align}
    &B_{\widetilde{\gamma},\mu,\kappa,\alpha_{1},\alpha_{2}}((w,\theta),(w,\theta)) \nonumber \\
    \geq &\mu\norm{w}_{\prescript{}{0}{H}^{1}(\Omega;\mathbb{R}^{n})}^{2}+2\kappa\norm{\theta}_{\prescript{}{0}{H}^{1}(\Omega;\mathbb{R})}^{2}-(\abs{\alpha_{1}}\norm{Q_{1}}+\abs{\alpha_{2}}\norm{Q_{2}})\norm{w}_{\prescript{}{0}{H}^{1}(\Omega;\mathbb{R}^{n})}\norm{\theta}_{\prescript{}{0}{H}^{1}(\Omega;\mathbb{R})}.
\end{align}
Since \((\mu,\kappa,\alpha_{1},\alpha_{2}) \in O_{n,b}\), there exists \(C_{0}>0\) such that
\begin{align}
    \biggl(\frac{\abs{\alpha_{1}}\norm{Q_{1}}+\abs{\alpha_{2}}\norm{Q_{2}}}{2}\biggr)^{2} \leq (\mu-C_{0})(2\kappa-C_{0}).
\end{align}
This inequality is equivalent to the condition that the real symmetric matrix
\begin{align}
    M=
    \begin{pmatrix}
        \mu-C_{0} & -\frac{\abs{\alpha_{1}}\norm{Q_{1}}+\abs{\alpha_{2}}\norm{Q_{2}}}{2} \\
        -\frac{\abs{\alpha_{1}}\norm{Q_{1}}+\abs{\alpha_{2}}\norm{Q_{2}}}{2} & 2\kappa - C_{0}
    \end{pmatrix}
\end{align}
be positive semi-definite, or equivalently,
\begin{align}
    &\mu\norm{w}_{\prescript{}{0}{H}^{1}(\Omega;\mathbb{R}^{n})}^{2}+2\kappa\norm{\theta}_{\prescript{}{0}{H}^{1}(\Omega;\mathbb{R})}^{2}-(\abs{\alpha_{1}}\norm{Q_{1}}+\abs{\alpha_{2}}\norm{Q_{2}})\norm{w}_{\prescript{}{0}{H}^{1}(\Omega;\mathbb{R}^{n})}\norm{\theta}_{\prescript{}{0}{H}^{1}(\Omega;\mathbb{R})} \nonumber\\
    \geq &C_{0}(\norm{w}_{\prescript{}{0}{H}^{1}(\Omega;\mathbb{R}^{n})}^{2}+\norm{\theta}_{\prescript{}{0}{H}^{1}(\Omega;\mathbb{R})}^{2}).
\end{align}
Hence, \(B_{\widetilde{\gamma},\mu,\kappa,\alpha_{1},\alpha_{2}}\) is indeed coercive.
\end{proof}

\subsection{Introduction of the pressure term}
We use an argument similar to that in Section \(2\) of \cite{leoni2023traveling} to introduce a pressure term to the bilinear form defined in Section \ref{blform}. Proposition \ref{wks} summarizes the result whose proof is suppressed to avoid redundancy. 
\begin{proposition} \label{wks}
    Let \(\widetilde{\gamma} \in \mathbb{R}\) and \((\mu,\kappa,\alpha_{1},\alpha_{2}) \in O_{n,b}\) as in Proposition \ref{paramsp}. If \(F \in (\prescript{}{0}{H}^{1}(\Omega;\mathbb{R}^{n}) \times \prescript{}{0}{H}^{1}(\Omega;\mathbb{R}))^{*}\) and \(g \in L^{2}(\Omega;\mathbb{R})\), then there exist unique \((u,\psi) \in \prescript{}{0}{H}^{1}(\Omega;\mathbb{R}^{n}) \times \prescript{}{0}{H}^{1}(\Omega;\mathbb{R})\) and \(p \in L^{2}(\Omega;\mathbb{R})\) such that for all \((v,\phi) \in \prescript{}{0}{H}^{1}(\Omega;\mathbb{R}^{n}) \times \prescript{}{0}{H}^{1}(\Omega;\mathbb{R})\),
    \begin{align}
    B_{\widetilde{\gamma},\mu,\kappa,\alpha_{1},\alpha_{2}}((u,\psi),(v,\phi))-\langle F, (v,\phi)\rangle = \int_{\Omega}p(\nabla \cdot v)
    \end{align}
    and \(\nabla \cdot u =g\). Moreover, \(\norm{u}_{\prescript{}{0}{H}^{1}(\Omega;\mathbb{R}^{n})}+\norm{\psi}_{\prescript{}{0}{H}^{1}(\Omega;\mathbb{R})}+\norm{p}_{L^{2}(\Omega;\mathbb{R})} \lesssim_{n,b} \norm{F}_{(\prescript{}{0}{H}^{1}(\Omega;\mathbb{R}^{n}) \times \prescript{}{0}{H}^{1}(\Omega;\mathbb{R}))^{*}}+\norm{g}_{L^{2}(\Omega;\mathbb{R})}\).
\end{proposition}

\subsection{Promotion of regularity of the weak solution}
We use standard regularity results for elliptic systems to promote the regularity of the weak solution in Proposition \ref{wks}.
\begin{proposition} \label{classicalsoln}
    Let \(\widetilde{\gamma} \in \mathbb{R}\) and \((\mu,\kappa,\alpha_{1},\alpha_{2}) \in O_{n,b}\) as in Proposition \ref{paramsp}. Suppose that \(s \geq 0\), \(f \in H^{s}(\Omega;\mathbb{R}^{n})\), \(g \in H^{s+1}(\Omega;\mathbb{R})\), \(k \in H^{s+1/2}(\Sigma_{b};\mathbb{R}^{n})\), \(l \in H^{s}(\Omega;\mathbb{R})\), and \(m \in H^{s+1/2}(\Sigma_{b};\mathbb{R})\). If \((u,\psi) \in \prescript{}{0}{H}^{1}(\Omega;\mathbb{R}^{n}) \times \prescript{}{0}{H}^{1}(\Omega;\mathbb{R})\) and \(p \in L^{2}(\Omega;\mathbb{R})\) satisfy \(\nabla \cdot u = g\) in \(\Omega\) and for all \((v,\phi) \in \prescript{}{0}{H}^{1}(\Omega;\mathbb{R}^{n}) \times \prescript{}{0}{H}^{1}(\Omega;\mathbb{R})\)
    \begin{align}
        B_{\widetilde{\gamma},\mu,\kappa,\alpha_{1},\alpha_{2}}((u,\psi),(v,\phi))-\int_{\Omega}p(\nabla \cdot v) = \langle F, (v,\phi) \rangle,
    \end{align}
    where
    \begin{align}
        \langle F, (v,\phi) \rangle = \int_{\Omega}v \cdot f + \phi l +\int_{\Sigma_{b}}-v \cdot k+\phi m,
    \end{align}
    then \((u, \psi) \in \prescript{}{0}{H}^{s+2}(\Omega; \mathbb{R}^{n}) \times \prescript{}{0}{H}^{s+2}(\Omega;\mathbb{R})\) and \(p \in H^{s+1}(\Omega;\mathbb{R})\). Moreover,
    \begin{align}
        &\norm{u}_{\prescript{}{0}{H}^{s+2}(\Omega;\mathbb{R}^{n})}+\norm{\psi}_{\prescript{}{0}{H}^{s+2}(\Omega;\mathbb{R})}+\norm{p}_{H^{s+1}(\Omega;\mathbb{R})} \nonumber\\
        \lesssim& \norm{g}_{H^{s+1}(\Omega;\mathbb{R})}+\norm{f}_{H^{s}(\Omega;\mathbb{R}^{n})}+\norm{k}_{H^{s+1/2}(\Sigma_{b};\mathbb{R}^{n})}+\norm{l}_{H^{s}(\Omega;\mathbb{R})}+\norm{m}_{H^{s+1/2}(\Sigma_{b};\mathbb{R})}.
    \end{align}
\end{proposition}
\begin{proof}
    This is a consequence of the regularity results for elliptic systems proved in \cite{agmon1964estimates}. A more elementary proof can be constructed using an argument similar to that of Theorem \(2.5\) in \cite{leoni2021travelingwavesolutionsfree}.
\end{proof}
Proposition \ref{msys} is now an immediate consequence of Proposition \ref{classicalsoln}.

\subsection{Characterization of solvability of \eqref{odp}}
Now that we have established the solvability of \eqref{mothereqn}, we proceed to prove Proposition \ref{Zsiso}, which characterizes the solvability of \eqref{odp}.

\begin{proof}[Proof of Proposition \ref{Zsiso}]
    Suppose that \(\gamma\in \mathbb{R}\) and \((\mu,\kappa,\sigma'(0)) \in \mathbb{R}^{+} \times \mathbb{R}^{+} \times \mathbb{R}\) satisfy
    \begin{align}
        \max\left\{\frac{\norm{Q_{1}}^{2}}{4},\frac{\norm{Q_{2}}^{2}}{4}\right\}\abs{\sigma'(0)}^{2} < 2\mu\kappa
    \end{align}
    where \(Q_{1}\) and \(Q_{2}\) are the bilinear forms defined in \eqref{Q1} and \eqref{Q2}. Let \(s \geq 0\) and \((u,\psi,p) \in \prescript{}{0}{H}^{s+2}(\Omega;\mathbb{R}^{n}) \times \prescript{}{0}{H}^{s+2}(\Omega;\mathbb{R})\times H^{s+1}(\Omega;\mathbb{R})\). By Theorem \(3.1\) of \cite{leoni2023traveling},
    \begin{align}
        \biggl[T_{5}(u,\psi,p)-\int_{0}^{b}T_{2}(u,\psi,p)(\cdot,x_{n})d x_{n}\biggr]_{\dot{H}^{-1}(\mathbb{R}^{n-1};\mathbb{R})} =& \biggl[u_{n} \mid_{\Sigma_{b}}-\int_{0}^{b}(\nabla \cdot u)(\cdot,x_{n})d x_{n}\biggr]_{\dot{H}^{-1}(\mathbb{R}^{n-1};\mathbb{R})} \nonumber\\
        \leq&2\pi\sqrt{b}\norm{u}_{L^{2}(\Omega;\mathbb{R}^{n})}.
    \end{align}
    For \(\chi \in H^{s+1/2}(\Sigma_{b};\mathbb{R})\), let \((v,q,\phi) \in \prescript{}{0}{H}^{s+2}(\Omega;\mathbb{R}^{n}) \times H^{s+1}(\Omega;\mathbb{R}) \times \prescript{}{0}{H}^{s+2}(\Omega;\mathbb{R})\) be the unique solution to the adjoint system (\ref{ap}) with \((f,g,l,k',m)=(0,0,0,0,0)\) and the normal stress condition \((qI-\mu\mathbb{D}v)e_{n} \cdot e_{n}=\chi\) on \(\Sigma_{b}\). Such a solution is guaranteed by Proposition \ref{msys} with \(\widetilde{\gamma}=\gamma\) and \((\mu,\kappa,\alpha_{1},\alpha_{2})=(\mu,\kappa,0,\sigma'(0)) \in O_{n,b}\). By Proposition \ref{compcond}, \(T(u,\psi,p) \in \mathcal{Z}^{s}\). Hence, \(T: \prescript{}{0}{H}^{s+2}(\Omega;\mathbb{R}^{n}) \times \prescript{}{0}{H}^{s+2}(\Omega;\mathbb{R})\times H^{s+1}(\Omega;\mathbb{R}) \to \mathcal{Z}^{s}\) is well-defined. To show injectivity of \(T\), let \(T(u,\psi,p) =0\). By Proposition \ref{msys} with \(\widetilde{\gamma}=-\gamma\) and \((\mu,\kappa,\alpha_{1},\alpha_{2})=(\mu,\kappa,\sigma'(0),0)\), we obtain
    \begin{align}
        &\Phi_{-\gamma,\mu,\kappa,\sigma'(0),0}(u,\psi,p) \nonumber\\
        =&(-\gamma\partial_{1}u+\nabla\cdot (pI-\mu\mathbb{D}u), \nabla \cdot u, -\gamma\partial_{1}\psi-\kappa\Delta\psi,(pI-\mu\mathbb{D}u)\mid_{\Sigma_{b}}e_{n}+(\sigma'(0)\nabla'\psi\mid_{\Sigma_{b}},0),\kappa\partial_{n}\psi \mid_{\Sigma_{b}}) \nonumber\\
        =&(0,0,0,0,0).
    \end{align}
    Then \((u,\psi,p)=(0,0,0)\), which shows that \(T\) is injective. To show surjectivity of \(T\), let \((f,g,l,k,h,m) \in \mathcal{Z}^{s}\). For any \(\chi \in H^{s+1/2}(\Sigma_{b};\mathbb{R})\), let \((v,q,\phi) \in \prescript{}{0}{H}^{s+2}(\Omega;\mathbb{R}^{n}) \times H^{s+1}(\Omega;\mathbb{R}) \times \prescript{}{0}{H}^{s+2}(\Omega;\mathbb{R})\) be the unique solution to the adjoint system (\ref{ap}) with \((f,g,l,k',m)=(0,0,0,0,0)\) and the normal stress condition \((qI-\mu\mathbb{D}v)e_{n} \cdot e_{n}=\chi\) on \(\Sigma_{b}\). Such a solution is guaranteed by Proposition \ref{msys} with \(\widetilde{\gamma}=\gamma\) and \((\mu,\kappa,\alpha_{1},\alpha_{2})=(\mu,\kappa,0,\sigma'(0)) \in O_{n,b}\). Then
    \begin{align}
        \int_{\Omega}f \cdot v -gq+l\phi=\int_{\Sigma_{b}}k \cdot v-m\phi-h\chi.
    \end{align}
    By Proposition \ref{msys} with \(\widetilde{\gamma}=-\gamma\) and \((\mu,\kappa,\alpha_{1},\alpha_{2})=(\mu,\kappa,\sigma'(0),0)\), there exist \((u,\psi,p) \in \prescript{}{0}{H}^{s+2}(\Omega;\mathbb{R}^{n}) \times \prescript{}{0}{H}^{s+2}(\Omega;\mathbb{R}) \times H^{s+1}(\Omega;\mathbb{R})\) such that
    \begin{align}
        &\Phi_{-\gamma,\mu,\kappa,\sigma'(0),0}(u,\psi,p) \nonumber\\
        =&(-\gamma\partial_{1}u+\nabla\cdot (pI-\mu\mathbb{D}u), \nabla \cdot u, -\gamma\partial_{1}\psi-\kappa\Delta\psi,(pI-\mu\mathbb{D}u)\mid_{\Sigma_{b}}e_{n}+(\sigma'(0)\nabla'\psi\mid_{\Sigma_{b}},0),\kappa\partial_{n}\psi \mid_{\Sigma_{b}}) \nonumber\\
        =&(f,g,l,k,m).
    \end{align}
    Then by Proposition \ref{compcond},
    \begin{equation}
        \int_{\Omega}f \cdot v -gq+l\phi=\int_{\Sigma_{b}}k \cdot v-m\phi-u_{n}\mid_{\Sigma_{b}}\chi \text{ and }\int_{\Sigma_{b}} h \chi=\int_{\Sigma_{b}}u_{n} \mid_{\Sigma_{b}}\chi,
    \end{equation}
    from which we deduce that \(h=u_{n} \mid_{\Sigma_{b}}\). Since \(T(u,\psi,p)=(f,g,l,k,h,m)\), \(T\) is indeed surjective. Lastly, it is straightforward to see that \(T\) is linear and bounded.
\end{proof}

\section{Fourier Symbols: The Hilbert Space Setting} \label{fst}
In this section, we adapt to our setting the theory of Fourier symbols for translation commuting, bounded linear maps as presented in Section \(3.1\) of \cite{MR4787851}. We record some essential definitions first. Let \(V\) and \(W\) be separable Hilbert spaces over \(\mathbb{C}\). A bounded linear map \(T:L^{2}(\mathbb{R}^{d};V) \to L^{2}(\mathbb{R}^{d};W)\) is translation commuting if \((Tf)(x + h)=T(f(\cdot + h))(x)\) for every \(f \in L^{2}(\mathbb{R}^{d};V)\) and every \(x, h \in \mathbb{R}^{d}\). Let \(X\) be a complete \(\sigma\)-finite measure space and \(Y\) a Banach space. A function \(g:X \to Y\) is Bochner measurable if it is the almost everywhere limit of a sequence of finite-valued measurable simple functions. Let \(\mathcal{L}(V,W)\) be the set of all linear maps from \(V\) to \(W\). A function \(f: X \to \mathcal{L}(V,W)\) is strongly measurable if for all \(v \in V\), the map \(f_{v}: X \to W\) is Bochner measurable. We let
\begin{align}
    &L_{*}^{\infty}(X;\mathcal{L}(V,W)) =\{[f]\mid \mbox{ \(f: X \to \mathcal{L}(V,W)\) is strongly measurable and \(\norm{[f]}_{L_{*}^{\infty}\mathcal{L}(V;W)}<\infty\)}\},
\end{align}
where \([f]\) is the equivalence class of \(f\) formed via almost everywhere equality and
\begin{align}
    \norm{[f]}_{L_{*}^{\infty}\mathcal{L}(V;W)}=&\mbox{esssup}_{x \in X}\norm{f(x)}_{\mathcal{L}(V;W)} =\inf\{\mbox{\(C \in \mathbb{R}^{+} \mid \norm{f}_{\mathcal{L}(V;W)} \leq C\) a.e.}\}.
\end{align}
This space extends the notion of essential supremum to the space of strongly measurable \(\mathcal{L}(V;W)\)-valued maps. From now on, we omit the equivalence class notation. For a bounded linear map \(T: L^{2}(\mathbb{R}^{d};V) \to L^{2}(\mathbb{R}^{d};W)\), if there exists \(m \in L^{\infty}_{*}(\mathbb{R}^{d};\mathcal{L}(V;W))\) such that \(Tf=\mathcal{F}^{-1}(m\mathcal{F}(f))\) for every \(f\in L^{2}(\mathbb{R}^{d};V)\), then we write \(T=m(D)\). If \(\sigma \in L^{\infty}(\mathbb{R}^{d};\mathbb{C})\), then we define \(\sigma(D): \mathscr{S}(\mathbb{R}^{d};X) \to \mathscr{S}'(\mathbb{R}^{d};X)\) via \(\sigma(D)f=\mathcal{F}^{-1}(\sigma\mathcal{F}f)\). To see that \(\sigma(D)\) is well-defined, observe that since \(\sigma \in L^{\infty}(\mathbb{R}^{d};\mathbb{C})\) and \(\mathcal{F}f \in \mathscr{S}(\mathbb{R}^{d};X)\), we have \(\sigma\mathcal{F}f \in L^{2}(\mathbb{R}^{d};X)\). Then \(\sigma(D)f=\mathcal{F}^{-1}(\sigma\mathcal{F}f) \in L^{2}(\mathbb{R}^{d};X)\). Since \(g \in \mathscr{S}(\mathbb{R}^{d};X) \mapsto \int_{\mathbb{R}^{d}}(\sigma(D)f)(x)g(x)dx\) is well-defined, \(\sigma(D)f \in \mathscr{S}'(\mathbb{R}^{d};X)\).

In the rest of the section, we prove the following result, which is an analog of Corollary \(3.6\) of \cite{MR4787851}.
\begin{proposition} \label{fs}
    Let \(V\) and \(W\) be separable Hilbert spaces over \(\mathbb{C}\) and \(t,s \in \mathbb{R}\). If \(T: H^{t}(\mathbb{R}^{d};V) \to H^{s}(\mathbb{R}^{d};W)\) is a translation commuting and continuous linear map, then there exists a unique (up to modification on sets of measure zero) locally essentially bounded and strongly measurable function \(m: \mathbb{R}^{d} \to \mathcal{L}(V;W)\) such that \(T=\mathcal{F}^{-1}m\mathcal{F}\). For \(\xi \in \mathbb{R}^{d}\) a.e.,
    \begin{align}
        \langle \xi \rangle^{s}\norm{m(\xi)x}_{W} \lesssim \norm{T}\langle \xi\rangle^{t}\norm{x}_{V} \mbox{ for all \(x \in V\),}
    \end{align}
    where the implicit constant depends only on \(d\), \(s\), and \(t\).
\end{proposition}

\begin{proof}
    For \(0 \leq l \in \mathbb{N}\), let $A_0 = B(0,1)$ and $A_{l} = B(0,2^{l}) \setminus B(0,2^{l-1})$ for $l \geq 1$.     Since \(\chi_{A_{l}} \in L^{\infty}(\mathbb{R}^{d};\mathbb{C})\), the map \(\chi_{A_{l}}(D):\mathscr{S}(\mathbb{R}^{d};X) \to \mathscr{S}'(\mathbb{R}^{d};X)\) is well-defined. Letting \(\norm{x}_{V^{(l)}}=\langle 2^{l} \rangle^{t}\norm{x}_{V}\) and \(\norm{y}_{W^{(l)}}=\langle 2^{l} \rangle^{s}\norm{y}_{W}\) for \(x \in V\) and \(y \in W\), we define \(W^{(l)}\) to be the \(W\) equipped with the norm \(\norm{\cdot}_{W^{(l)}}\) and \(V^{(l)}\) to be the \(V\) equipped with the norm \(\norm{\cdot}_{V^{(l)}}\). For \(f \in H^{s}(\mathbb{R}^{d};W)\), let \((\mathcal{J}^{s}f)(x)=\mathcal{F}^{-1}(\langle \cdot \rangle^{s}\mathcal{F}(f)(\cdot))(x) \in L^{2}(\mathbb{R}^{d};W)\), which is well-defined since
    \begin{align}
        \norm{\mathcal{J}^{s}f}_{L^{2}(\mathbb{R}^{d};W)}^{2}=\int_{\mathbb{R}^{d}}\norm{\langle \xi \rangle^{s}\mathcal{F}(f)(\xi)}_{W}^{2}d\xi=\int_{\mathbb{R}^{d}}\langle \xi \rangle^{2s}\norm{\mathcal{F}(f)(\xi)}_{W}^{2}d\xi=\norm{f}_{H^{s}}^{2}<\infty.
    \end{align}
    Letting \(T_{l}=T \circ \chi_{A_{l}}(D)\), we define \(\widetilde{T_{l}}:L^{2}(\mathbb{R}^{d};V) \to L^{2}(\mathbb{R}^{d};W)\) by \(\widetilde{T_{l}}=\mathcal{J}^{s}\circ T_{l}\). The map \(\widetilde{T_{l}}\) is bounded because if \(f \in L^{2}(\mathbb{R}^{d};V)\) then
    \begin{align}
        \norm{\widetilde{T_{l}}f}_{L^{2}(\mathbb{R}^{d};W)}^{2}=&\int_{\mathbb{R}^{d}}\norm{(\widetilde{T_{l}}f)(x)}_{W}^{2}dx =\int_{\mathbb{R}^{d}}\norm{\langle \xi \rangle^{s}\mathcal{F}(T_{l}f)(\xi)}_{W}^{2}d\xi =\int_{\mathbb{R}^{d}}\langle\xi\rangle^{2s}\norm{\mathcal{F}(T_{l}f)(\xi)}_{W}^{2}d\xi \nonumber \\
        =& \norm{T\chi_{A_{l}}(D)f}_{H^{s}(\mathbb{R}^{d};W)}^{2} \leq \norm{T}_{\mathcal{L}(H^{t}(\mathbb{R}^{d};V);H^{s}(\mathbb{R}^{d};W))}^{2}\norm{\chi_{A_{l}}(D)f}_{H^{t}(\mathbb{R}^{d};V)}^{2},
    \end{align}
    where
    \begin{align}
        \norm{\chi_{A_{l}}(D)f}_{H^{t}(\mathbb{R}^{d};V)}^{2}=&\int_{\mathbb{R}^{d}}\langle \xi\rangle^{2t}\norm{\mathcal{F}(\chi_{A_{l}}(D)f)(\xi)}_{V}^{2}d\xi =\int_{\mathbb{R}^{d}}\langle\xi \rangle^{2t}\norm{\chi_{A_{l}}(\xi)\mathcal{F}(f)(\xi)}_{V}^{2}d\xi \nonumber\\
        =&\int_{A_{l}}\langle \xi \rangle^{2t}\norm{\mathcal{F}(f)(\xi)}_{V}^{2}d\xi \leq \langle 2^{l} \rangle^{2t}\norm{f}_{L^{2}(\mathbb{R}^{d};V)}^{2}<\infty.
    \end{align}
    It follows that \(\norm{\widetilde{T_{l}}f}_{L^{2}(\mathbb{R}^{d};W)} \leq \norm{T}_{\mathcal{L}(H^{t}(\mathbb{R}^{d};V);H^{s}(\mathbb{R}^{d};W))}\langle 2^{l} \rangle^{t}\norm{f}_{L^{2}(\mathbb{R}^{d};V)}\). Hence,
    \begin{align}
        \norm{\widetilde{T_{l}}}_{\mathcal{L}(L^{2}(\mathbb{R}^{d};V);L^{2}(\mathbb{R}^{d};W))} \leq \langle 2^{l} \rangle^{t}\norm{T}_{\mathcal{L}(H^{t}(\mathbb{R}^{d};V);H^{s}(\mathbb{R}^{d};W))}.
    \end{align}
    If \(\tau^{\alpha}\) is the operator for translation by \(-\alpha \in \mathbb{R}^{d}\), then
    \begin{align}
        \mathcal{F}(\mathcal{J}^{s}\tau^{\alpha}f)(\xi)=\langle \xi\rangle^{s}\mathcal{F}(\tau^{\alpha}f)(\xi)=\langle\xi \rangle^{s}e^{-2\pi i \alpha \cdot \xi}\mathcal{F}(f)(\xi) =e^{-2\pi i \alpha \cdot \xi}\mathcal{F}(\mathcal{J}^{s}f)(\xi)=\mathcal{F}(\tau^{\alpha}\mathcal{J}^{s}f)(\xi).
    \end{align}
    Thus, \(\mathcal{J}^{s}\tau^{\alpha}f = \tau^{\alpha}\mathcal{J}^{s}f\). Since \(T_{l}\) is translation commuting, we conclude that \(\widetilde{T_{l}}\) is translation commuting. By Theorem \(3.5\) of \cite{MR4787851}, there exists \(\widetilde{m_{l}} \in L^{\infty}_{*}(\mathbb{R}^{d};\mathcal{L}(V,W))\) such that \(\widetilde{T_{l}}=\widetilde{m_{l}}(D)\) and
    \begin{align}
        \norm{\widetilde{m_{l}}}_{L^{\infty}_{*}(\mathbb{R}^{d};\mathcal{L}(V,W))}=\norm{\widetilde{T_{l}}}_{\mathcal{L}(L^{2}(\mathbb{R}^{d};V); L^{2}(\mathbb{R}^{d};W))}.
    \end{align}
    Note that  \(\mathcal{J}^{-s}\circ \mathcal{J}^{s}=id\) implies  \(T_{l}=\mathcal{J}^{-s}\circ \widetilde{T_{l}}=\mathcal{J}^{-s}\circ \widetilde{m_{l}}(D)\). Similarly,  $T_{l} \circ \chi_{A_{l}}(D)   = (T \circ \chi_{A_{l}}(D))\circ \chi_{A_{l}}(D)   =T \circ \chi_{A_{l}}(D)   = T_{l}$. Therefore, \(T_{l}=\mathcal{J}^{-s}\circ \widetilde{m_{l}}(D) \circ \chi_{A_{l}}(D)\). It follows that \(T_{l}\) has a Fourier multiplier \(m_{l}(\xi)=\langle \xi\rangle^{-s}\widetilde{m_{l}}(\xi)\chi_{A_{l}}(\xi)\). Since
    \begin{align}
        \norm{m_{l}(\xi)}_{\mathcal{L}(V;W)}\leq& \langle 2^{l} \rangle^{-s}\chi_{A_{l}}(\xi)\norm{\widetilde{m_{l}}(\xi)}_{\mathcal{L}(V;W)} \leq \langle 2^{l} \rangle^{-s} \norm{\widetilde{m_{l}}}_{L^{\infty}_{*}(\mathbb{R}^{d};\mathcal{L}(V;W))} \nonumber\\
        =&\langle 2^{l}\rangle^{-s} \norm{\widetilde{T_{l}}}_{\mathcal{L}(L^{2}(\mathbb{R}^{d};V);L^{2}(\mathbb{R}^{d};W))} \leq \langle 2^{l} \rangle^{-s}\langle 2^{l}\rangle^{t}\norm{T}_{\mathcal{L}(H^{t}(\mathbb{R}^{d};V);H^{s}(\mathbb{R}^{d};W))},
    \end{align}
    we obtain \(\norm{m_{l}(\xi)x}_{W} \leq  \langle 2^{l} \rangle^{-s} \langle 2^{l}\rangle^{t}\norm{T}_{\mathcal{L}(H^{t}(\mathbb{R}^{d};V);H^{s}(\mathbb{R}^{d};W))}\norm{x}_{V}\) for all \(x \in V\). Let us take \(m=\sum_{l=0}^{\infty}m_{l}\). Then
    \begin{align}
        m(\xi) = \sum_{l=0}^{\infty}m_{l}(\xi) = \sum_{l=0}^{\infty}\langle \xi\rangle^{-s}\widetilde{m_{l}}(\xi)\chi_{A_{l}}(\xi),
    \end{align}
    where the summands have pairwise disjoint supports. For each \(\xi \in \mathbb{R}^{d}\), there exists \(l^{*} \in \mathbb{N}\) such that for all \(x \in V\),
    \begin{align}
        \norm{m(\xi)x}_{W} \leq& \norm{m(\xi)}_{\mathcal{L}(V;W)}\norm{x}_{V} =\norm{\langle \xi \rangle^{-s}\widetilde{m_{l^{*}}}(\xi)\chi_{A_{l^{*}}}(\xi)}_{\mathcal{L}(V;W)}\norm{x}_{V}=\norm{m_{l^{*}}(\xi)}_{\mathcal{L}(V;W)}\norm{x}_{V}.
    \end{align}
    Since
    \begin{align}
        \norm{m_{l^{*}}(\xi)}_{\mathcal{L}(V;W)}\norm{x}_{V} &\leq \langle 2^{l^{*}} \rangle^{-s}\langle 2^{l^{*}} \rangle^{t}\norm{T}_{\mathcal{L}(H^{t}(\mathbb{R}^{d};V); H^{s}(\mathbb{R}^{d};W))}\norm{x}_{V} \nonumber\\
        &\lesssim \langle\xi \rangle^{-s}\langle\xi \rangle^{t}\norm{T}_{\mathcal{L}(H^{t}(\mathbb{R}^{d};V);H^{s}(\mathbb{R}^{d};W))}\norm{x}_{V},
    \end{align}
    we obtain \(\langle \xi \rangle^{s}\norm{m(\xi)x}_{W} \lesssim \norm{T}_{\mathcal{L}(H^{t}(\mathbb{R}^{d};V);H^{s}(\mathbb{R}^{d};W))}\langle \xi\rangle^{t}\norm{x}_{V}\).
\end{proof}

\section{Instantiation of Fourier Symbols} \label{instfs}
In this section, we apply Proposition \ref{fs} to instantiate Fourier symbols associated with the adjoint system \eqref{ap} with \((f,g,l,k',m)=(0,0,0,0,0)\) and a normal stress condition on \(\Sigma_{b}\). As Proposition \ref{fs} is formulated with \(\mathbb{C}\) as the base field, we need to extend all of the maps associated with the augmented adjoint system to the complex setting. Let \(\chi_{r},\chi_{i} \in H^{s+1/2}(\mathbb{R}^{n-1};\mathbb{R})\). They are the real and imaginary parts, respectively, of the normal stress \(\chi=\chi_{r}+i\chi_{i}\). By Proposition \ref{msys}, if \(\widetilde{\gamma} \in \mathbb{R}\) and \((\mu, \kappa,\alpha_{1},\alpha_{2}) \in O_{n,b}\) in Proposition \ref{paramsp}, then the adjoint system augmented with a real normal stress condition admits an isomorphism of Banach spaces
\begin{align}
    \Phi_{\widetilde{\gamma},\mu,\kappa,\alpha_{1},\alpha_{2}}:&\prescript{}{0}{H}^{s+2}(\Omega;\mathbb{R}^{n}) \times \prescript{}{0}{H}^{s+2}(\Omega;\mathbb{R}) \times H^{s+1}(\Omega;\mathbb{R}) \nonumber\\
    &\to H^{s}(\Omega;\mathbb{R}^{n}) \times H^{s+1}(\Omega;\mathbb{R}) \times H^{s}(\Omega;\mathbb{R}) \times H^{s+1/2}(\Sigma_{b};\mathbb{R}^{n}) \times H^{s+1/2}(\Sigma_{b};\mathbb{R})
\end{align}
that maps the solution triple \((w,\theta,r)\) to the data \((f,g,l,k,m)\). For
\begin{equation}
    (v_{r},\phi_{r},q_{r}) = \Phi^{-1}_{\gamma,\mu,\kappa,0,\sigma'(0)}(0,0,0,\chi_{r}e_{n},0) \text{ and }
    (v_{i},\phi_{i},q_{i}) =\Phi^{-1}_{\gamma,\mu,\kappa,0,\sigma'(0)}(0,0,0,\chi_{i}e_{n},0),
\end{equation}
consider the maps
\begin{align}
    &\chi_{r,i} \in H^{s+1/2}(\mathbb{R}^{n-1};\mathbb{R}) \xmapsto{\mathfrak{T}_{v}} v_{r,i} \in H^{s+2}(\mathbb{R}^{n-1}; L^{2}((0,b);\mathbb{R}^{n})) \label{rvmap} \\
    &\chi_{r,i} \in H^{s+1/2}(\mathbb{R}^{n-1};\mathbb{R}) \xmapsto{\mathfrak{T}_{v_{n}\mid_{\Sigma_{b}}}} (v_{r,i})_{n} \mid_{\Sigma_{b}} \in H^{s+3/2}(\mathbb{R}^{n-1}; \mathbb{R}) \\
    &\chi_{r,i} \in H^{s+1/2}(\mathbb{R}^{n-1};\mathbb{R}) \xmapsto{\mathfrak{T}_{\phi}} \phi_{r,i} \in H^{s+2}(\mathbb{R}^{n-1}; L^{2}((0,b);\mathbb{R})) \\
    &\chi_{r,i} \in H^{s+1/2}(\mathbb{R}^{n-1};\mathbb{R}) \xmapsto{\mathfrak{T}_{\phi\mid_{\Sigma_{b}}}} \phi_{r,i}\mid_{\Sigma_{b}} \in H^{s+3/2}(\mathbb{R}^{n-1}; \mathbb{R}) \\
    &\chi_{r,i} \in H^{s+1/2}(\mathbb{R}^{n-1};\mathbb{R}) \xmapsto{\mathfrak{T}_{q}} q_{r,i} \in H^{s+1}(\mathbb{R}^{n-1}; L^{2}((0,b);\mathbb{R})), \label{rqmap}
\end{align}
extended in the obvious way to complex setting via $\chi = \chi_r + i \chi_i \mapsto \mathfrak{T}_\ast \chi_r + i \mathfrak{T}_\ast \chi_i = \mathfrak{T}_\ast \chi$.  Applying  Proposition \ref{fs} to these, 
we obtain locally essentially bounded and strongly measurable functions $\omega_{v}: \mathbb{R}^{n-1} \to \mathcal{L}(\mathbb{C}; L^{2}((0,b);\mathbb{C}^{n}))$ $\omega_{v_{n}\mid_{\Sigma_{b}}} : \mathbb{R}^{n-1} \to \mathcal{L}(\mathbb{C}; \mathbb{C})$, 
$\omega_{\phi} : \mathbb{R}^{n-1} \to \mathcal{L}(\mathbb{C}; L^{2}((0,b);\mathbb{C}))$, 
$\omega_{\phi\mid_{\Sigma_{b}}}: \mathbb{R}^{n-1} \to \mathcal{L}(\mathbb{C}; \mathbb{C})$, and $\omega_{q} : \mathbb{R}^{n-1} \to \mathcal{L}(\mathbb{C}; L^{2}((0,b);\mathbb{C}))$
that satisfy \(T_{v}=\mathcal{F}^{-1}\omega_{v}\mathcal{F}\); \(T_{v_{n}\mid_{\Sigma_{b}}}=\mathcal{F}^{-1}\omega_{v_{n}\mid_{\Sigma_{b}}}\mathcal{F}\); \(T_{\phi}=\mathcal{F}^{-1}\omega_{\phi}\mathcal{F}\); \(T_{\phi\mid_{\Sigma_{b}}}=\mathcal{F}^{-1}\omega_{\phi\mid_{\Sigma_{b}}}\mathcal{F}\); and \(T_{q}=\mathcal{F}^{-1}\omega_{q}\mathcal{F}\), and for a.e. \(\xi \in \mathbb{R}^{n-1}\),
\begin{align}
    \norm{\omega_{v}(\xi)x}_{L^{2}((0,b);\mathbb{C}^{n})} &\lesssim \langle\xi \rangle^{-(s+2)}\norm{T_{v}}\langle \xi\rangle^{s+1/2} \abs{x}, 
    &\abs{\omega_{v_{n} \mid_{\Sigma_{b}}}(\xi)x} &\lesssim \langle\xi \rangle^{-(s+3/2)}\norm{T_{v_{n}\mid_{\Sigma_{b}}}}\langle \xi\rangle^{s+1/2} \abs{x},  \nonumber\\
    \norm{\omega_{\phi}(\xi)x}_{L^{2}((0,b);\mathbb{C})} &\lesssim \langle\xi \rangle^{-(s+2)}\norm{T_{\phi}}\langle \xi\rangle^{s+1/2} \abs{x}, &
    \abs{\omega_{\phi\mid_{\Sigma_{b}}}(\xi)x}&\lesssim \langle\xi \rangle^{-(s+3/2)}\norm{T_{\phi\mid_{\Sigma_{b}}}}\langle \xi\rangle^{s+1/2} \abs{x}, \nonumber\\
    \norm{\omega_{q}(\xi)x}_{L^{2}((0,b);\mathbb{C})} &\lesssim \langle\xi \rangle^{-(s+1)}\norm{T_{q}}\langle \xi\rangle^{s+1/2} \abs{x},
\end{align}
for all \(x \in \mathbb{C}\).
Consequently, for a.e. \(\abs{\xi} >1\),
\begin{align}
\int_{0}^{b}\abs{\omega_{v}(\xi,x_{n})}^{2}dx_{n} &\lesssim \abs{\xi}^{-3},  &  \abs{\omega_{v_{n}\mid_{\Sigma_{b}}}(\xi)}  &\lesssim \abs{\xi}^{-1},  &  \int_{0}^{b}\abs{\omega_{\phi}(\xi,x_{n})}^{2}dx_{n} &\lesssim \abs{\xi}^{-3}, \nonumber \\
\abs{\omega_{\phi\mid_{\Sigma_{b}}}(\xi)} &\lesssim (1+\abs{\xi}^{2})^{-1/2},  &   \int_{0}^{b}\abs{\omega_{q}(\xi,x_{n})}^{2}dx_{n} &\lesssim (1+\abs{\xi}^{2})^{-1/2}.
\end{align}
 In Section  \ref{odetheory}, we obtain the asymptotics of these Fourier symbols for \(\abs{\xi} \leq 1\) via the ODE formulation of (\ref{mothereqn}).

\section{Estimates of Fourier Symbols via ODE Theory} \label{odetheory}
In this section, we compute the asymptotics of the Fourier symbols instantiated in Section \ref{instfs} for \(\abs{\xi} \leq 1\) via an ODE formulation of \eqref{mothereqn}. Applying the horizontal Fourier transform to \eqref{mothereqn} with \((\widetilde{\gamma},\alpha_{1},\alpha_{2})=(\gamma,0,\sigma'(0))\) and \((f,g,l,k,m)=(0,0,0,\chi e_{n},0)\), where \(\chi \in H^{1/2}(\mathbb{R}^{n-1};\mathbb{R})\), we obtain
\begin{align} \label{hftlinsys}
    \begin{cases}
        \gamma 2\pi i \xi_{1}\mathcal{F}(w')(\xi,x_{n})+2\pi i \xi\mathcal{F}(r)(\xi,x_{n}) -\mu(\partial_{n}^{2}-4\pi^{2}\abs{\xi}^{2})\mathcal{F}(w')(\xi,x_{n})=0 &\mbox{in \(\mathbb{R}^{n-1} \times (0,b)\)} \\
        \gamma 2\pi i \xi_{1}\mathcal{F}(w_{n})(\xi,x_{n})+\partial_{n}\mathcal{F}(r)(\xi,x_{n}) -\mu(\partial_{n}^{2}-4\pi^{2}\abs{\xi}^{2})\mathcal{F}(w_{n})(\xi,x_{n})=0 &\mbox{in \(\mathbb{R}^{n-1} \times (0,b)\)} \\
        2\pi i \xi \cdot \mathcal{F}(w')(\xi,x_{n})+\partial_{n}\mathcal{F}(w_{n})(\xi,x_{n})=0 &\mbox{in \(\mathbb{R}^{n-1} \times (0,b)\)} \\
        \gamma 2\pi i \xi_{1}\mathcal{F}(\theta)(\xi,x_{n})-\kappa(\partial_{n}^{2}-4\pi^{2}\abs{\xi}^{2})\mathcal{F}(\theta)(\xi,x_{n}) =0 &\mbox{in \(\mathbb{R}^{n-1} \times (0,b)\)} \\
        -\mu(\partial_{n}\mathcal{F}(w')(\xi,b)+2\pi i \xi\mathcal{F}(w_{n})(\xi,b))=0 &\mbox{on \(\mathbb{R}^{n-1}\)} \\
        \mathcal{F}(r)(\xi,b)-2\mu\partial_{n}\mathcal{F}(w_{n})(\xi,b)=\mathcal{F}(\chi)(\xi) &\mbox{on \(\mathbb{R}^{n-1}\)} \\
        \sigma'(0)2\pi i \xi \cdot \mathcal{F}(w')(\xi,b)+\kappa\partial_{n}\mathcal{F}(\theta)(\xi,b) = 0 &\mbox{on \(\mathbb{R}^{n-1}\)} \\
        \mathcal{F}(w)(\xi,0)=0 &\mbox{on \(\mathbb{R}^{n-1}\)} \\
        \mathcal{F}(\theta)(\xi,0) = 0 &\mbox{on \(\mathbb{R}^{n-1}\)}.        
    \end{cases}
\end{align}
Let \(q(\xi,x_{n})=\mathcal{F}(r)(\xi,x_{n})\); \(\varphi(\xi,x_{n}) = \mathcal{F}(w')(\xi,x_{n}) \cdot \frac{i\xi}{\abs{\xi}}\); \(\psi(\xi,x_{n}) = \mathcal{F}(w_{n})(\xi,x_{n})\); \(\delta(\xi,x_{n})=\mathcal{F}(\theta)(\xi,x_{n})\); and \(\beta(\xi,x_{n}) = (1-\frac{\xi \otimes \xi}{\abs{\xi}^{2}})\mathcal{F}(w')(\xi,x_{n})\). Then \(\mathcal{F}(w')(\xi,x_{n})=-i\varphi(\xi,x_{n})\frac{\xi}{\abs{\xi}}+\beta(\xi,x_{n})\). Under this change of variables, we obtain
\begin{align}
    \begin{cases}
        \gamma 2\pi i \xi_{1}\varphi - 2\pi \abs{\xi}q+\mu(-\partial_{n}^{2}+4\pi^{2}\abs{\xi}^{2})\varphi = 0 &\mbox{in \((0,b)\)} \\
        \gamma 2\pi i \xi_{1}\psi + \partial_{n}q+\mu(-\partial_{n}^{2}+4\pi^{2}\abs{\xi}^{2})\psi = 0 &\mbox{in \((0,b)\)} \\
        2\pi\abs{\xi}\varphi + \partial_{n}\psi = 0 &\mbox{in \((0,b)\)} \\
        \gamma 2\pi i \xi_{1}\delta + \kappa(-\partial_{n}^{2}+4\pi^{2}\abs{\xi}^{2})\delta = 0 &\mbox{in \((0,b)\)} \\
        -\mu(\partial_{n}\varphi - 2\pi\abs{\xi}\psi)= 0 &\mbox{for \(x_{n}=b\)} \\
        q - 2\mu\partial_{n}\psi = \mathcal{F}(\chi)(\xi) &\mbox{for \(x_{n}=b\)} \\
        \sigma'(0) 2\pi \abs{\xi}\varphi + \kappa\partial_{n}\delta = 0 &\mbox{for \(x_{n}=b\)} \\
        \varphi = \psi = \delta = 0 &\mbox{for \(x_{n}=0\)},
    \end{cases}
\end{align}
where \(\beta\) is governed by
\begin{align}
    \begin{cases}
        \gamma 2\pi i \xi_{1}\beta - \mu(\partial_{n}^{2}-4\pi^{2}\abs{\xi}^{2})\beta = 0 &\mbox{in \((0,b)\)} \\
        -\mu\partial_{n}\beta = 0 &\mbox{for \(x_{n}=b\)} \\
        \beta = 0 &\mbox{for \(x_{n}=0\)}.
    \end{cases}
\end{align}
Let us consider this system with general data. For \(F=(F_{1},F_{2}) \in L^{2}((0,b);\mathbb{C}^{2})\), \(G \in H^{1}((0,b);\mathbb{C})\), \(L \in L^{2}((0,b);\mathbb{C})\), \(K=(K_{1},K_{2}) \in \mathbb{C}^{2}\), and \(M \in \mathbb{C}\), consider
\begin{align}
    \begin{cases}
        \gamma 2\pi i \xi_{1}\varphi-2\pi\abs{\xi}q+\mu(-\partial_{n}^{2}+4\pi^{2}\abs{\xi}^{2})\varphi = F_{1} &\mbox{in \((0,b)\)} \\
        \gamma 2\pi i \xi_{1}\psi+ \partial_{n}q+\mu(-\partial_{n}^{2}+4\pi^{2}\abs{\xi}^{2})\psi = F_{2} &\mbox{in \((0,b)\)} \\
        2\pi \abs{\xi}\varphi + \partial_{n}\psi = G &\mbox{in \((0,b)\)} \\
        \gamma 2\pi i \xi_{1}\delta+\kappa(-\partial_{n}^{2}+4\pi^{2}\abs{\xi}^{2})\delta = L &\mbox{in \((0,b)\)} \\
        -\mu(\partial_{n}\varphi -2\pi\abs{\xi}\psi) = K_{1} &\mbox{for \(x_{n}=b\)} \\
        q-2\mu\partial_{n}\psi=K_{2} &\mbox{for \(x_{n}=b\)} \\
        \sigma'(0)2\pi \abs{\xi}\varphi +\kappa\partial_{n}\delta = M &\mbox{for \(x_{n}=b\)} \\
        \varphi=\psi=\delta=0 &\mbox{for \(x_{n}=0\)}.
    \end{cases}
\end{align}
Letting \(y = (\varphi,\psi,\delta,q,\partial_{n}\varphi,\partial_{n}\delta)\), we obtain
\begin{align} \label{odebulk}
    \begin{cases}
        \partial_{n}y_{1}=y_{5} &\mbox{in \((0,b)\)} \\
        \partial_{n}y_{2}=G-2\pi\abs{\xi}y_{1} &\mbox{in \((0,b)\)} \\
        \partial_{n}y_{3}=y_{6} &\mbox{in \((0,b)\)} \\
        \partial_{n}y_{4}=F_{2}+\mu(-2\pi\abs{\xi}y_{5}+\partial_{n}G)-\mu 4\pi^{2}\abs{\xi}^{2}y_{2} -\gamma 2\pi i \xi_{1}y_{2} &\mbox{in \((0,b)\)} \\
        \partial_{n}y_{5}=-\frac{1}{\mu}(F_{1}-\mu 4\pi^{2}\abs{\xi}^{2}y_{1}+2\pi \abs{\xi}y_{4}-\gamma 2\pi i \xi_{1}y_{1}) &\mbox{in \((0,b)\)} \\
        \partial_{n}y_{6}=-\frac{1}{\kappa}(L-\kappa 4\pi^{2}\abs{\xi}^{2}y_3 - \gamma 2\pi i \xi_{1} y_3) &\mbox{in \((0,b)\)}
    \end{cases}
\end{align}
with boundary conditions
\begin{align} \label{odebound}
    \begin{cases}
        -\mu(y_{5}-2\pi\abs{\xi}y_{2}) =K_{1} &\mbox{for \(x_{n}=b\)} \\
        y_{4}-2\mu(G-2\pi\abs{\xi}y_{1})=K_{2} &\mbox{for \(x_{n}=b\)} \\
        \kappa y_{6}+\sigma'(0) 2\pi \abs{\xi}y_{1}=M &\mbox{for \(x_{n}=b\)} \\
        y_{1}=y_{2}=y_{3}=0 &\mbox{for \(x_{n}=0\)}.
    \end{cases}
\end{align}
More compactly, we have
\begin{align}
    \begin{cases}
        \partial_{n}y=Ay+z &\mbox{in \((0,b)\)} \\
        My(0)+Ny(b)=d,
    \end{cases}
\end{align}
where
\begin{align}
    A=&
    \begin{pmatrix}
        0 & 0 & 0 & 0 & 1 & 0 \\
        -2\pi \abs{\xi} & 0 & 0 & 0 & 0 & 0 \\
        0 & 0 & 0 & 0 & 0 & 1 \\
        0 & -\mu 4\pi^{2}\abs{\xi}^{2}-\gamma 2\pi i \xi_{1} & 0 & 0 & -\mu 2\pi \abs{\xi} & 0 \\
        4\pi^{2}\abs{\xi}^{2}+\frac{\gamma 2\pi i \xi_{1}}{\mu} & 0 & 0 & -\frac{2\pi\abs{\xi}}{\mu} & 0 & 0 \\
        0 & 0 & 4\pi^{2}\abs{\xi}^{2}+\frac{\gamma 2\pi i \xi_{1}}{\kappa} & 0 & 0 & 0
    \end{pmatrix}
    ;
\end{align}
\(z=(0,G,0,F_{2}+\mu\partial_{n}G,-\frac{F_{1}}{\mu},-\frac{L}{\kappa})^{t}\);
\begin{align}
    M=&
    \begin{pmatrix}
        I_{3 \times 3} & 0_{3 \times 3} \\
        0_{3 \times 3} & 0_{3 \times 3}
    \end{pmatrix}
    , \quad N=
    \begin{pmatrix}
        0_{3 \times 3} & 0_{3 \times 3} \\
        N_{1} & N_{2}
    \end{pmatrix}
\end{align}
in block matrix notation, with \(3\times 3\) identity matrix \(I_{3 \times 3}\), \(3 \times 3\) zero matrix \(0_{3 \times 3}\), and 
\begin{align}
    N_{1}=
    \begin{pmatrix}
        0 & \mu 2\pi\abs{\xi} & 0 \\
        2\mu 2\pi \abs{\xi} & 0 & 0 \\
        \sigma'(0)2\pi \abs{\xi} & 0 & 0
    \end{pmatrix}
    , \quad N_{2}=
    \begin{pmatrix}
        0 & -\mu & 0 \\
        1 & 0 & 0 \\
        0 & 0 & \kappa
    \end{pmatrix}
    ;
\end{align}
and \(d=(0,0,0,K_{1},K_{2}+2\mu G,M)^{t}\). Given \(z \in L^{2}((0,b); \mathbb{C}^{6})\),
\begin{align} \label{yansatz}
    y(x_{n})=\exp(x_{n}A)y_{0} +\int_{0}^{x_{n}}\exp((x_{n}-t)A)z(t)dt
\end{align}
is a solution to the initial value problem
\begin{align}
    \begin{cases}
        \partial_{n}y = Ay+z &\mbox{in \((0,b)\)} \\
        y(0) = y_{0}.
    \end{cases}
\end{align}
Hence, the question of solvability of the system
\begin{align} \label{ode}
    \begin{cases}
        \partial_{n}y=Ay+z &\mbox{in \((0,b)\)} \\
        My(0)+Ny(b)=d
    \end{cases}
\end{align}
reduces to whether there exists \(y_{0} \in \mathbb{C}^{6}\) such that \(d=My_{0}+Ny(b)\), where \(y(b)\) is given by (\ref{yansatz}) with \(x_{n}=b\). If we let \(B=M+N\exp(bA) \in \mathbb{C}^{6 \times 6}\), then the question becomes whether there exists \(y_{0} \in \mathbb{C}^{6}\) such that
\begin{align}
    By_{0}=My_{0}+N\exp(bA)y_{0}=d-Ny(b)+N\exp(bA)y_{0}=d-N\int_{0}^{b}\exp((b-t)A)z(t)dt.
\end{align}
We answer in the positive by showing that \(B\) is invertible for each \(\xi \in \mathbb{R}^{n-1}\). In block matrix notation, we write
\begin{align}
    \exp(bA)=
    \begin{pmatrix}
        C_{1} & C_{2} \\
        C_{3} & C_{4}
    \end{pmatrix}
    ,
\end{align}
where \(C_{1}\), \(C_{2}\), \(C_{3}\), and \(C_{4}\) are the \(3 \times 3\) block matrices that make up \(\exp(bA)\). Then
\begin{align}
    B=&M+N\exp(bA)=
    \begin{pmatrix}
        I_{3 \times 3} & 0_{3 \times 3} \\
        N_{1}C_{1}+N_{2}C_{3} & N_{1}C_{2}+N_{2}C_{4}
    \end{pmatrix}
    .
\end{align}
For \(B\) to be invertible for each \(\xi \in \mathbb{R}^{n-1}\), it suffices to show that \(N_{1}C_{2}+N_{2}C_{4}\) is invertible for each \(\xi \in \mathbb{R}^{n-1}\). First, suppose that \(\xi = 0\). Then
\begin{align}
    \exp(bA\mid_{\xi = 0 })=\sum_{m=0}^{\infty}\frac{(bA\mid_{\xi = 0})^{m}}{m!}=I+bA\mid_{\xi = 0}= I_{6 \times 6} + b(e_1 \otimes e_5 + e_3 \otimes e_6 ),
\end{align}
which implies that
\begin{align}
    N_{1}C_{2}+N_{2}C_{4}\mid_{\xi = 0}=N_{2}C_{4}\mid_{\xi = 0}=
    \begin{pmatrix}
        0 & -\mu & 0 \\
        1 & 0 & 0 \\
        0 & 0 & \kappa
    \end{pmatrix}
    \begin{pmatrix}
        1 & 0 & 0 \\
        0 & 1 & 0 \\
        0 & 0 & 1
    \end{pmatrix}
    =
    \begin{pmatrix}
        0 & -\mu & 0 \\
        1 & 0 & 0 \\
        0 & 0 & \kappa
    \end{pmatrix}
    .
\end{align}
Since \(\det(N_{1}C_{2}+N_{2}C_{4}\mid_{\xi = 0})=\mu\kappa \neq 0\), we conclude that \(B\mid_{\xi = 0}\) is invertible.

Now, suppose that \(\xi \neq 0\). Choose \(m \in \mathbb{N}\) and a radial function \(\zeta \in C_{c}^{\infty}(\mathbb{R}^{n-1})\) such that \(\zeta=1\) on \(B(0,2^{m}) \setminus \overline{B(0,2^{-m})}\). We define \(k^{1}, k^{2} \in \mathscr{S}(\mathbb{R}^{n-1};\mathbb{C}^{n})\) via \(\mathcal{F}(k^{1})(\xi)=(\zeta(\xi)\frac{-i\xi}{\abs{\xi}},0)^{t}\) and \(\mathcal{F}(k^{2})(\xi)=\zeta(\xi)e_{n}\). Let \(j \in \{1,2\}\). Since \(\overline{\mathcal{F}(k^{j})(\xi)}=\mathcal{F}(k^{j})(-\xi)\), \(k^{j}\) takes values in \(\mathbb{R}^{n}\) by Lemma \(A.2\) of \cite{leoni2023traveling}. As \(k^{j}\in \cap_{s \geq 0}H^{s+1/2}(\Sigma_{b};\mathbb{R}^{n})\), we obtain \((u^{j}, \psi^{j}, p^{j}) \in \cap_{s \geq 0}\prescript{}{0}{H}^{s+2}(\Omega;\mathbb{R}^{n}) \times \prescript{}{0}{H}^{s+2}(\Omega;\mathbb{R}) \times H^{s+1}(\Omega;\mathbb{R})\) from Proposition \ref{classicalsoln} by letting \((f,g,l,k,m)=(0,0,0,k^{j},0)\). Let us define \(y^{j}(\xi,\cdot):[0,b] \to \mathbb{C}^{6}\) via
\begin{align}
    &y^{j}(\xi,x_{n}) \nonumber\\
    =&\left(\mathcal{F}((u^{j})')(\xi,x_{n})\cdot\frac{i\xi}{\abs{\xi}},\mathcal{F}(u_{n}^{j})(\xi,x_{n}), \mathcal{F}(\psi^{j})(\xi,x_{n}),\mathcal{F}(p^{j})(\xi,x_{n}),\partial_{n}\mathcal{F}((u^{j})')(\xi,x_{n}) \cdot \frac{i\xi}{\abs{\xi}},\partial_{n}\mathcal{F}(\psi^{j})(\xi,x_{n})\right)^{t}.
\end{align}
Then for \(2^{-m}<\abs{\xi}<2^{m}\),
\begin{align}
    By^{j}(\xi,0)=
    e_{j+3}    
    -N\int_{0}^{b}\exp((b-t)A) 0    dt=e_{j+3}.
\end{align}
Since \(y^{j}(\xi,0) \cdot e_{1}=y^{j}(\xi,0) \cdot e_{2} =y^{j}(\xi,0) \cdot e_{3}=0\), let us write \(y^{j}(\xi,0)=(0_{3 \times 1},\nu^{j}(\xi))^{t}\) in block vector notation. Then
\begin{align}
    By^{j}(\xi,0)=
    \begin{pmatrix}
        I_{3 \times 3} & 0_{3 \times 3} \\
        B_{3} & B_{4}
    \end{pmatrix}
    \begin{pmatrix}
        0_{3 \times 1} \\
        \nu^{j}(\xi)
    \end{pmatrix}
    =
    \begin{pmatrix}
        0_{3 \times 1} \\
        (B_{4}^{11}, B_{4}^{12}, B_{4}^{13}) \cdot \nu^{j}(\xi) \\
        (B_{4}^{21}, B_{4}^{22}, B_{4}^{23}) \cdot \nu^{j}(\xi) \\
        (B_{4}^{31}, B_{4}^{32}, B_{4}^{33}) \cdot \nu^{j}(\xi)
    \end{pmatrix}
    =e_{j+3},
\end{align}
where \(B_{3}=(B_{3}^{ij})_{1 \leq i,j \leq 3}=N_{1}C_{1}+N_{2}C_{3}\) and \(B_{4}=(B_{4}^{ij})_{1 \leq i,j \leq 3}=N_{1}C_{2}+N_{2}C_{4}\). Hence, \(B_{4}\nu^{j}(\xi)=e_{j}\). Next, we obtain \((u^{3}, \psi^{3}, p^{3}) \in \cap_{s \geq 0}\prescript{}{0}{H}^{s+2}(\Omega;\mathbb{R}^{n}) \times \prescript{}{0}{H}^{s+2}(\Omega;\mathbb{R}) \times H^{s+1}(\Omega)\) from Proposition \ref{classicalsoln} by letting \((f,g,l,k,m)=(0,0,0,0,m^{1})\), where \(\mathcal{F}(m^{1})(\xi)=\zeta(\xi)\). Let us define \(y^{3}(\xi,\cdot): [0,b] \to \mathbb{C}^{6}\) via
\begin{align}
    &y^{3}(\xi,x_{n}) \nonumber\\
    =&\left(\mathcal{F}((u^{3})')(\xi,x_{n})\cdot\frac{i\xi}{\abs{\xi}},\mathcal{F}(u_{n}^{3})(\xi,x_{n}), \mathcal{F}(\psi^{3})(\xi,x_{n}),\mathcal{F}(p^{3})(\xi,x_{n}),\partial_{n}\mathcal{F}((u^{3})')(\xi,x_{n}) \cdot \frac{i\xi}{\abs{\xi}},\partial_{n}\mathcal{F}(\psi^{3})(\xi,x_{n})\right)^{t}.
\end{align}
Then for \(2^{-m}<\abs{\xi}<2^{m}\),
\begin{align}
    By^{3}(\xi,0) = 
    e_6
    -N\int_{0}^{b}\exp((b-t)A) 0
    dt=
    e_6    .
\end{align}
Since \(y^{3}(\xi,0) \cdot e_{1}=y^{3}(\xi,0)\cdot e_{2}=y^{3}(\xi,0)\cdot e_{3}=0\), let us write \(y^{3}(\xi,0)=(0_{3 \times 1},\nu(\xi))^{t}\) in block vector notation. Then
\begin{align}
    By^{3}(\xi,0)=
    \begin{pmatrix}
        I_{3 \times 3} & 0_{3 \times 3} \\
        B_{3} & B_{4}
    \end{pmatrix}
    \begin{pmatrix}
        0_{3 \times 1} \\
        \nu(\xi)
    \end{pmatrix}
    =
    \begin{pmatrix}
        0_{3 \times 1} \\
        (B_{4}^{11}, B_{4}^{12}, B_{4}^{13})\cdot \nu(\xi) \\
        (B_{4}^{21}, B_{4}^{22}, B_{4}^{23})\cdot \nu(\xi) \\
        (B_{4}^{31}, B_{4}^{32}, B_{4}^{33})\cdot \nu(\xi)
    \end{pmatrix}
    = e_6    .
\end{align}
Hence, \(B_{4}\nu(\xi)=e_{3}\). In summary, we have shown that \(B_{4}\) has rank \(3\) for any \(\xi \in \mathbb{R}^{n-1}\) such that \(2^{-m}<\abs{\xi} <2^{m}\). As \(m \in \mathbb{N}\) is arbitrary, we conclude that \(B_{4}\) is invertible for all \(\xi \neq 0\). Consequently, \(B\) is invertible for any \(\xi \in \mathbb{R}^{n-1}\) with inverse
\begin{align}
    B^{-1}=
    \begin{pmatrix}
        I_{3 \times 3} & 0_{3 \times 3} \\
        -B_{4}^{-1} B_{3} & B_{4}^{-1}
    \end{pmatrix}
    ,
\end{align}
and
\begin{align}
    y(x_{n})&=\exp(x_{n}A)y_{0} +\int_{0}^{x_{n}}\exp((x_{n}-t)A)z(t)dt \nonumber\\
    &=\exp(x_{n}A)B^{-1}\left(d-N\int_{0}^{b}\exp((b-t)A)z(t)dt\right)+\int_{0}^{x_{n}}\exp((x_{n}-t)A)z(t)dt
\end{align}
is a solution to (\ref{ode}). Setting \(d=(0,0,0,0,\mathcal{F}(\chi)(\xi),0)^{t}\) and \(z=(0,0,0,0,0,0)^{t}\), we obtain
\begin{align}
    &\exp(x_{n}A)B^{-1}d \nonumber\\
    =&(\mathcal{F}(w')(\xi,x_{n})\cdot\frac{i\xi}{\abs{\xi}},\mathcal{F}(w_{n})(\xi,x_{n}),\mathcal{F}(\theta)(\xi,x_{n}), \mathcal{F}(r)(\xi,x_{n}),\partial_{n}\mathcal{F}(w')(\xi,x_{n})\cdot\frac{i\xi}{\abs{\xi}},\partial_{n}\mathcal{F}(\theta)(\xi,x_{n})).
\end{align}
Feeding this equation to a computer algebra system, we obtain the following low-frequency asymptotics for the Fourier symbols instantiated in Section \ref{instfs}.
\begin{theorem} \label{lfasymp}
    For \(\abs{\xi}\) small,
    \begin{align}
        \omega_{v_{n}}(\xi,x_{n})&=-\frac{2(-x_{n}+3b)x_{n}^{2}\pi^{2}}{3\mu}\abs{\xi}^{2}+\mathcal{O}(\abs{\xi}^{3}), & 
        \omega_{v_{n}\mid_{\Sigma_{b}}}(\xi) &=-\frac{4\pi^{2}b^{3}}{3\mu}\abs{\xi}^{2}+\mathcal{O}(\abs{\xi}^{3}), \nonumber \\
        \omega_{\phi}(\xi,x_{n})&=\frac{-2\pi^{2}b^{2}\sigma'(0)x_{n}}{\kappa\mu}\abs{\xi}^{2}+\mathcal{O}(\abs{\xi}^{3}), & 
        \omega_{\phi\mid_{\Sigma_{b}}}(\xi)&=\frac{-2\sigma'(0)\pi^{2}b^{3}}{\mu\kappa}\abs{\xi}^{2}+\mathcal{O}(\abs{\xi}^{3}), \nonumber\\
        \omega_{q}(\xi,x_{n})&=1+(-2b^{2}\pi^{2}-4b\pi^{2}x_{n}+2\pi^{2}x_{n}^{2})\abs{\xi}^{2}+\mathcal{O}(\abs{\xi}^{3}),  \nonumber \\
        \abs{\omega_{v'}(\xi,x_{n})}^{2}&=\varphi\overline{\varphi}=\frac{4\pi^{2}(b-x_{n}/2)^{2}x_{n}^{2}}{\mu^{2}}\abs{\xi}^{2}+\mathcal{O}(\abs{\xi}^{3}) \label{vntrlow}.
\end{align}
\end{theorem}

\section{Solvability of the Overdetermined Problem with Free Boundary Terms} \label{welldef}
Let us consider the system
\begin{align} \label{lfrs}
    \begin{cases}
        -\gamma \partial_{1}u + \nabla \cdot (pI-\mu\mathbb{D}u) + \mathfrak{g}(\nabla'\eta,0) = f &\mbox{in \(\Omega\)} \\
        \nabla \cdot u = g &\mbox{in \(\Omega\)} \\
        -\gamma \partial_{1}\psi - \kappa \Delta \psi = l &\mbox{in \(\Omega\)} \\
        (pI-\mu\mathbb{D}u)e_{n} +\sigma'(0)(\nabla'\psi,0)+\sigma(0)(0,\Delta'\eta)=k &\mbox{on \(\Sigma_{b}\)} \\
        u_{n}+\gamma\partial_{1}\eta=h &\mbox{on \(\Sigma_{b}\)} \\
        \kappa\partial_{n}\psi = m &\mbox{on \(\Sigma_{b}\)} \\
        u = 0 &\mbox{on \(\Sigma_{0}\)} \\
        \psi = 0 &\mbox{on \(\Sigma_{0}\)}.
    \end{cases}
\end{align}
We show that \eqref{lfrs} resolves into an isomorphism \(\Upsilon\) of Banach spaces between \(\mathcal{X}^{s}\) and \(\mathcal{Y}^{s}\), defined in \eqref{Xs} and \eqref{Ys}, respectively, for any \(s \geq 0\).
\subsection{Boundedness of \(\Upsilon\)}
\begin{lemma} \label{upsilonbounded}
    Let \(s \geq 0\) and \((u,\psi,p,\eta) \in \mathcal{X}^{s}\). Define \((f,g,l,k,h,m)\) via \eqref{lfrs}. Then \((f,g,l,k,h,m) \in \mathcal{Y}^{s}\) and \(\norm{(f,g,l,k,h,m)}_{\mathcal{Y}^{s}} \leq c\norm{(u,\psi,p,\eta)}_{\mathcal{X}^{s}}\) for some \(c>0\).
\end{lemma}
\begin{proof}
    Let \(s \geq 0\) and \((u,\psi,p,\eta) \in \mathcal{X}^{s}\). Since \(u \in \prescript{}{0}{H}^{s+2}(\Omega;\mathbb{R}^{n})\), \(\psi \in \prescript{}{0}{H}^{s+2}(\Omega;\mathbb{R})\), \(p \in H^{s+1}(\Omega;\mathbb{R})\), and \(\eta \in X^{s+5/2}(\mathbb{R}^{n-1};\mathbb{R})\), we obtain \(-\gamma\partial_{1}u \in H^{s+1}(\Omega;\mathbb{R}^{n})\), \(\nabla p \in H^{s}(\Omega;\mathbb{R}^{n})\), \(-\mu\Delta u \in H^{s}(\Omega;\mathbb{R}^{n})\), and \(-\mu\nabla(\nabla \cdot u) \in H^{s}(\Omega;\mathbb{R}^{n})\). By the seventh item of Theorem \(5.6\) of \cite{leoni2023traveling}, the map \(\nabla':X^{s+5/2}(\mathbb{R}^{n-1};\mathbb{R}) \to H^{s+3/2}(\mathbb{R}^{n-1};\mathbb{R}^{n-1})\) is linear and bounded. Therefore, \(f=-\gamma\partial_{1}u+\nabla p-\mu(\Delta u+\nabla(\nabla \cdot u))+\mathfrak{g}(\nabla'\eta,0) \in H^{s}(\Omega;\mathbb{R}^{n})\) and \(\norm{f}_{H^{s}(\Omega;\mathbb{R}^{n})} \lesssim \norm{u}_{H^{s+2}(\Omega;\mathbb{R}^{n})}+\norm{p}_{H^{s+1}(\Omega;\mathbb{R})}+\norm{\eta}_{X^{s+5/2}(\mathbb{R}^{n-1};\mathbb{R})}\). We note that \(g = \nabla \cdot u \in H^{s+1}(\Omega;\mathbb{R})\) and \(\norm{g}_{H^{s+1}(\Omega;\mathbb{R})} \lesssim \norm{u}_{H^{s+2}(\Omega;\mathbb{R}^{n})}\). By the eighth item of Theorem \(5.6\) of \cite{leoni2023traveling}, the map \(\partial_{1}:X^{s+5/2}(\mathbb{R}^{n-1};\mathbb{R}) \to H^{s+3/2}(\mathbb{R}^{n-1};\mathbb{R})\cap\dot{H}^{-1}(\mathbb{R}^{n-1};\mathbb{R})\) is linear and bounded. Therefore, \(h=u_{n}+\gamma\partial_{1}\eta \in H^{s+3/2}(\mathbb{R}^{n-1};\mathbb{R})\) and \(\norm{h}_{H^{s+3/2}(\mathbb{R}^{n-1};\mathbb{R})} \lesssim \norm{u}_{H^{s+2}(\Omega;\mathbb{R}^{n})}+\norm{\eta}_{X^{s+5/2}(\mathbb{R}^{n-1};\mathbb{R})}\). By Theorem \(3.1\) of \cite{leoni2023traveling},
    \begin{align}
        \left[u_{n}\mid_{\Sigma_{b}}-\int_{0}^{b}g(\cdot,x_{n})dx_{n}\right]_{\dot{H}^{-1}(\mathbb{R}^{n-1};\mathbb{R})} \leq 2\pi\sqrt{b}\norm{u}_{L^{2}(\Omega;\mathbb{R}^{n})}.
    \end{align}
    Then
    \begin{align}
        \left[h-\int_{0}^{b}g(\cdot,x_{n})dx_{n}\right]_{\dot{H}^{-1}(\mathbb{R}^{n-1};\mathbb{R})} \leq& \left[u_{n}(\cdot,b)-\int_{0}^{b}g(\cdot,x_{n})dx_{n}\right]_{\dot{H}^{-1}(\mathbb{R}^{n-1};\mathbb{R})}+\left[\gamma\partial_{1}\eta\right]_{\dot{H}^{-1}(\mathbb{R}^{n-1};\mathbb{R})} \nonumber\\
        \lesssim &\norm{\eta}_{X^{s+5/2}(\mathbb{R}^{n-1};\mathbb{R})}+\norm{u}_{L^{2}(\Omega;\mathbb{R}^{n})}.
    \end{align}
    Consequently, \(h-\int_{0}^{b}g(\cdot,x_{n})dx_{n} \in \dot{H}^{-1}(\mathbb{R}^{n-1};\mathbb{R})\). By Theorem \(5.6\) of \cite{leoni2023traveling}, the map \(\Delta': X^{s+5/2}(\mathbb{R}^{n-1};\mathbb{R}) \to H^{s+1/2}(\mathbb{R}^{n-1};\mathbb{R})\) is linear and bounded. Therefore, \(k=(\sigma'(0)\nabla'\psi,\sigma(0)\Delta'\eta)+(0',p)-\mu(\nabla u_{n}+\partial_{n}u)\in H^{s+1/2}(\mathbb{R}^{n-1};\mathbb{R}^{n})\) and
    \begin{align}
        \norm{k}_{H^{s+1/2}(\mathbb{R}^{n-1};\mathbb{R}^{n})} \lesssim \norm{\psi}_{H^{s+2}(\Omega;\mathbb{R})}+\norm{\eta}_{X^{s+5/2}(\mathbb{R}^{n-1};\mathbb{R})}+\norm{p}_{H^{s+1}(\Omega;\mathbb{R})}+\norm{u}_{H^{s+2}(\Omega;\mathbb{R}^{n})}.
    \end{align}
    We note that \(l=-\gamma\partial_{1}\psi-\kappa\Delta\psi \in H^{s}(\Omega;\mathbb{R})\) and \(\norm{l}_{H^{s}(\Omega;\mathbb{R})} \lesssim \norm{\psi}_{H^{s+2}(\Omega;\mathbb{R})}+\norm{u}_{H^{s+2}(\Omega;\mathbb{R}^{n})}\) and that \(m=\kappa\partial_{n}\psi \in H^{s+1/2}(\Sigma_{b};\mathbb{R})\) and \(\norm{m}_{H^{s+1/2}(\Sigma_{b};\mathbb{R})} \lesssim \norm{\psi}_{H^{s+2}(\Omega;\mathbb{R})}\). Then
    \begin{align}
        &\norm{(f,g,l,k,h,m)}_{\mathcal{Y}^{s}}^{2} \nonumber \\
        =&\norm{f}^{2}_{H^{s}(\Omega;\mathbb{R}^{n})}+\norm{g}^{2}_{H^{s+1}(\Omega;\mathbb{R})}+\norm{l}^{2}_{H^{s}(\Omega;\mathbb{R})}+\norm{k}^{2}_{H^{s+1/2}(\mathbb{R}^{n-1};\mathbb{R}^{n})} \nonumber \\
        &+\norm{h}^{2}_{H^{s+3/2}(\mathbb{R}^{n-1};\mathbb{R})}+\norm{m}^{2}_{H^{s+1/2}(\mathbb{R}^{n-1};\mathbb{R})}+\left[h-\int_{0}^{b}g(\cdot,x_{n})dx_{n}\right]_{\dot{H}^{-1}(\mathbb{R}^{n-1};\mathbb{R})}^{2} \nonumber \\
        \lesssim&\norm{u}_{H^{s+2}(\Omega;\mathbb{R}^{n})}^{2}+\norm{p}_{H^{s+1}(\Omega;\mathbb{R})}^{2}+\norm{\eta}_{X^{s+5/2}(\mathbb{R}^{n-1};\mathbb{R})}^{2}+\norm{\psi}_{H^{s+2}(\Omega;\mathbb{R})}^{2} \nonumber \\
        \lesssim&\norm{(u,\psi,p,\eta)}_{\mathcal{X}^{s}}^{2}.
    \end{align}
\end{proof}

\subsection{Injectivity of \(\Upsilon\)} \label{injectivity}
\begin{lemma}
    For any \(s \geq 0\), the bounded linear operator \(\Upsilon:\mathcal{X}^{s} \to \mathcal{Y}^{s}\) defined by
    \begin{align}
        \Upsilon(u,\psi,p,\eta)=\biggl(&-\gamma\partial_{1}u+\nabla \cdot (pI-\mu\mathbb{D}u)+\mathfrak{g}(\nabla'\eta,0),\nabla \cdot u, -\gamma\partial_{1}\psi-\kappa\Delta\psi,   \nonumber \\
        &(pI-\mu\mathbb{D}u)\mid_{\Sigma_{b}}e_{n}+(\sigma'(0)\nabla'\psi\mid_{\Sigma_{b}},\sigma(0)\Delta'\eta),u_{n}\mid_{\Sigma_{b}}+\gamma\partial_{1}\eta,\kappa\partial_{n}\psi \mid_{\Sigma_{b}}\biggr) \label{upsilondef}
    \end{align}
    is injective.
\end{lemma}
\begin{proof}
    Let \(s \geq 0\). By Lemma \ref{upsilonbounded}, \(\Upsilon\) is a well-defined bounded linear operator. Since \(u \in \prescript{}{0}{H}^{s+2}(\Omega;\mathbb{R}^{n})\),
    \begin{align}
        \infty>&\int_{0}^{b}\int_{\mathbb{R}^{n-1}}\abs{u(x',x_{n})}^{2}dx'dx_{n}=\int_{\mathbb{R}^{n-1}}\int_{0}^{b}\abs{\mathcal{F}(u)(\xi,x_{n})}^{2}dx_{n}d\xi, \nonumber\\
        \infty>&\int_{0}^{b}\int_{\mathbb{R}^{n-1}}\abs{\partial_{n}u(x',x_{n})}^{2}dx'dx_{n}=\int_{\mathbb{R}^{n-1}}\int_{0}^{b}\abs{\partial_{n}\mathcal{F}(u)(\xi,x_{n})}^{2}dx_{n}d\xi, \nonumber\\
        \infty>&\int_{0}^{b}\int_{\mathbb{R}^{n-1}}\abs{\partial_{n}^{2}u(x',x_{n})}^{2}dx'dx_{n}=\int_{\mathbb{R}^{n-1}}\int_{0}^{b}\abs{\partial_{n}^{2}\mathcal{F}(u)(\xi,x_{n})}^{2}dx_{n}d\xi.
    \end{align}
    Therefore, \(\mathcal{F}(u)(\xi,\cdot) \in H^{2}((0,b);\mathbb{C}^{n})\) for a.e. \(\xi \in \mathbb{R}^{n-1}\). Similarly, \(\mathcal{F}(\psi)(\xi,\cdot) \in H^{2}((0,b);\mathbb{C})\) and \(\mathcal{F}(p)(\xi,\cdot) \in H^{1}((0,b);\mathbb{C})\) for a.e. \(\xi \in \mathbb{R}^{n-1}\). By Lemma \(5.4\) of \cite{leoni2023traveling}, for any \(R>0\), there exists \(c=c(s,R,d)>0\) such that
    \begin{align}
        \int_{B(0,R)}\abs{\mathcal{F}(\eta)(\xi)}d\xi+\left(\int_{B(0,R)^{c}}(1+\abs{\xi}^{2})^{s+5/2}\abs{\mathcal{F}(\eta)(\xi)}^{2}d\xi\right)^{1/2} \leq c\norm{\eta}_{X^{s+5/2}(\mathbb{R}^{n-1};\mathbb{R})}.
    \end{align}
    Applying the horizontal Fourier transform to \eqref{lfrs}, we obtain that \(w:=\mathcal{F}(u)(\xi,x_{n})\), \(\theta:=\mathcal{F}(\psi)(\xi,x_{n})\), and \(q:=\mathcal{F}(p)(\xi,x_{n})\) satisfy
    \begin{align} \label{hftsys}
        \begin{cases}
            -\gamma2\pi i \xi_{1}w'+2\pi i \xi q-\mu(\partial_{n}^{2}-4\pi^{2}\abs{\xi}^{2})w'-\mu 2\pi i \xi \partial_{n}w_{n}+\mu 4\pi^{2}\xi(\xi \cdot w')=-\mathfrak{g}2\pi i \xi\mathcal{F}(\eta)(\xi) &\mbox{in \((0,b)\)} \\
            -\gamma 2\pi i \xi_{1}w_{n}+\partial_{n}q-\mu(\partial_{n}^{2}-4\pi^{2}\abs{\xi}^{2})w_{n}-\mu\partial_{n}^{2}w_{n}-\mu 2\pi i \xi \cdot \partial_{n}w'=0 &\mbox{in \((0,b)\)}\\
            2\pi i \xi \cdot w'+\partial_{n}w_{n}=0 &\mbox{in \((0,b)\)} \\
            -\gamma2\pi i \xi_{1}\theta-\kappa(\partial_{n}^{2}-4\pi^{2}\abs{\xi}^{2})\theta=0 &\mbox{in \((0,b)\)} \\
            q-2\mu\partial_{n}w_{n}=\mathcal{F}(-\sigma(0)\Delta'\eta)(\xi) &\mbox{for \(x_{n}=b\)} \\
            -\mu(\partial_{n}w'+2\pi i \xi w_{n})+\sigma'(0)2\pi i \xi\theta=0 &\mbox{for \(x_{n}=b\)} \\
            w_{n}=\mathcal{F}(-\gamma\partial_{1}\eta)(\xi) &\mbox{for \(x_{n}=b\)} \\
            \kappa\partial_{n}\theta=0 &\mbox{for \(x_{n}=b\)} \\
            w=0 &\mbox{for \(x_{n}=0\)} \\
            \theta=0 &\mbox{for \(x_{n}=0\)}.
        \end{cases}
    \end{align}
    Let \(v \in H^{1}((0,b);\mathbb{C}^{n})\) and \(\phi \in H^{1}((0,b);\mathbb{C})\) such that \(v(0)=0\) and \(\phi(0)=0\). Then
    \begin{align}
        &\int_{0}^{b}-\gamma 2\pi i \xi_{1}w'\cdot\overline{v'}=\int_{0}^{b}-2\pi i \xi q \cdot\overline{v'}+\mu((\partial_{n}^{2}-4\pi^{2}\abs{\xi}^{2})w')\cdot\overline{v'}+\int_{0}^{b}-\mathfrak{g}2\pi i \xi\mathcal{F}(\eta)(\xi)\cdot\overline{v'} \nonumber\\
        &\int_{0}^{b}-\gamma 2\pi i \xi_{1}w_{n}\overline{v_{n}}=\int_{0}^{b}-\partial_{n}q\overline{v_{n}}+\mu((\partial_{n}^{2}-4\pi^{2}\abs{\xi}^{2})w_{n})\overline{v_{n}} \nonumber\\
        &\int_{0}^{b}-\gamma 2\pi i \xi_{1}\theta\overline{\phi}=\int_{0}^{b}\kappa((\partial_{n}^{2}-4\pi^{2}\abs{\xi}^{2})\theta)\overline{\phi}.
    \end{align}
   Since
    \begin{align}
        -2\pi i\xi q+\mu(\partial_{n}^{2}-4\pi^{2}\abs{\xi}^{2})w'=& -2\pi i\xi q+\mu\partial_{n}(\partial_{n}w'+2\pi i \xi w_{n})+\mu(w' \otimes 2\pi i \xi+2\pi i \xi \otimes w')2\pi i \xi \nonumber\\
        -\partial_{n}q+\mu(\partial_{n}^{2}-4\pi^{2}\abs{\xi}^{2})w_{n}=&\partial_{n}(-q+2\mu\partial_{n}w_{n})-\mu(2\pi i\xi \cdot (-2\pi i \xi w_{n}-\partial_{n}w')) \nonumber \\
        \kappa(\partial_{n}^{2}-4\pi^{2}\abs{\xi}^{2})\theta=&\kappa\partial_{n}^{2}\theta-\kappa 4\pi^{2}\abs{\xi}^{2}\theta,
    \end{align}
    integration by parts yields
    \begin{align}
        \int_{0}^{b}-\gamma 2\pi i \xi_{1}w' \cdot\overline{v'}=&\int_{0}^{b}-2\pi i q\xi \cdot \overline{v'}+\mu(w' \otimes 2\pi i\xi+2\pi i \xi \otimes w')2\pi i \xi\cdot\overline{v'}-\mathfrak{g}2\pi i \xi\mathcal{F}(\eta)(\xi)\cdot\overline{v'} \nonumber\\
        &-\mu(\partial_{n}w'+2\pi i \xi w_{n})\cdot\partial_{n}\overline{v'}+(\sigma'(0)2\pi i \xi\theta(b))\cdot\overline{v'(b)} \nonumber\\
        \int_{0}^{b}-\gamma2\pi i\xi_{1}w_{n}\overline{v_{n}}=&\int_{0}^{b}\mu 2\pi i \xi \cdot (2\pi i \xi w_{n}+\partial_{n}w')\overline{v_{n}}-(-q+2\mu\partial_{n}w_{n})\partial_{n}\overline{v_{n}}+\mathcal{F}(\sigma(0)\Delta'\eta)(\xi)\overline{v_{n}(b)} \nonumber\\
        \int_{0}^{b}-\gamma2\pi i \xi_{1}\theta\overline{\phi}=&\int_{0}^{b}-\kappa 4\pi^{2}\abs{\xi}^{2}\theta\overline{\phi}-\kappa\partial_{n}\theta\partial_{n}\overline{\phi}.
    \end{align}
    Letting \(v=w\) and \(\phi=\theta\), we obtain that
    \begin{align}
        \int_{0}^{b}-\gamma 2\pi i \xi_{1}\abs{w'}^{2} =& \int_{0}^{b}-2\pi i q \xi \cdot\overline{w'}+\mu(w' \otimes 2\pi i \xi + 2\pi i \xi \otimes w')2\pi i \xi \cdot\overline{w'} \nonumber\\
        &+\mathfrak{g}\mathcal{F}(\eta)(\xi)\overline{\mathcal{F}(\nabla' \cdot u')(\xi,x_{n})}-\mu(\partial_{n}w'+2\pi i \xi w_{n})\cdot\partial_{n}\overline{w'} \nonumber\\
        &+(\sigma'(0)2\pi i \xi\theta(b))\cdot\overline{w'(b)} \nonumber\\
        \int_{0}^{b}-\gamma 2\pi i \xi_{1}\abs{w_{n}}^{2}=&\int_{0}^{b}\mu 2\pi i \xi \cdot (2\pi i \xi w_{n}+\partial_{n}w')\overline{w_{n}}-(-q+2\mu\partial_{n}w_{n})\partial_{n}\overline{w_{n}} \nonumber\\
        &+\mathcal{F}(\sigma(0)\Delta'\eta)(\xi)\overline{w_{n}(b)} \nonumber\\
        \int_{0}^{b}-\gamma 2\pi i \xi_{1}\abs{\theta}^{2}=&\int_{0}^{b}-\kappa 4\pi^{2}\abs{\xi}^{2}\abs{\theta}^{2}-\kappa\abs{\partial_{n}\theta}^{2}. \label{tempid}
    \end{align}
    Taking the real part of (\ref{tempid}), we obtain that \(\theta=0\) for all \(\xi \in \mathbb{R}^{n-1} \setminus \{0\}\). Since
    \begin{align}
        &(w' \otimes 2\pi i \xi +2\pi i \xi \otimes w')2\pi i \xi\cdot\overline{w'}=-\frac{1}{2}(w' \otimes 2\pi i \xi + 2\pi i \xi \otimes w'):\overline{w' \otimes 2\pi i \xi + 2\pi i \xi \otimes w'} \nonumber \\
        -\mu(\partial_{n}&w'+2\pi i \xi w_{n})\cdot \partial_{n}\overline{w'}+\mu 2\pi i \xi \cdot(2\pi i \xi w_{n}+\partial_{n}w')\overline{w_{n}} =-\mu(\partial_{n}w'+2\pi i \xi w_{n})(\overline{\partial_{n}w'+2\pi i \xi w_{n}}),
    \end{align}
    we obtain
    \begin{align}
        \int_{0}^{b}&-\gamma2\pi i \xi_{1}\abs{w'}^{2}-\gamma2\pi i \xi_{1}\abs{w_{n}}^{2} \nonumber\\
        = \int_{0}^{b}&-\frac{\mu}{2}\abs{w' \otimes 2\pi i \xi + 2\pi i \xi \otimes w'}^{2}-\mu\abs{\partial_{n}w'+2\pi i \xi w_{n}}^{2}-2\mu\abs{\partial_{n}w_{n}}^{2}   \nonumber\\
        &+(-\sigma(0)4\pi^{2}\abs{\xi}^{2}-\mathfrak{g})\gamma2\pi i \xi_{1}\abs{\mathcal{F}(\eta)(\xi)}^{2}. \label{velid}
    \end{align}
    Taking the real part of this identity, we obtain that \(w' \otimes 2\pi i \xi + 2\pi i \xi \otimes w'=0\), \(\partial_{n}w'+2\pi i \xi w_{n}=0\), and \(\partial_{n}w_{n}=0\). The last equation, combined with the initial condition \(w(0)=0\), implies that \(w_{n}=0\) on \([0,b]\). The second equation then reduces to \(\partial_{n}w'=0\), which, when combined with the initial condition \(w(0)=0\), implies that \(w'=0\) on \([0,b]\). Moreover, \(\mathcal{F}(-\gamma\partial_{1}\eta)(\xi)=w_{n}(b)=0\). Since \(\gamma \neq 0\), we conclude that \(\mathcal{F}(\eta)(\xi)=0\) for a.e. \(\xi \in \mathbb{R}^{n-1}\). The first equation of (\ref{hftsys}) then reduces to \(2\pi i \xi q=0\) in \((0,b)\) for a.e. \(\xi \in \mathbb{R}^{n-1}\), which implies that \(q=0\) for a.e. \(\xi \in \mathbb{R}^{n-1}\). In summary, for a.e. \(\xi \in \mathbb{R}^{n-1}\), we have \(\mathcal{F}(u)(\xi,\cdot)=0\), \(\mathcal{F}(\psi)(\xi,\cdot)=0\), \(\mathcal{F}(p)(\xi,\cdot)=0\), and \(\mathcal{F}(\eta)(\xi)=0\), which implies the injectivity of \(\Upsilon\).
\end{proof}

\subsection{The asymptotics of \(\rho(\xi)\)} \label{asymptotics}
When we incorporate all the terms involving \(\eta\) as part of the data in \eqref{lfrs}, the compatibility condition \eqref{cc} implies that for a.e. \(\xi \in \mathbb{R}^{n-1}\),
\begin{align}
    &\int_{0}^{b}\mathcal{F}(f-\mathfrak{g}(\nabla'\eta,0))(\xi,x_{n})\cdot \overline{\mathcal{F}(v)(\xi,x_{n})}-\mathcal{F}(g)(\xi,x_{n})\overline{\mathcal{F}(q)(\xi,x_{n})}+\mathcal{F}(l)(\xi,x_{n})\overline{\mathcal{F}(\phi)(\xi,x_{n})}dx_{n} \nonumber \\
    =&\mathcal{F}(k-(0,\sigma(0)\Delta'\eta))(\xi)\cdot \overline{\mathcal{F}(v)(\xi,b)}-\mathcal{F}(m)(\xi)\overline{\mathcal{F}(\phi)(\xi,b)}-\mathcal{F}(h-\gamma\partial_{1}\eta)(\xi)\overline{\mathcal{F}(\chi)(\xi)}.
\end{align}
Since
\begin{align}
    \int_{0}^{b}\mathfrak{g}\mathcal{F}((\nabla'\eta,0))(\xi) \cdot \overline{\mathcal{F}(v)(\xi,x_{n})}dx_{n}=\int_{0}^{b}\mathfrak{g}\mathcal{F}(\eta)(\xi)\cdot\overline{\partial_{n}\mathcal{F}(v_{n})(\xi,x_{n})}dx_{n}=\mathfrak{g}\overline{\mathcal{F}(v_{n})(\xi,b)}\mathcal{F}(\eta)(\xi),
\end{align}
we obtain
\begin{align}
    &\overline{\mathcal{F}(\chi)(\xi)}\int_{0}^{b}\mathcal{F}(f)(\xi,x_{n})\cdot \overline{\omega_{v}(\xi,x_{n})}-\mathcal{F}(g)(\xi,x_{n})\overline{\omega_{q}(\xi,x_{n})}+\mathcal{F}(l)(\xi,x_{n})\overline{\omega_{\phi}(\xi,x_{n})}dx_{n} \nonumber\\
    =&\overline{\mathcal{F}(\chi)(\xi)}\biggl(\mathcal{F}(k)(\xi)\cdot\overline{\omega_{v\mid_{\Sigma_{b}}}(\xi)}-\mathcal{F}(m)(\xi)\overline{\omega_{\phi\mid_{\Sigma_{b}}}(\xi)}-\mathcal{F}(h)(\xi)+\mathcal{F}(\eta)(\xi)\rho(\xi)\biggr)
\end{align}
for a.e. \(\xi \in \mathbb{R}^{n-1}\), where
\begin{align} \label{rho}
    \rho(\xi)=(\sigma(0)4\pi^{2}\abs{\xi}^{2}+\mathfrak{g})\overline{\omega_{v_{n}\mid_{\Sigma_{b}}}(\xi)}+\gamma 2\pi i \xi_{1}.
\end{align}
Let us now study the asymptotics of \(\rho(\xi)\), which is essential to show the surjectivity of \(\Upsilon\).

Suppose that \(\gamma\in \mathbb{R}\) and \((\mu,\kappa,\sigma'(0)) \in \mathbb{R}^{+} \times \mathbb{R}^{+} \times \mathbb{R}\) satisfy
\begin{align} \label{paramcond}
    \max\left\{\frac{\norm{Q_{1}}^{2}}{4},\frac{\norm{Q_{2}}^{2}}{4}\right\} \abs{\sigma'(0)}^{2} < 2\mu\kappa
\end{align}
where \(Q_{1}\) and \(Q_{2}\) are the bilinear forms defined in \eqref{Q1} and \eqref{Q2}. If \(\chi_{r},\chi_{i} \in H^{-1/2}(\Sigma_{b};\mathbb{R})\) are the real and imaginary parts of the normal stress \(\chi=\chi_{r}+i\chi_{i} \in H^{-1/2}(\Sigma_{b};\mathbb{C})\) on \(\Sigma_{b}\), then by Proposition \ref{wks} there exist \((u_{r},p_{r},\psi_{r}), (u_{i},p_{i},\psi_{i}) \in \prescript{}{0}{H}^{1}_{\sigma}(\Omega;\mathbb{R}^{n}) \times \prescript{}{0}{H}^{1}(\Omega;\mathbb{R}) \times L^{2}(\Omega;\mathbb{R})\) that uniquely satisfy
\begin{align} \label{weakform}
    B_{\gamma,\mu,\kappa,0,\sigma'(0)}((u_{r,i},p_{r,i},\psi_{r,i}),(v,\phi))=\langle F_{r,i},(v,\phi) \rangle,
\end{align}
where
\begin{align}
    B_{\gamma,\mu,\kappa,0,\sigma'(0)}((u,p,\psi),(v,\phi))=\int_{\Omega}&-\gamma \partial_{1}v \cdot u +\frac{\mu}{2}\mathbb{D}v:\mathbb{D}u-p(\nabla \cdot v)-\gamma\partial_{1}\phi\psi \nonumber\\
    &+\kappa\nabla\phi \cdot \nabla \psi +\sigma'(0)Q_{2}(u,\phi) \nonumber\\
    \langle F_{r,i},(v,\phi)\rangle = -\int_{\Sigma_{b}}&v_{n}\chi_{r,i}.
\end{align}
Letting \((v,\phi)=(u_{r,i},\psi_{r,i})\) in the bilinear form, we obtain
\begin{align}
    &\mu\norm{u_{r,i}}_{\prescript{}{0}{H}^{1}(\Omega;\mathbb{R}^{n})}^{2}+2\kappa\norm{\psi_{r,i}}_{\prescript{}{0}{H}^{1}(\Omega;\mathbb{R})}^{2}+\sigma'(0)Q_{2}(u_{r,i},\psi_{r,i}) = -\int_{\Sigma_{b}}(u_{r,i})_{n}\chi_{r,i}.
\end{align}
By \eqref{paramcond}, there exists some constant \(C_{3}(n,b) >0\) such that \(C_{3}(n,b)\sigma'(0)^{2} < \kappa\mu\). Then
\begin{align}
    \abs{Q(u_{r,i},\psi_{r,i})}=&\abs{\sigma'(0)Q_{2}(u_{r,i},\psi_{r,i})} \lesssim \frac{1}{2}\biggl(\mu\norm{u_{r,i}}_{\prescript{}{0}{H}^{1}_{\sigma}(\Omega;\mathbb{R}^{n})}^{2}+2\kappa\norm{\psi_{r,i}}_{\prescript{}{0}{H}^{1}(\Omega;\mathbb{R})}^{2}\biggr),
\end{align}
where \(Q\) is defined in (\ref{q}). It follows that
\begin{align}
    &\mu\norm{u_{r,i}}_{\prescript{}{0}{H}^{1}(\Omega;\mathbb{R}^{n})}^{2}+2\kappa\norm{\psi_{r,i}}_{\prescript{}{0}{H}^{1}(\Omega;\mathbb{R})}^{2} =-\int_{\Sigma_{b}}(u_{r,i})_{n}\chi_{r,i}-\sigma'(0)Q_{2}(u_{r,i},\psi_{r,i}) \nonumber\\
    =& -\int_{\Sigma_{b}}(u_{r,i})_{n}\chi_{r,i}-Q(u_{r,i},\psi_{r,i}) \lesssim-\int_{\Sigma_{b}}(u_{r,i})_{n}\chi_{r,i}+\frac{1}{2}\biggl(\mu\norm{u_{r,i}}_{\prescript{}{0}{H}^{1}_{\sigma}(\Omega;\mathbb{R}^{n})}^{2}+2\kappa\norm{\psi_{r,i}}_{\prescript{}{0}{H}^{1}(\Omega;\mathbb{R})}^{2}\biggr),
\end{align}
which implies that
\begin{align}
    \frac{1}{2}\biggl(\mu\norm{u_{r,i}}_{\prescript{}{0}{H}^{1}(\Omega;\mathbb{R}^{n})}^{2}+2\kappa\norm{\psi_{r,i}}_{\prescript{}{0}{H}^{1}(\Omega;\mathbb{R})}^{2}\biggr) \leq -\int_{\Sigma_{b}}(u_{r,i})_{n}\chi_{r,i}.
\end{align}
We note that
\begin{align}
    &-\int_{\Sigma_{b}}(u_{r,i})_{n}\chi_{r,i}=-\langle \chi_{r,i},(u_{r,i})_{n}\mid_{\Sigma_{b}} \rangle_{H^{-1/2}(\Sigma_{b};\mathbb{R}),H^{1/2}(\Sigma_{b};\mathbb{R})}=-\int_{\mathbb{R}^{n-1}}\mathcal{F}(\chi_{r,i})(\xi)\mathcal{F}((u_{r,i})_{n}\mid_{\Sigma_{b}})(\xi)d\xi \nonumber \\
    \leq&\biggl(\int_{\mathbb{R}^{n-1}}\min\{\abs{\xi}^{2},\abs{\xi}^{-1}\}\abs{\mathcal{F}(\chi_{r,i})(\xi)}^{2}d\xi\biggr)^{1/2} \biggl(\int_{\mathbb{R}^{n-1}}\max\{\abs{\xi}^{-2},\abs{\xi}\}\abs{\mathcal{F}((u_{r,i})_{n}\mid_{\Sigma_{b}})(\xi)}^{2}d\xi\biggr)^{1/2}.
\end{align}
By Theorem \(3.1\) of \cite{leoni2023traveling}, \([(u_{r,i})_{n}\mid_{\Sigma_{b}}]_{\dot{H}^{-1}(\mathbb{R}^{n-1};\mathbb{R})}\leq 2\pi\sqrt{b}\norm{u_{r,i}}_{L^{2}(\Omega;\mathbb{R}^{n})}\). Then
\begin{align}
    &\int_{\mathbb{R}^{n-1}}\max\{\abs{\xi}^{-2},\abs{\xi}\}\abs{\mathcal{F}((u_{r,i})_{n}\mid_{\Sigma_{b}})(\xi)}^{2}d\xi \nonumber\\
    =&\int_{\abs{\xi}\leq 1}\abs{\xi}^{-2}\abs{\mathcal{F}((u_{r,i})_{n}\mid_{\Sigma_{b}})(\xi)}^{2}d\xi+\int_{\abs{\xi}>1}\abs{\xi}\abs{\mathcal{F}((u_{r,i})_{n}\mid_{\Sigma_{b}})(\xi)}^{2}d\xi \nonumber\\
    \leq&[(u_{r,i})_{n}\mid_{\Sigma_{b}})]_{\dot{H}^{-1}(\mathbb{R}^{n-1};\mathbb{R})}^{2}+\norm{(u_{r,i})_{n}\mid_{\Sigma_{b}}}_{H^{1/2}(\Sigma_{b};\mathbb{R})}^{2} \lesssim\norm{u_{r,i}}_{\prescript{}{0}{H}^{1}(\Omega;\mathbb{R}^{n})}^{2}.
\end{align}
Hence,
\begin{align}
    &\frac{1}{2}\biggl(\mu\norm{u_{r,i}}_{\prescript{}{0}{H}^{1}(\Omega;\mathbb{R}^{n})}^{2}+2\kappa\norm{\psi_{r,i}}_{\prescript{}{0}{H}^{1}(\Omega;\mathbb{R})}^{2}\biggr) \\
    \lesssim&\biggl(\int_{\mathbb{R}^{n-1}}\min\{\abs{\xi}^{2},\abs{\xi}^{-1}\}\abs{\mathcal{F}(\chi_{r,i})(\xi)}^{2}d\xi\biggr)^{1/2}\norm{u_{r,i}}_{\prescript{}{0}{H}^{1}(\Omega;\mathbb{R}^{n})}.
\end{align}
Letting \(u=u_{r}+iu_{i}\) and \(\psi=\psi_{r}+i\psi_{i}\), we then have
\begin{align}
    \frac{1}{2}\biggl(\mu\norm{u}_{\prescript{}{0}{H}^{1}(\Omega;\mathbb{R}^{n})}^{2}+2\kappa\norm{\psi}_{\prescript{}{0}{H}^{1}(\Omega;\mathbb{R})}^{2}\biggr) \lesssim &\norm{u}_{\prescript{}{0}{H}^{1}(\Omega;\mathbb{R}^{n})}\biggl(\biggl(\int_{\mathbb{R}^{n-1}}\min\{\abs{\xi}^{2},\abs{\xi}^{-1}\}\abs{\mathcal{F}(\chi_{r})(\xi)}^{2}d\xi\biggr)^{1/2}  \nonumber\\
    &+\biggl(\int_{\mathbb{R}^{n-1}}\min\{\abs{\xi}^{2},\abs{\xi}^{-1}\}\abs{\mathcal{F}(\chi_{i})(\xi)}^{2}d\xi\biggr)^{1/2}\biggr). \label{compnorm}
\end{align}
\begin{lemma} \label{energyestimate}
    We have the following estimate.
    \begin{align}
        &\biggl(\int_{\mathbb{R}^{n-1}}\min\{\abs{\xi}^{2},\abs{\xi}^{-1}\}\abs{\mathcal{F}(\chi_{r})(\xi)}^{2}d\xi\biggr)^{1/2} \lesssim_{\gamma,\kappa} \norm{u_{r}}_{\prescript{}{0}{H}^{1}(\Omega;\mathbb{R}^{n})}, \label{estimater} \\
        &\biggl(\int_{\mathbb{R}^{n-1}}\min\{\abs{\xi}^{2},\abs{\xi}^{-1}\}\abs{\mathcal{F}(\chi_{i})(\xi)}^{2}d\xi\biggr)^{1/2} \lesssim_{\gamma,\kappa} \norm{u_{i}}_{\prescript{}{0}{H}^{1}(\Omega;\mathbb{R}^{n})}. \label{estimatei}
    \end{align}
\end{lemma}
\begin{proof}
    We focus on the proof of (\ref{estimater}), as the proof of (\ref{estimatei}) is similar. Consider the divergence problem
    \begin{align} \label{divprob}
        \begin{cases}
            \nabla \cdot u = g &\mbox{in \(\Omega\)} \\
            u_{n}=h &\mbox{on \(\Sigma_{b}\)}.
        \end{cases}
    \end{align}
    By Proposition \(2.4\) of \cite{MR4337506}, there exists a bounded linear operator \(G:\mathcal{H}(\Omega;\mathbb{R}) \to \prescript{}{0}{H}^{1}(\Omega;\mathbb{R}^{n})\) such that \(u=G(g,h)\) satisfies (\ref{divprob}), where \(\mathcal{H}(\Omega;\mathbb{R})=\{(g,h)\in L^{2}(\Omega;\mathbb{R}) \times H^{1/2}(\Sigma_{b};\mathbb{R})\mid \norm{(g,h)}_{\mathcal{H}}<\infty\}\), in which \(\norm{(g,h)}_{\mathcal{H}(\Omega;\mathbb{R})}^{2}=\norm{g}_{L^{2}(\Omega;\mathbb{R})}^{2}+\norm{h}_{H^{1/2}(\Sigma_{b};\mathbb{R})}^{2}+[h-\int_{0}^{b}g(\cdot,x_{n})dx_{n}]_{\dot{H}^{-1}(\mathbb{R}^{n-1};\mathbb{R})}^{2}\). Let us define \(\varphi\) via \(\mathcal{F}(\varphi)(\xi)=\min\{\abs{\xi}^{2},\abs{\xi}^{-1}\}\mathcal{F}(\chi_{r})(\xi)\). Observe that
    \begin{align}
        &\norm{\varphi}_{H^{1/2}(\Sigma_{b};\mathbb{R})\cap \dot{H}^{-1}(\mathbb{R}^{n-1};\mathbb{R})}^{2} \nonumber \\
        =&\int_{\abs{\xi} \leq 1}(1+\abs{\xi}^{2})^{1/2}\abs{\mathcal{F}(\varphi)(\xi)}^{2}d\xi+\int_{\abs{\xi}>1}(1+\abs{\xi}^{2})^{1/2}\abs{\mathcal{F}(\varphi)(\xi)}^{2}d\xi+\int_{\mathbb{R}^{n-1}}\abs{\xi}^{-2}\abs{\mathcal{F}(\varphi)(\xi)}^{2}d\xi \nonumber\\
        \leq&\int_{\abs{\xi}\leq 1}\sqrt{2}\abs{\xi}^{-2}\abs{\mathcal{F}(\varphi)(\xi)}^{2}d\xi+\int_{\abs{\xi}>1}\sqrt{2}\abs{\xi}\abs{\mathcal{F}(\varphi)(\xi)}^{2}d\xi+\int_{\mathbb{R}^{n-1}}\abs{\xi}^{-2}\abs{\mathcal{F}(\varphi)(\xi)}^{2}d\xi \nonumber\\
        \leq&(1+\sqrt{2})\int_{\mathbb{R}^{n-1}}\abs{\xi}^{-2}\abs{\mathcal{F}(\varphi)(\xi)}^{2}d\xi+\sqrt{2}\int_{\mathbb{R}^{n-1}}\abs{\xi}\abs{\mathcal{F}(\varphi)(\xi)}^{2}d\xi \nonumber\\
        =&2(1+\sqrt{2})\int_{\mathbb{R}^{n-1}}\min\{\abs{\xi}^{2},\abs{\xi}^{-1}\}\abs{\mathcal{F}(\chi_{r})(\xi)}^{2}d\xi.
    \end{align}
    Since \(\abs{\xi}^{2} \lesssim (1+\abs{\xi}^{2})^{-1/2}\) for \(\abs{\xi} \leq 1\) and \(\abs{\xi}^{-1} \lesssim (1+\abs{\xi}^{2})^{-1/2}\) for \(\abs{\xi}>1\), we obtain
    \begin{align}
        \norm{\varphi}_{H^{1/2}(\Sigma_{b};\mathbb{R}) \cap \dot{H}^{-1}(\mathbb{R}^{n-1};\mathbb{R})}^{2} \leq& 2(1+\sqrt{2})\norm{\chi_{r}}_{H^{-1/2}(\Sigma_{b};\mathbb{R})}^{2}<\infty.
    \end{align}
    Then by Proposition \(2.4\) of \cite{MR4337506}, \(w=G(0,\varphi)\) solves (\ref{divprob}). Letting \((v,\phi)=(w,0) \in \prescript{}{0}{H}^{1}_{\sigma}(\Omega;\mathbb{R}^{n}) \times \prescript{}{0}{H}^{1}(\Omega;\mathbb{R})\) in the weak formulation (\ref{weakform}) with \(\chi_{r} \in H^{-1/2}(\Sigma_{b};\mathbb{R})\), we obtain that for all \(\phi \in \prescript{}{0}{H}^{1}(\Omega;\mathbb{R})\),
    \begin{align} \label{weakform2}
        &\int_{\Omega}-\gamma\partial_{1}w\cdot u_{r} +\frac{\mu}{2}\mathbb{D}w:\mathbb{D}u_{r} =-\int_{\Sigma_{b}}w_{n}\chi_{r}.
    \end{align}
    Note that
    \begin{align}
        -\int_{\Sigma_{b}}w_{n}\chi_{r} =& -\int_{\mathbb{R}^{n-1}}\mathcal{F}(\chi_{r})(\xi)\mathcal{F}(\varphi)(\xi) =-\int_{\mathbb{R}^{n-1}}\min\{\abs{\xi}^{2},\abs{\xi}^{-1}\}\abs{\mathcal{F}(\chi_{r})(\xi)}^{2}d\xi.
    \end{align}
    Hence,
    \begin{align}
        \int_{\mathbb{R}^{n-1}}\min\{\abs{\xi}^{2},\abs{\xi}^{-1}\}\abs{\mathcal{F}(\chi_{r})(\xi)}^{2}d\xi &=\int_{\Omega}\gamma\partial_{1}w\cdot u_{r} -\frac{\mu}{2}\mathbb{D}w:\mathbb{D}u_{r} \nonumber \\
        &\lesssim_{\gamma,\kappa} \norm{w}_{\prescript{}{0}{H}^{1}(\Omega;\mathbb{R}^{n})}\norm{u_{r}}_{\prescript{}{0}{H}^{1}(\Omega;\mathbb{R}^{n})}.
    \end{align}
    Since
    \begin{align}
        \norm{w}_{\prescript{}{0}{H}^{1}(\Omega;\mathbb{R}^{n})}^{2}=&\norm{G(0,\varphi)}_{\prescript{}{0}{H}^{1}(\Omega;\mathbb{R}^{n})}^{2} \lesssim \norm{(0,\varphi)}_{\mathcal{H}(\Omega;\mathbb{R})}^{2} =\norm{\varphi}_{H^{1/2}(\Sigma_{b};\mathbb{R})}^{2}+[\varphi]_{\dot{H}^{-1}(\mathbb{R}^{n-1};\mathbb{R})}^{2} \nonumber\\
        \lesssim& \int_{\mathbb{R}^{n-1}}\min\{\abs{\xi}^{2},\abs{\xi}^{-1}\}\abs{\mathcal{F}(\chi_{r})(\xi)}^{2}d\xi,
    \end{align}
    we obtain
    \begin{align}
        &\int_{\mathbb{R}^{n-1}}\min\{\abs{\xi}^{2},\abs{\xi}^{-1}\}\abs{\mathcal{F}(\chi_{r})(\xi)}^{2}d\xi \nonumber\\
        \lesssim_{\gamma,\kappa} &\norm{u_{r}}_{\prescript{}{0}{H}^{1}(\Omega;\mathbb{R}^{n})}\biggl(\int_{\mathbb{R}^{n-1}}\min\{\abs{\xi}^{2},\abs{\xi}^{-1}\}\abs{\mathcal{F}(\chi_{r})(\xi)}^{2}d\xi\biggr)^{1/2}.
    \end{align}
    Hence,
    \begin{align}
        \biggl(\int_{\mathbb{R}^{n-1}}\min\{\abs{\xi}^{2},\abs{\xi}^{-1}\}\abs{\mathcal{F}(\chi_{r})(\xi)}^{2}d\xi\biggr)^{1/2} \lesssim_{\gamma,\kappa} \norm{u_{r}}_{\prescript{}{0}{H}^{1}(\Omega;\mathbb{R}^{n})}.
    \end{align}
\end{proof}
By the inequality in (\ref{compnorm})  and Lemma \ref{energyestimate}, it follows that
\begin{align}
    \norm{u}_{\prescript{}{0}{H}^{1}(\Omega;\mathbb{R}^{n})}\asymp_{\gamma,\kappa}& \biggl(\int_{\mathbb{R}^{n-1}}\min\{\abs{\xi}^{2},\abs{\xi}^{-1}\}\abs{\mathcal{F}(\chi_{r})(\xi)}^{2}d\xi\biggr)^{1/2} \nonumber \\
    &+\biggl(\int_{\mathbb{R}^{n-1}}\min\{\abs{\xi}^{2},\abs{\xi}^{-1}\}\abs{\mathcal{F}(\chi_{i})(\xi)}^{2}d\xi\biggr)^{1/2}.
\end{align}
Let \(\varphi_{r},\varphi_{i} \in L^{1}(\mathbb{R}^{n-1};\mathbb{R})\) such that \(\varphi_{r}, \varphi_{i} \geq 0\) a.e. with compact support. Define \(\phi_{r},\phi_{i} \in \cap_{s\in \mathbb{R}}H^{s}(\mathbb{R}^{n-1};\mathbb{C})\) via \(\phi_{r}=\mathcal{F}^{-1}(\sqrt{\varphi_{r}})\) and \(\phi_{i}=\mathcal{F}^{-1}(\sqrt{\varphi_{i}})\). Letting \(\chi_{r}=\phi_{r}\) and \(\chi_{i}=\phi_{i}\), we obtain
\begin{align}
    &\int_{\mathbb{R}^{n-1}}\min\{\abs{\xi}^{2},\abs{\xi}^{-1}\}\varphi_{r}(\xi)d\xi = \int_{\mathbb{R}^{n-1}}\min\{\abs{\xi}^{2},\abs{\xi}^{-1}\}\abs{\mathcal{F}(\chi_{r})(\xi)}^{2}d\xi\lesssim_{\gamma,\kappa} \norm{u}_{\prescript{}{0}{H}^{1}(\Omega;\mathbb{R}^{n})}^{2} \nonumber \\
    \lesssim&(\mu\norm{u}_{\prescript{}{0}{H}^{1}(\Omega;\mathbb{R}^{n})}^{2}+2\kappa\norm{\psi}_{\prescript{}{0}{H}^{1}(\Omega;\mathbb{R})}^{2})-\frac{1}{2}(\mu\norm{u}_{\prescript{}{0}{H}^{1}(\Omega;\mathbb{R}^{n})}^{2}+2\kappa\norm{\psi}_{\prescript{}{0}{H}^{1}(\Omega;\mathbb{R})}^{2}).
\end{align}
Since
\begin{align}
    &\frac{1}{2}(\mu\norm{u}_{\prescript{}{0}{H}^{1}(\Omega;\mathbb{R}^{n})}^{2}+2\kappa\norm{\psi}_{\prescript{}{0}{H}^{1}(\Omega;\mathbb{R})}^{2}) \nonumber\\
    =&\frac{1}{2}(\mu\norm{u_{r}}_{\prescript{}{0}{H}^{1}(\Omega;\mathbb{R}^{n})}^{2}+2\kappa\norm{\psi_{r}}_{\prescript{}{0}{H}^{1}(\Omega;\mathbb{R})}^{2})+\frac{1}{2}(\mu\norm{u_{i}}_{\prescript{}{0}{H}^{1}(\Omega;\mathbb{R}^{n})}^{2}+2\kappa\norm{\psi_{i}}_{\prescript{}{0}{H}^{1}(\Omega;\mathbb{R})}^{2}) \nonumber \\
    \gtrsim&-\sigma'(0)Q_{2}(u_{r},\psi_{r})-\sigma'(0)Q_{2}(u_{i},\psi_{i}),
\end{align}
we have
\begin{align}
    &(\mu\norm{u}_{\prescript{}{0}{H}^{1}(\Omega;\mathbb{R}^{n})}^{2}+2\kappa\norm{\psi}_{\prescript{}{0}{H}^{1}(\Omega;\mathbb{R})}^{2})-\frac{1}{2}(\mu\norm{u}_{\prescript{}{0}{H}^{1}(\Omega;\mathbb{R}^{n})}^{2}+2\kappa\norm{\psi}_{\prescript{}{0}{H}^{1}(\Omega;\mathbb{R})}^{2}) \nonumber\\
    \leq&(\mu\norm{u_{r}}_{\prescript{}{0}{H}^{1}(\Omega;\mathbb{R}^{n})}^{2}+2\kappa\norm{\psi_{r}}_{\prescript{}{0}{H}^{1}(\Omega;\mathbb{R})}^{2})+\sigma'(0)Q_{2}(u_{r},\psi_{r}) \nonumber\\
    &+(\mu\norm{u_{i}}_{\prescript{}{0}{H}^{1}(\Omega;\mathbb{R}^{n})}^{2}+2\kappa\norm{\psi_{i}}_{\prescript{}{0}{H}^{1}(\Omega;\mathbb{R})}^{2})+\sigma'(0)Q_{2}(u_{i},\psi_{i}) \nonumber\\
    =&-\int_{\Sigma_{b}}(u_{r})_{n}\chi_{r}-\int_{\Sigma_{b}}(u_{i})_{n}\chi_{i} =-\re\int_{\Sigma_{b}}u_{n}\overline{\chi} =-\re\int_{\mathbb{R}^{n-1}}\mathcal{F}(u_{n}\mid_{\Sigma_{b}})(\xi)\overline{\mathcal{F}(\chi)(\xi)} \nonumber\\
    =&-\int_{\mathbb{R}^{n-1}}\re(\omega_{v_{n}\mid_{\Sigma_{b}}}(\xi))\abs{\mathcal{F}(\chi)(\xi)}^{2} \nonumber\\
    =&-\int_{\mathbb{R}^{n-1}}\re(\omega_{v_{n}\mid_{\Sigma_{b}}}(\xi))\abs{\mathcal{F}(\chi_{r})(\xi)}^{2} -\int_{\mathbb{R}^{n-1}}\re(\omega_{v_{n}\mid_{\Sigma_{b}}}(\xi))\abs{\mathcal{F}(\chi_{i})(\xi)}^{2} \nonumber\\
    =&-\int_{\mathbb{R}^{n-1}}\re(\omega_{v_{n}\mid_{\Sigma_{b}}}(\xi))\varphi_{r}(\xi)-\int_{\mathbb{R}^{n-1}}\re(\omega_{v_{n}\mid_{\Sigma_{b}}}(\xi))\varphi_{i}(\xi).
\end{align}
Thus,
\begin{align}
    \int_{\mathbb{R}^{n-1}}\min\{\abs{\xi}^{2},\abs{\xi}^{-1}\}\varphi_{r}(\xi)d\xi \lesssim -\int_{\mathbb{R}^{n-1}}\re(\omega_{v_{n}\mid_{\Sigma_{b}}}(\xi))\varphi_{r}(\xi)-\int_{\mathbb{R}^{n-1}}\re(\omega_{v_{n}\mid_{\Sigma_{b}}}(\xi))\varphi_{i}(\xi).
\end{align}
It follows that for a.e. \(\xi \in \mathbb{R}^{n-1}\),
\begin{align} \label{revntrfs}
    \min\{\abs{\xi}^{2},\abs{\xi}^{-1}\} \lesssim -\re(\omega_{v_{n} \mid_{\Sigma_{b}}}(\xi)).
\end{align}
Recalling (\ref{rho}), we have
\begin{equation}
    \re\rho(\xi)=(\sigma(0)4\pi^{2}\abs{\xi}^{2}+\mathfrak{g})\re(\omega_{v_{n}\mid_{\Sigma_{b}}}(\xi)) \text{ and }
    \im\rho(\xi)=-(\sigma(0)4\pi^{2}\abs{\xi}^{2}+\mathfrak{g})\im(\omega_{v_{n}\mid_{\Sigma_{b}}}(\xi))+\gamma 2\pi \xi_{1}.
\end{equation}
In view of (\ref{revntrfs}) and (\ref{vntrlow}), we observe that for a.e. \(\abs{\xi} \leq 1\),
\begin{align}
    \abs{\xi_{1}}\abs{\xi}^{2} \lesssim&-\abs{\xi_{1}}\re(\omega_{v_{n}\mid_{\Sigma_{b}}}(\xi)) \lesssim \abs{2\pi \gamma\xi_{1}\re(\omega_{v_{n}\mid_{\Sigma_{b}}}(\xi))} \nonumber\\
    =&\abs{\im\biggl(2\pi i \gamma\xi_{1}\omega_{v_{n}\mid_{\Sigma_{b}}}(\xi)+(\sigma(0)4\pi^{2}\abs{\xi}^{2}+\mathfrak{g})\overline{\omega_{v_{n}\mid_{\Sigma_{b}}}(\xi)}\omega_{v_{n}\mid_{\Sigma_{b}}}(\xi)\biggr)} \nonumber\\
    =&\abs{2\pi i \gamma\xi_{1}+(\sigma(0)4\pi^{2}\abs{\xi}^{2}+\mathfrak{g})\overline{\omega_{v_{n}\mid_{\Sigma_{b}}}(\xi)}}\abs{\omega_{v_{n}\mid_{\Sigma_{b}}}(\xi)} \nonumber\\
    =&\abs{\rho(\xi)}\abs{\omega_{v_{n}\mid_{\Sigma_{b}}}(\xi)} \nonumber\\
    \lesssim&\abs{\rho(\xi)}\abs{\xi}^{2}
\end{align}
and
\begin{align}
    \abs{\xi}^{2} \lesssim &-\re(\omega_{v_{n}\mid_{\Sigma_{b}}}(\xi))\lesssim \abs{(\sigma(0)4\pi^{2}\abs{\xi}^{2}+\mathfrak{g})\re(\omega_{v_{n}\mid_{\Sigma_{b}}}(\xi))} =\abs{\re\rho(\xi)} \leq \abs{\rho(\xi)}.
\end{align}
Hence, \(\abs{\xi_{1}}^{2} + \abs{\xi}^{4} \lesssim \abs{\rho(\xi)}^{2}\) for a.e. \(\abs{\xi} \leq 1\). Since \(\sigma(0) \neq 0\), in view of (\ref{revntrfs}), we obtain that for a.e. \(\abs{\xi} >1\),
\begin{align}
    \abs{\xi} =\abs{\xi}^{2}\abs{\xi}^{-1} \lesssim -\abs{\xi}^{2}\re(\omega_{v_{n} \mid_{\Sigma_{b}}}(\xi)) \lesssim \abs{\re\rho(\xi)} \leq \abs{\rho(\xi)}.
\end{align}
Hence, \(1+\abs{\xi}^{2} \lesssim \abs{\xi}^{2} \lesssim \abs{\rho(\xi)}^{2}\) for a.e. \(\abs{\xi}>1\).

\subsection{Surjectivity of \(\Upsilon\)}
In Section \ref{injectivity}, we have shown that for \(s \geq 0\) the bounded linear operator \(\Upsilon: \mathcal{X}^{s} \to \mathcal{Y}^{s}\) defined in \eqref{upsilondef} is injective. Now, we prove its surjectivity. Fix \((f,g,l,k,h,m) \in \mathcal{Y}^{s}\). Let us define
\begin{align}
    \Xi(\xi)=&\int_{0}^{b}\mathcal{F}(f)(\xi,x_{n})\cdot\overline{\omega_{v}(\xi,x_{n})}-\mathcal{F}(g)(\xi,x_{n})\overline{\omega_{q}(\xi,x_{n})}+\mathcal{F}(l)(\xi,x_{n})\overline{\omega_{\phi}(\xi,x_{n})}dx_{n}\nonumber \\
    &-\mathcal{F}(k)(\xi)\cdot\overline{\omega_{v\mid_{\Sigma_{b}}}(\xi)}+\mathcal{F}(m)(\xi)\overline{\omega_{\phi\mid_{\Sigma_{b}}}(\xi)}+\mathcal{F}(h)(\xi)
\end{align}
and \(\mathcal{F}(\eta)(\xi)=\Xi(\xi)/\rho(\xi)\) for a.e. \(\xi \in \mathbb{R}^{n-1}\).
In view of the results in Section \ref{asymptotics}, we observe that
\begin{align}
    &\int_{B'(0,1)}\frac{\xi_{1}^{2}+\abs{\xi}^{4}}{\abs{\xi}^{2}}\abs{\mathcal{F}(\eta)(\xi)}^{2}d\xi+\int_{B'(0,1)^{c}}(1+\abs{\xi}^{2})^{s+5/2}\abs{\mathcal{F}(\eta)(\xi)}^{2}d\xi \nonumber\\
    \lesssim&\int_{B'(0,1)}\frac{\abs{\rho(\xi)}^{2}}{\abs{\xi}^{2}}\abs{\mathcal{F}(\eta)(\xi)}^{2}d\xi+\int_{B'(0,1)^{c}}(1+\abs{\xi}^{2})^{s+3/2}\abs{\rho(\xi)}^{2}\abs{\mathcal{F}(\eta)(\xi)}^{2}d\xi \nonumber\\
    =&\int_{B'(0,1)}\frac{\abs{\Xi(\xi)}^{2}}{\abs{\xi}^{2}}d\xi+\int_{B'(0,1)^{c}}(1+\abs{\xi}^{2})^{s+3/2}\abs{\Xi(\xi)}^{2}d\xi.
\end{align}
\begin{lemma} \label{xiintestimate}
The following estimate holds:
    \begin{align}
        \int_{B'(0,1)}\frac{\abs{\Xi(\xi)}^{2}}{\abs{\xi}^{2}}d\xi+\int_{B'(0,1)^{c}}(1+\abs{\xi}^{2})^{s+3/2}\abs{\Xi(\xi)}^{2}d\xi\lesssim \norm{(f,g,l,k,h,m)}_{\mathcal{Y}^{s}}^{2}.
    \end{align}
\end{lemma}
\begin{proof}
    Since
    \begin{align}
        \Xi(\xi)=&\int_{0}^{b}\mathcal{F}(f)(\xi,x_{n})\cdot \overline{\omega_{v}(\xi,x_{n})}-\mathcal{F}(g)(\xi,x_{n})(\overline{\omega_{q}(\xi,x_{n})}-1)+\mathcal{F}(l)(\xi,x_{n})\overline{\omega_{\phi}(\xi,x_{n})}dx_{n} \nonumber\\
        &-\mathcal{F}(k)(\xi)\cdot\overline{\omega_{v\mid_{\Sigma_{b}}}(\xi)}+\mathcal{F}(m)(\xi)\overline{\omega_{\phi\mid_{\Sigma_{b}}}(\xi)}+\mathcal{F}(h)(\xi)-\int_{0}^{b}\mathcal{F}(g)(\xi,x_{n})dx_{n},
    \end{align}
    by the Cauchy-Schwarz inequality,
    \begin{align}
        \abs{\Xi(\xi)}^{2} \lesssim&\left(\int_{0}^{b}\abs{\mathcal{F}(f)(\xi,x_{n})}^{2}dx_{n}\right)\left(\int_{0}^{b}\abs{\omega_{v}(\xi,x_{n})}^{2}dx_{n}\right) \nonumber\\
        &+\left(\int_{0}^{b}\abs{\mathcal{F}(g)(\xi,x_{n})}^{2}dx_{n}\right)\left(\int_{0}^{b}\abs{\omega_{q}(\xi,x_{n})-1}^{2}dx_{n}\right) \nonumber\\
        &+\left(\int_{0}^{b}\abs{\mathcal{F}(l)(\xi,x_{n})}^{2}dx_{n}\right)\left(\int_{0}^{b}\abs{\omega_{\phi}(\xi,x_{n})}^{2}dx_{n}\right) \nonumber\\
        &+\abs{\mathcal{F}(k)(\xi)}^{2}\abs{\omega_{v\mid_{\Sigma_{b}}}(\xi)}^{2} +\abs{\mathcal{F}(m)(\xi)}^{2}\abs{\omega_{\phi\mid_{\Sigma_{b}}}(\xi)}^{2}+\abs{\mathcal{F}(h)(\xi)-\int_{0}^{b}\mathcal{F}(g)(\xi,x_{n})dx_{n}}^{2}.
    \end{align}
    By Theorem \ref{lfasymp}, we obtain that for \(\abs{\xi} \leq 1\),
    \begin{align}
        \frac{\abs{\Xi(\xi)}^{2}}{\abs{\xi}^{2}} \lesssim &\int_{0}^{b}\abs{\mathcal{F}(f)(\xi,x_{n})}^{2}dx_{n}+\int_{0}^{b}\abs{\mathcal{F}(g)(\xi,x_{n})}^{2}dx_{n}+\int_{0}^{b}\abs{\mathcal{F}(l)(\xi,x_{n})}^{2}dx_{n} \nonumber \\
        &+\abs{\mathcal{F}(k)(\xi)}^{2}+\abs{\mathcal{F}(m)(\xi)}^{2}+\frac{1}{\abs{\xi}^{2}}\abs{\mathcal{F}(h)(\xi)-\int_{0}^{b}\mathcal{F}(g)(\xi,x_{n})dx_{n}}^{2}.
    \end{align}
    Then
    \begin{align}
        &\int_{B'(0,1)}\frac{\abs{\Xi(\xi)}^{2}}{\abs{\xi}^{2}}d\xi \nonumber\\
        \lesssim &\int_{B'(0,1)}\int_{0}^{b}\abs{\mathcal{F}(f)(\xi,x_{n})}^{2}dx_{n}d\xi+\int_{B'(0,1)}\int_{0}^{b}\abs{\mathcal{F}(g)(\xi,x_{n})}^{2}dx_{n}d\xi \nonumber\\
        &+\int_{B'(0,1)}\int_{0}^{b}\abs{\mathcal{F}(l)(\xi,x_{n})}^{2}dx_{n}d\xi +\int_{B'(0,1)}\abs{\mathcal{F}(k)(\xi)}^{2}d\xi+\int_{B'(0,1)}\abs{\mathcal{F}(m)(\xi)}^{2}d\xi \nonumber \\
        &+\int_{B'(0,1)}\frac{1}{\abs{\xi}^{2}}\abs{\mathcal{F}(h)(\xi)-\int_{0}^{b}\mathcal{F}(g)(\xi,x_{n})dx_{n}}^{2}d\xi \nonumber\\
        \lesssim& \norm{f}_{L^{2}(\Omega;\mathbb{R}^{n})}^{2}+\norm{g}_{L^{2}(\Omega;\mathbb{R})}^{2}+\norm{l}_{L^{2}(\Omega;\mathbb{R})}^{2}+\norm{k}_{L^{2}(\Sigma_{b};\mathbb{R}^{n})}^{2}+\norm{m}_{L^{2}(\Sigma_{b};\mathbb{R})}^{2} \nonumber\\
        &+\left[h-\int_{0}^{b}g(\cdot,x_{n})dx_{n}\right]_{\dot{H}^{-1}(\mathbb{R}^{n-1};\mathbb{R})}^{2} \nonumber\\
        \lesssim&\norm{(f,g,l,k,h,m)}_{\mathcal{Y}^{s}}^{2}.
    \end{align}
    For \(\abs{\xi}>1\), by the Cauchy-Schwarz inequality,
    \begin{align}
        \abs{\Xi(\xi)}^{2} \lesssim &\left(\int_{0}^{b}\abs{\mathcal{F}(f)(\xi,x_{n})}^{2}dx_{n}\right)\left(\int_{0}^{b}\abs{\omega_{v}(\xi,x_{n})}^{2}dx_{n}\right) \nonumber\\
        &+\left(\int_{0}^{b}\abs{\mathcal{F}(g)(\xi,x_{n})}^{2}dx_{n}\right)\left(\int_{0}^{b}\abs{\omega_{q}(\xi,x_{n})}^{2}dx_{n}\right) \nonumber\\
        &+\left(\int_{0}^{b}\abs{\mathcal{F}(l)(\xi,x_{n})}^{2}dx_{n}\right)\left(\int_{0}^{b}\abs{\omega_{\phi}(\xi,x_{n})}^{2}dx_{n}\right) \nonumber\\
        &+\abs{\mathcal{F}(k)(\xi)}^{2}\abs{\omega_{v\mid_{\Sigma_{b}}}(\xi)}^{2}+\abs{\mathcal{F}(m)(\xi)}^{2}\abs{\omega_{\phi\mid_{\Sigma_{b}}}(\xi)}^{2}+\abs{\mathcal{F}(h)(\xi)}^{2} \nonumber\\
        \lesssim&\left(\int_{0}^{b}\abs{\mathcal{F}(f)(\xi,x_{n})}^{2}dx_{n}\right)\min\{\abs{\xi}^{2},\abs{\xi}^{-2}\} \nonumber\\
        &+\left(\int_{0}^{b}\abs{\mathcal{F}(g)(\xi,x_{n})}^{2}dx_{n}\right)(1+\abs{\xi}^{2})^{-1/2} +\left(\int_{0}^{b}\abs{\mathcal{F}(l)(\xi,x_{n})}^{2}dx_{n}\right)\min\{\abs{\xi}^{2},\abs{\xi}^{-2}\} \nonumber\\
        &+\abs{\mathcal{F}(k)(\xi)}^{2}(\min\{\abs{\xi}^{2},\abs{\xi}^{-1}\})^{2}+\abs{\mathcal{F}(m)(\xi)}^{2}(1+\abs{\xi}^{2})^{-1}+\abs{\mathcal{F}(h)(\xi)}^{2}.
    \end{align}
    Then
\begin{multline}
    \int_{B'(0,1)^{c}}(1+\abs{\xi}^{2})^{s+3/2}\abs{\Xi(\xi)}^{2}d\xi  
        \leq \norm{f}_{H^{s}(\Omega;\mathbb{R}^{n})}^{2}+\norm{g}_{H^{s+1}(\Omega;\mathbb{R})}^{2}+\norm{l}_{H^{s}(\Omega;\mathbb{R})}^{2} \\
         +\norm{k}_{H^{s+1/2}(\Sigma_{b};\mathbb{R}^{n})}^{2}+\norm{m}_{H^{s+1/2}(\Sigma_{b};\mathbb{R})}^{2}+\norm{h}_{H^{s+3/2}(\Sigma_{b};\mathbb{R})}^{2}          \lesssim \norm{(f,g,l,k,h,m)}_{\mathcal{Y}^{s}}^{2}.
\end{multline}
\end{proof}
By Lemma \ref{xiintestimate},
\begin{align}
    &\int_{B'(0,1)}\frac{\xi_{1}^{2}+\abs{\xi}^{4}}{\abs{\xi}^{2}}\abs{\mathcal{F}(\eta)(\xi)}^{2}d\xi+\int_{B'(0,1)^{c}}(1+\abs{\xi}^{2})^{s+5/2}\abs{\mathcal{F}(\eta)(\xi)}^{2}d\xi \nonumber \\
    \lesssim&\int_{B'(0,1)}\frac{\abs{\Xi(\xi)}^{2}}{\abs{\xi}^{2}}d\xi +\int_{B'(0,1)^{c}}(1+\abs{\xi}^{2})^{s+3/2}\abs{\Xi(\xi)}^{2}d\xi \nonumber\\
    \lesssim&\norm{(f,g,l,k,h,m)}_{\mathcal{Y}^{s}}^{2}.
\end{align}
Hence, \(\eta \in X^{s+5/2}(\mathbb{R}^{n-1};\mathbb{R})\) with \(\norm{\eta}_{X^{s+5/2}(\mathbb{R}^{n-1};\mathbb{R})} \lesssim \norm{(f,g,l,k,h,m)}_{\mathcal{Y}^{s}}\). Since \(\omega_{v_{n}\mid_{\Sigma_{b}}}(\xi)=\omega_{v}(\xi,b) \cdot e_{n}\), we may write \(\Xi(\xi)=\rho(\xi)\mathcal{F}(\eta)(\xi)\) as
\begin{align}
    0=&\int_{0}^{b}\mathcal{F}(f)(\xi,x_{n})\cdot\overline{\omega_{v}(\xi,x_{n})}-\mathcal{F}(g)(\xi,x_{n})\overline{\omega_{q}(\xi,x_{n})}+\mathcal{F}(l)(\xi,x_{n})\overline{\omega_{\phi}(\xi,x_{n})}dx_{n} \nonumber \\
    &-\biggl(\mathcal{F}(k)(\xi)+(\sigma(0)4\pi^{2}\abs{\xi}^{2}+\mathfrak{g})\mathcal{F}(\eta)(\xi)e_{n}\biggr)\cdot \overline{\omega_{v}(\xi,b)}+\mathcal{F}(m)(\xi)\overline{\omega_{\phi \mid_{\Sigma_{b}}}(\xi)} \nonumber\\
    &+\mathcal{F}(h)(\xi)-\gamma2\pi i \xi_{1}\mathcal{F}(\eta)(\xi).
\end{align}
We note that \(0=2\pi i \xi \cdot \mathcal{F}(v')(\xi,x_{n})+\partial_{n}\mathcal{F}(v_{n})(\xi,x_{n}) =2\pi i \xi \cdot \omega_{v}'(\xi,x_{n})\mathcal{F}(\chi)(\xi)+\partial_{n}\omega_{v_{n}}(\xi,x_{n})\mathcal{F}(\chi)(\xi)\). Since \(v_{n}(x',0)=0\) for all \(x' \in \mathbb{R}^{n-1}\), we obtain that \(0=\mathcal{F}(v_{n})(\xi,0)=\omega_{v_{n}}(\xi,0)\mathcal{F}(\chi)(\xi)\) for all \(\xi \in \mathbb{R}^{n-1}\). As \(\chi \in H^{s+1/2}(\mathbb{R}^{n-1};\mathbb{R})\) is arbitrary, we have that \(\omega_{v_{n}}(\xi,0)=0\). Then
\begin{align}
    \omega_{v_{n}}(\xi,b)=\omega_{v_{n}}(\xi,b)-\omega_{v_{n}}(\xi,0)=\int_{0}^{b}\partial_{n}\omega_{v_{n}}(\xi,x_{n})dx_{n}=\int_{0}^{b}-2\pi i \xi \cdot \omega_{v}'(\xi,x_{n})dx_{n}.
\end{align}
Hence,
\begin{align}
    0=&\int_{0}^{b}\biggl(\mathcal{F}(f)(\xi,x_{n})-(\mathfrak{g}\mathcal{F}(\eta)(\xi)2\pi i \xi,0)\biggr)\cdot \overline{\omega_{v}(\xi,x_{n})}-\mathcal{F}(g)(\xi,x_{n})\overline{\omega_{q}(\xi,x_{n})} \nonumber \\
    &+\mathcal{F}(l)(\xi,x_{n})\overline{\omega_{\phi}(\xi,x_{n})}dx_{n} -\biggl(\mathcal{F}(k)(\xi)+\sigma(0)4\pi^{2}\abs{\xi}^{2}\mathcal{F}(\eta)(\xi)e_{n}\biggr)\cdot \overline{\omega_{v}(\xi,b)} \nonumber \\
    &+\mathcal{F}(m)(\xi)\overline{\omega_{\phi\mid_{\Sigma_{b}}}(\xi)}+\mathcal{F}(h)(\xi) -\gamma 2\pi i \xi_{1}\mathcal{F}(\eta)(\xi). \label{line1}
\end{align}
Since \(\eta \in X^{s+5/2}(\mathbb{R}^{n-1})\), we have \((-\nabla'\eta,0) \in H^{s+3/2}(\Omega;\mathbb{R}^{n}) \subseteq H^{s}(\Omega;\mathbb{R}^{n})\), \(\Delta'\eta \in H^{s+1/2}(\mathbb{R}^{n-1})\), and \(\partial_{1}\eta \in H^{s+3/2}(\mathbb{R}^{n-1}) \cap \dot{H}^{-1}(\mathbb{R}^{n-1})\), all of which follow from Theorem \(5.6\) of \cite{leoni2023traveling}. Hence, \(f-\mathfrak{g}(\nabla'\eta,0)\in H^{s}(\Omega;\mathbb{R}^{n})\), \(g \in H^{s+1}(\Omega;\mathbb{R})\), \(l \in H^{s}(\Omega;\mathbb{R})\), \(h-\gamma\partial_{1}\eta \in H^{s+3/2}(\mathbb{R}^{n-1})\), \(k-\sigma(0)\Delta'\eta e_{n} \in H^{s+1/2}(\Sigma_{b};\mathbb{R}^{n})\), \(m \in H^{s+1/2}(\Sigma_{b};\mathbb{R})\), and \((h-\gamma\partial_{1}\eta)-\int_{0}^{b}g(\cdot,x_{n})dx_{n}\in \dot{H}^{-1}(\mathbb{R}^{n-1})\). That is,
\begin{align}
    &(f-\mathfrak{g}(\nabla'\eta,0),g,l,k-\sigma(0)\Delta'\eta e_{n},h-\gamma\partial_{1}\eta,m) \nonumber \\
    \in&H^{s}(\Omega;\mathbb{R}^{n}) \times H^{s+1}(\Omega;\mathbb{R}) \times H^{s}(\Omega;\mathbb{R}) \times H^{s+1/2}(\Sigma_{b};\mathbb{R}^{n}) \times H^{s+3/2}(\mathbb{R}^{n-1};\mathbb{R}) \times H^{s+1/2}(\Sigma_{b};\mathbb{R}).
\end{align}
Moreover, from the equation in (\ref{line1}), we have
\begin{align}
    0=&\int_{0}^{b}\mathcal{F}(f-\mathfrak{g}(\nabla'\eta,0))(\xi,x_{n})\cdot \overline{\mathcal{F}(v)(\xi,x_{n})}-\mathcal{F}(g)(\xi,x_{n})\overline{\mathcal{F}(q)(\xi,x_{n})} \nonumber\\
    &+\mathcal{F}(l)(\xi,x_{n})\overline{\mathcal{F}(\phi)(\xi,x_{n})}dx_{n} -\mathcal{F}(k-\sigma(0)\Delta'\eta e_{n})(\xi)\cdot\overline{\mathcal{F}(v)(\xi,b)}+\mathcal{F}(m)(\xi)\overline{\mathcal{F}(\phi)(\xi,b)} \nonumber\\
    &+\mathcal{F}(h-\gamma\partial_{1}\eta)(\xi,b)\overline{\mathcal{F}(\chi)(\xi)}.
\end{align}
Hence,
\begin{align}
    &\int_{\Omega}(f(x)-\mathfrak{g}(\nabla'\eta(x'),0))\cdot v(x)-g(x)q(x)+l(x)\phi(x)dx \nonumber\\
    &-\int_{\Sigma_{b}}(k(x')-\sigma(0)\Delta'\eta(x')e_{n})\cdot v(x',b)-m(x')\phi(x',b)-(h(x')-\gamma\partial_{1}\eta(x'))\chi(x')=0.
\end{align}
It follows that \((f-\mathfrak{g}(\nabla'\eta,0),g,l,k-\sigma(0)\Delta'\eta e_{n},h-\gamma\partial_{1}\eta,m) \in\mathcal{Z}^{s}\), where \(\mathcal{Z}^{s}\) is as in \eqref{Zs}. By Proposition \ref{Zsiso}, there exist \((u,\psi,p) \in \prescript{}{0}{H}^{s+2}(\Omega;\mathbb{R}^{n}) \times \prescript{}{0}{H}^{s+2}(\Omega;\mathbb{R})\times H^{s+1}(\Omega;\mathbb{R})\) such that \(T(u,\psi,p)=(f-\mathfrak{g}(\nabla'\eta,0),g,l,k-\sigma(0)\Delta'\eta e_{n},h-\gamma\partial_{1}\eta,m)\). It follows that \(\Upsilon\) is surjective.

\section{Nonlinear Analysis}
In this section, we show that all of the nonlinear maps that make up \eqref{fprsmf} are at least \(C^{1}\), which is necessary to apply an implicit function theorem as explained in Section \ref{ptc}.
\begin{lemma} \label{c2cont}
    Fix \(d \geq 1\), \(p \geq 1\), and \(q \geq 1\). Let \(\mathfrak{D} \subseteq \mathbb{R}^{d}\) be an open set and \(\mathbb{N} \ni k \geq \frac{d}{2}+1\). If \(f \in C_{b}^{k+2}(\mathbb{R}^{p};\mathbb{R}^{q})\) and \(f(0)=0\), then the map \(\Lambda_{f}:H^{k}(\mathfrak{D};\mathbb{R}^{p}) \to H^{k}(\mathfrak{D};\mathbb{R}^{q})\) defined by \(\Lambda_{f}(u)=f \circ u\) is well-defined and \(C^{1}\).
\end{lemma}
We omit the proof of Lemma \ref{c2cont}, as it is similar to that of Theorem A.8 of \cite{MR4785303}.
\begin{proposition} \label{8.3analog}
    Let \(n \geq 2\) and \(\mathbb{N} \ni s > n/2\). Suppose that \(\sigma\in C_{b}^{s+4}(\mathbb{R};\mathbb{R})\), \(\Gamma(\cdot,\cdot) \in C_{b}^{s+3}(\mathbb{R} \times\mathbb{R}_{sym}^{n \times n};\mathbb{R}_{sym}^{n \times n})\), and \(\Phi(\cdot,\cdot) \in C_{b}^{s+3}(\mathbb{R} \times \mathbb{R}^{n};\mathbb{R}^{n})\). Then there exists a constant \(\delta=\delta(n,s,b)>0\) such that, if for \(\gamma \in \mathbb{R}\) and \((u,\psi,p,\eta) \in U_{\delta}^{s}=\{(\mathrm{u},\varphi,\mathrm{p},\varepsilon) \in \mathcal{X}^{s} \mid \norm{\varepsilon}_{X^{s+5/2}}<\delta\} \subseteq \mathcal{X}^{s}\) we define
    \begin{align}
        &f: \Omega \to \mathbb{R}^{n} \text{ via } f=-\gamma \partial_{1}^{\mathcal{A}}u + u \cdot \nabla_{\mathcal{A}}u+\mathfrak{g}(\nabla'\eta,0) + \nabla_{\mathcal{A}}p-\nabla_{\mathcal{A}}\cdot \Gamma(\psi,\mathbb{D}_{\mathcal{A}}u) \\
        &g: \Omega \to \mathbb{R} \text{ via } g=J (\nabla_{\mathcal{A}} \cdot u) \\
        &l: \Omega \to \mathbb{R} \text{ via } l=-\gamma \partial_{1}^{\mathcal{A}}\psi + u \cdot \nabla_{\mathcal{A}}\psi +\nabla_{\mathcal{A}} \cdot \Phi(\psi,\nabla_{\mathcal{A}}\psi) \\
        &k:\Sigma_{b} \to \mathbb{R}^{n} \text{ via } k=-p\mathcal{N}(\eta)+\Gamma(\psi,\mathbb{D}_{\mathcal{A}}u)\mathcal{N}(\eta)-\sigma(\psi)\mathcal{H}(\eta)\mathcal{N}(\eta)-\nabla_{\Sigma}^{\mathcal{A}}\sigma(\psi)\sqrt{1+\abs{\nabla'\eta}^{2}} \\
        &h: \Sigma_{b} \to \mathbb{R} \text{ via } h=u \cdot \mathcal{N}(\eta)+\gamma \partial_{1}\eta \\
        &m:\Sigma_{b} \to \mathbb{R} \text{ via } m=\Phi(\psi,\nabla_{\mathcal{A}}\psi)\cdot\frac{\mathcal{N}(\eta)}{\sqrt{1+\abs{\nabla'\eta}^{2}}},
    \end{align}
    where \(\mathcal{A}\) is given by (\ref{AJ}), \(J\) is given by (\ref{AJ}), and \(\mathcal{N}(\eta)=(-\nabla'\eta,1)\), then \((f,g,l,k,h,m) \in \mathcal{Y}^{s}\). Moreover, the map \(\mathbb{R} \times U_{\delta}^{s} \to \mathcal{Y}^{s}\) defined via \((\gamma,u,\psi,p,\eta) \mapsto (f,g,l,k,h,m)\) is \(C^{1}\).
\end{proposition}
\begin{proof}
    Before we commence the proof, let us introduce a family of function spaces that are originally defined in \cite{leoni2023traveling} and make a brief appearance in our proof. For \(t \geq 0\), \(\mathbb{N} \ni j \geq 2\), and \(\zeta \in C_{b}^{0,1}(\mathbb{R}^{j-1};\mathbb{R})\) such that \(\inf \zeta>0\), the linear space
    \begin{align}
        Y^{t}(\Omega_{\zeta})=&H^{t}(\Omega_{\zeta})+X^{t}(\mathbb{R}^{j-1}) \nonumber\\
        =&\{f \in L^{1}_{loc}(\Omega_{\zeta}) \mid \mbox{there exists \(g \in H^{t}(\Omega_{\zeta})\) and \(h \in X^{t}(\mathbb{R}^{j-1})\) such that} \nonumber\\
        &\mbox{\(f(x)=g(x)+h(x')\) for a.e. \(x \in \Omega_{\zeta}\)}\}
    \end{align}
    equipped with the norm \(\norm{f}_{Y^{t}}=\inf\{\norm{g}_{H^{t}}+\norm{h}_{X^{t}} \mid f=g+h\}\) is Banach by Theorem \(5.7\) of \cite{leoni2023traveling}. 

    To begin the proof, we fix \(n \geq 2\) and \(r> n/2\). By the second item of Theorem \(5.11\) of \cite{leoni2023traveling}, \(X^{r}(\mathbb{R}^{n-1};\mathbb{R}) \subseteq Y^{r}(\Omega;\mathbb{R})\). By the third item of Theorem \(5.11\) of \cite{leoni2023traveling} with \(k=0\) and \(s=r\), there exists \(c_{1}=c_{1}(n,r,b)>0\) such that \(\norm{f}_{C_{b}^{0}(\Omega;\mathbb{R})} \leq c_{1}\norm{f}_{Y^{r}(\Omega;\mathbb{R})}\) for any \(f \in Y^{r}(\Omega;\mathbb{R})\) and we have the continuous inclusion \(Y^{r}(\Omega;\mathbb{R}) \hookrightarrow \{f \in C_{b}^{0}(\Omega;\mathbb{R}) \mid \lim_{\abs{x'} \to \infty}f(x)=0 \} \subseteq C_{b}^{0}(\Omega;\mathbb{R})\). By Theorem \(5.14\) of \cite{leoni2023traveling}, there exists \(c_{2}=c_{2}(n,r,b)>0\) such that \(\norm{fg}_{H^{r}(\Omega;\mathbb{R})} \leq c_{2}\norm{f}_{Y^{r}(\Omega;\mathbb{R})}\norm{g}_{H^{r}(\Omega;\mathbb{R})}\) for any \(f \in Y^{r}(\Omega;\mathbb{R})\) and any \(g \in H^{r}(\Omega;\mathbb{R})\). Let \(c=c(n,r,b)>0\) denote the larger of the constants \(c_{1},c_{2}>0\) and define \(B_{Y^{r}(\Omega;\mathbb{R})}(0,b/(2c))=\{f \in Y^{r}(\Omega;\mathbb{R})\mid \norm{f}_{Y^{r}(\Omega;\mathbb{R})}<b/(2c)\}\). Then by Theorem \(5.16\) of \cite{leoni2023traveling}, \(\Gamma_{1}, \Gamma_{2}:B_{Y^{r}(\Omega;\mathbb{R})}(0,b/(2c)) \times H^{r}(\Omega;\mathbb{R}) \to H^{r}(\Omega;\mathbb{R})\) given by \(\Gamma_{1}(f,g)=\frac{g}{b+f}\) and \(\Gamma_{2}(f,g)=\frac{gf}{b+f}\) are well-defined and smooth. If \(\delta_{1}(n,r,b) = b/(2c(n,r,b))\), then \(\norm{f}_{Y^{r}(\Omega;\mathbb{R})} \leq \norm{f}_{X^{r}(\mathbb{R}^{n-1};\mathbb{R})} < \delta_{1} = b/(2c)\) for any \(f \in B_{X^{r}(\mathbb{R}^{n-1};\mathbb{R})}(0,\delta_{1})\). Hence, the maps \(\Gamma_{1}, \Gamma_{2}: B_{X^{r}(\mathbb{R}^{n-1};\mathbb{R})}(0,\delta_{1}) \times H^{r}(\Omega;\mathbb{R}) \to H^{r}(\Omega;\mathbb{R})\) given by \(\Gamma_{1}(f,g)=\frac{g}{b+f}\) and \(\Gamma_{2}(f,g)=\frac{g f}{b+f}\) are well-defined and smooth. By Theorem \(A.12\) of \cite{leoni2023traveling}, since \(r > (n-1)/2\), there exists a constant \(\delta_{2}=\delta_{2}(n,r)>0\) such that the map \(f \in B_{H^{r}(\mathbb{R}^{n-1};\mathbb{R}^{n-1})}(0,\delta_{2}) \mapsto f/\sqrt{1+\abs{f}^{2}} \in H^{r}(\mathbb{R}^{n-1};\mathbb{R}^{n-1})\) is well-defined and smooth. Similarly, there exist \(\delta_{3}=\delta_{3}(n,r)>0\) and \(\delta_{4}=\delta_{4}(n,r)>0\) such that the maps \(f \in B_{H^{r}(\mathbb{R}^{n-1};\mathbb{R})}(0,\delta_{3}) \mapsto \sqrt{1+f^{2}}-1\in H^{r}(\mathbb{R}^{n-1};\mathbb{R})\) and \(f\in B_{H^{r}(\mathbb{R}^{n-1};\mathbb{R})}(0,\delta_{4}) \mapsto 1/\sqrt{1+f^{2}}-1 \in H^{r}(\mathbb{R}^{n-1};\mathbb{R})\) are well-defined and smooth. Let \(c_{3}=c_{3}(n,r)>0\) be the constant that appears in the seventh item of Theorem \(5.6\) of \cite{leoni2023traveling} with \(d=n-1\) and \(s=r\). By Lemma \(7.1\) of \cite{leoni2023traveling}, since \(r>(n-1)/2\), there exists a constant \(\delta_{5}=\delta_{5}(n,r,b)>0\) such that if \(\eta \in X^{r}(\mathbb{R}^{n-1};\mathbb{R})\) such that \(\norm{\eta}_{X^{r}(\mathbb{R}^{n-1};\mathbb{R})}<\delta_{5}\), then \(\norm{\eta}_{C_{b}^{0}}<b/2\).

    Fix \(s>n/2\) and let \(\delta\) be the smallest of \(\delta_{1}(n,s,b)\), \(\delta_{1}(n,s+1/2,b)\), \(\delta_{1}(n,s+1,b)\), \(\delta_{2}(n,s+3/2)/c_{3}(n,s+5/2)\), \(\delta_{3}(n,s+3/2)/c_{3}(n,s+5/2)\), \(\delta_{4}(n,s+3/2)/c_{3}(n,s+5/2)\), and \(\delta_{5}(n,s+5/2,b)\). First, let us show that the map
    \begin{align}
        &(\gamma,u,\eta) \in \mathbb{R} \times \prescript{}{0}{H}^{s+2}(\Omega;\mathbb{R}^{n}) \times B_{X^{s+5/2}(\mathbb{R}^{n-1};\mathbb{R})}(0,\delta) \mapsto \nonumber\\
        &((u-\gamma e_{1}) \cdot \nabla_{\mathcal{A}}u, \mathfrak{g}(\nabla'\eta,0),J(\nabla_{\mathcal{A}} \cdot u), u \cdot (-\nabla'\eta,1),\gamma \partial_{1}\eta) \nonumber\\
        \in &H^{s}(\Omega;\mathbb{R}^{n}) \times H^{s}(\Omega;\mathbb{R}^{n})  \times H^{s+1}(\Omega;\mathbb{R}) \times H^{s+3/2}(\Sigma_{b};\mathbb{R}) \times H^{s+3/2}(\Sigma_{b};\mathbb{R})
    \end{align}
    is well-defined.
    
    For \(i \in \{1, \dots, n\}\), we note that \([(u-\gamma e_{1})\cdot \nabla_{\mathcal{A}}u]_{i}=u_{j}\mathcal{A}_{jk}\partial_{k}u_{i}-\gamma \mathcal{A}_{1l}\partial_{l}u_{i}\) in Einstein notation. If \(\mathcal{A}_{jk}\) and \(\mathcal{A}_{1l}\) are \(0\) or \(1\), then the terms inside the expression \(u_{j}\mathcal{A}_{jk}\partial_{k}u_{i}-\gamma \mathcal{A}_{1l}\partial_{l}u_{i}\) resolve into \(0\), \(u_{j}\partial_{k}u_{i}\), or \(-\gamma\partial_{l}u_{i}\). Since \(H^{r}(\Omega;\mathbb{R})\) is an algebra for \(r > n/2\), we have that \(u_{j}\partial_{k}u_{i}, -\gamma\partial_{l}u_{i} \in H^{s+1}(\Omega;\mathbb{R})\). Hence, \(u_{j}\mathcal{A}_{jk}\partial_{k}u_{i}-\gamma \mathcal{A}_{1l}\partial_{l}u_{i} \in H^{s+1}(\Omega;\mathbb{R})\). If \(\mathcal{A}_{jk}\) and \(\mathcal{A}_{1l}\) are neither \(0\) nor \(1\), then they are either \(-x_{n}\frac{\partial_{m'}\eta}{b+\eta}\) for some \(m' \in \{1, \dots, n-1\}\), or \(\frac{b}{b+\eta}\). In the former case, \(u_{j}\mathcal{A}_{jk}\partial_{k}u_{i}=u_{j}(-x_{n}\frac{\partial_{m'}\eta}{b+\eta})\partial_{k}u_{i}=\Gamma_{1}(\eta,-x_{n}u_{j}\partial_{k}u_{i}\partial_{m'}\eta)\). By the seventh item of Theorem \(5.6\) of \cite{leoni2023traveling}, \(\partial_{m'} \eta \in H^{s+3/2}(\mathbb{R}^{n-1};\mathbb{R})\) since \(\eta \in B_{X^{s+5/2}(\mathbb{R}^{n-1};\mathbb{R})}(0,\delta)\). Hence, \(-x_{n}u_{j}\partial_{k}u_{i}\partial_{m'}\eta \in H^{s+1}(\Omega;\mathbb{R})\). Since \(X^{s+5/2}(\mathbb{R}^{n-1};\mathbb{R}) \hookrightarrow X^{s+1}(\mathbb{R}^{n-1};\mathbb{R})\) by the third item of Theorem \(5.6\) of \cite{leoni2023traveling}, it follows that \(u_{j}(-x_{n}\frac{\partial_{m'}\eta}{b+\eta})\partial_{k}u_{i} \in H^{s+1}(\Omega;\mathbb{R})\). In the latter case, \(u_{j}\mathcal{A}_{jk}\partial_{k}u_{i}=u_{j}(\frac{b}{b+\eta})\partial_{k}u_{i}=\Gamma_{1}(\eta,bu_{j}\partial_{k}u_{i})\). Since \(b u_{j}\partial_{k}u_{i} \in H^{s+1}(\Omega;\mathbb{R})\), we have \(u_{j}(\frac{b}{b+\eta})\partial_{k}u_{i} \in H^{s+1}(\Omega;\mathbb{R})\). Thus, \(u_{j}\mathcal{A}_{jk}\partial_{k}u_{i} \in H^{s+1}(\Omega;\mathbb{R})\). Similarly, \(-\gamma \mathcal{A}_{1l}\partial_{l}u_{i} \in H^{s+1}(\Omega;\mathbb{R})\). It follows that \((u-\gamma e_{1})\cdot \nabla_{\mathcal{A}}u \in H^{s+1}(\Omega;\mathbb{R}^{n})\).
    
    We note that \(\mathfrak{g}(\nabla'\eta,0)=\mathfrak{g}(\partial_{1}\eta, \dots, \partial_{n-1}\eta,0) \in H^{s+3/2}(\Omega;\mathbb{R}^{n})\). Recall that \(J(\nabla_{\mathcal{A}} \cdot u)=J\partial_{j}^{\mathcal{A}}u_{j}=J\mathcal{A}_{jk}\partial_{k}u_{j}\) in Einstein notation. Depending on the value of \(\mathcal{A}_{jk}\), we note that \(J\mathcal{A}_{jk}\partial_{k}u_{j}\) resolves into \(0\), \(J\partial_{k}u_{j}\), \(J(-x_{n}\frac{\partial_{m'}\eta}{b+\eta})\partial_{k}u_{j}\), or \(J\frac{b}{b+\eta}\partial_{k}u_{j}\), where
    \begin{align}
        &J\partial_{k}u_{j}=\partial_{k}u_{j}+\frac{\eta \partial_{k}u_{j}}{b} \in H^{s+1}(\Omega;\mathbb{R}) \\
        &J(-x_{n}\frac{\partial_{m'}\eta}{b+\eta})\partial_{k}u_{j}=-x_{n}\frac{\partial_{m'}\eta}{b+\eta}\partial_{k}u_{j}+\frac{\eta}{b}(-1)x_{n}\frac{\partial_{m'}\eta}{b+\eta}\partial_{k}u_{j} \in H^{s+1}(\Omega;\mathbb{R}) \\
        &J\frac{b}{b+\eta}\partial_{k}u_{j}=\frac{b}{b+\eta}\partial_{k}u_{j}+\frac{\eta}{b}\frac{b}{b+\eta}\partial_{k}u_{j} \in H^{s+1}(\Omega;\mathbb{R}).
    \end{align}
    Hence, \(J(\nabla_{\mathcal{A}} \cdot u) \in H^{s+1}(\Omega;\mathbb{R})\). We note that \(u \cdot (-\nabla'\eta,1)\mid_{\Sigma_{b}} \in H^{s+3/2}(\Sigma_{b};\mathbb{R})\). Lastly, \(\gamma\partial_{1}\eta \in H^{s+3/2}(\Sigma_{b};\mathbb{R})\).

    From above, we learn that the map \((\gamma,u,\eta)  \in \mathbb{R} \times \prescript{}{0}{H}^{s+2}(\Omega;\mathbb{R}^{n}) \times B_{X^{s+5/2}(\mathbb{R}^{n-1};\mathbb{R})}(0,\delta) \mapsto (u-\gamma e_{1})\cdot \nabla_{\mathcal{A}}u \in H^{s}(\Omega;\mathbb{R}^{n})\) is not only well-defined, but also smooth. Similarly, the maps \(\eta \in B_{X^{s+5/2}(\mathbb{R}^{n-1};\mathbb{R})}(0,\delta)\mapsto \mathfrak{g}(\nabla'\eta,0) \in H^{s}(\Omega;\mathbb{R}^{n})\); \((u,\eta) \in \prescript{}{0}{H}^{s+2}(\Omega;\mathbb{R}^{n}) \times B_{X^{s+5/2}(\mathbb{R}^{n-1};\mathbb{R})}(0,\delta) \mapsto J (\nabla_{\mathcal{A}} \cdot u) \in H^{s+1}(\Omega;\mathbb{R})\); \((u,\eta) \in \prescript{}{0}{H}^{s+2}(\Omega;\mathbb{R}^{n}) \times B_{X^{s+5/2}(\mathbb{R}^{n-1};\mathbb{R})}(0,\delta) \mapsto u \cdot (-\nabla'\eta,1) \mid_{\Sigma_{b}} \in H^{s+3/2}(\Sigma_{b};\mathbb{R})\); and \((u,\eta) \in \prescript{}{0}{H}^{s+2}(\Omega;\mathbb{R}^{n}) \times B_{X^{s+5/2}(\mathbb{R}^{n-1};\mathbb{R})}(0,\delta) \mapsto \gamma\partial_{1}\eta \in H^{s+3/2}(\Sigma_{b};\mathbb{R})\) are not only well-defined, but also smooth.

    Now, let us show that the map
    \begin{align}
        (p,\eta)\in H^{s+1}(\Omega;\mathbb{R}) \times B_{X^{s+5/2}(\mathbb{R}^{n-1})}(0,\delta) \mapsto (\nabla_{\mathcal{A}}p, p\mid_{\Sigma_{b}}\mathcal{N})\in H^{s}(\Omega;\mathbb{R}^{n}) \times H^{s+1/2}(\Sigma_{b};\mathbb{R}^{n})
    \end{align}
    is well-defined. For \(i \in \{1, \dots, n\}\), \([\nabla_{\mathcal{A}}p]_{i}=\mathcal{A}_{ij}\partial_{j}p\) in Einstein notation, so \(\nabla_{\mathcal{A}}p \in H^{s}(\Omega;\mathbb{R}^{n})\). We note that \(p \mid_{\Sigma_{b}}\mathcal{N}(\eta)=p\mid_{\Sigma_{b}}(-\nabla'\eta,1)=(-p\mid_{\Sigma_{b}}\partial_{1}\eta, \dots, -p\mid_{\Sigma_{b}} \partial_{n-1}\eta, p\mid_{\Sigma_{b}}) \in H^{s+1/2}(\Sigma_{b};\mathbb{R}^{n})\). It follows that \((p,\eta)\in H^{s+1}(\Omega;\mathbb{R}) \times B_{X^{s+5/2}(\mathbb{R}^{n-1})}(0,\delta) \mapsto (\nabla_{\mathcal{A}}p, p\mid_{\Sigma_{b}}\mathcal{N})\in H^{s}(\Omega;\mathbb{R}^{n}) \times H^{s+1/2}(\Sigma_{b};\mathbb{R}^{n})\) is well-defined and also smooth.

    Now, let us show that the map
    \begin{align}
        &(\gamma,u,\psi,\eta) \in \mathbb{R} \times \prescript{}{0}{H}^{s+2}(\Omega;\mathbb{R}^{n}) \times \prescript{}{0}{H}^{s+2}(\Omega;\mathbb{R}) \times B_{X^{s+5/2}(\mathbb{R}^{n-1};\mathbb{R})}(0,\delta) \mapsto \nonumber \\
        &(-\gamma\partial_{1}^{\mathcal{A}}\psi,u \cdot \nabla_{\mathcal{A}}\psi, \sigma(\psi)\mathcal{H}(\eta)\mathcal{N}(\eta),\nabla_{\Sigma}^{\mathcal{A}}\sigma(\psi)\sqrt{1+\abs{\nabla'\eta}^{2}})  \nonumber \\
        \in &H^{s}(\Omega;\mathbb{R}) \times H^{s}(\Omega;\mathbb{R}) \times H^{s+1/2}(\Sigma_{b};\mathbb{R}^{n}) \times H^{s+1/2}(\Sigma_{b};\mathbb{R}^{n})
    \end{align}
    is well-defined and \(C^{1}\). We note that \(-\gamma \partial_{1}^{\mathcal{A}}\psi = -\gamma \mathcal{A}_{1j}\partial_{j}\psi\), so \(-\gamma \partial_{1}^{\mathcal{A}}\psi \in H^{s+1}(\Omega;\mathbb{R})\) and the map \((\gamma,\psi,\eta) \in \mathbb{R} \times \prescript{}{0}{H}^{s+2}(\Omega;\mathbb{R}) \times B_{X^{s+5/2}(\mathbb{R}^{n-1};\mathbb{R})}(0,\delta)\mapsto -\gamma\partial_{1}^{\mathcal{A}}\psi \in H^{s}(\Omega;\mathbb{R})\) is smooth. Since \(u \cdot \nabla_{\mathcal{A}}\psi =u_{i}\partial_{i}^{\mathcal{A}}\psi=u_{i}\mathcal{A}_{ij}\partial_{j}\psi\) in Einstein notation, \(u \cdot \nabla_{\mathcal{A}}\psi \in H^{s+1}(\Omega;\mathbb{R})\) and the map \((u,\psi,\eta) \in \prescript{}{0}{H}^{s+2}(\Omega;\mathbb{R}^{n}) \times \prescript{}{0}{H}^{s+2}(\Omega;\mathbb{R}) \times B_{X^{s+5/2}(\mathbb{R}^{n-1};\mathbb{R})}(0,\delta) \mapsto u \cdot \nabla_{\mathcal{A}}\psi \in H^{s}(\Omega;\mathbb{R})\) is smooth.
    
    We note that \(\sigma(\psi)\mathcal{H}(\eta)\mathcal{N}(\eta)=(\sigma \circ \psi)\nabla_{\mathcal{A}}^{'}\cdot (\frac{\nabla'\eta}{\sqrt{1+\abs{\nabla'\eta}^{2}}})(-\nabla'\eta,1)\). Since \(\delta \leq \delta_{2}(n,s+3/2)/c_{3}(n,s+5/2)\), the map \(\eta \in B_{X^{s+5/2}(\mathbb{R}^{n-1};\mathbb{R})}(0,\delta) \mapsto \frac{\nabla'\eta}{\sqrt{1+\abs{\nabla'\eta}^{2}}} \in H^{s+3/2}(\mathbb{R}^{n-1};\mathbb{R}^{n-1})\) is well-defined and smooth. Since the map \((f',\eta) \in H^{s+3/2}(\mathbb{R}^{n-1};\mathbb{R}^{n-1})\times B_{X^{s+5/2}(\mathbb{R}^{n-1};\mathbb{R})}(0,\delta) \mapsto \nabla_{\mathcal{A}}' \cdot f' \in H^{s+1/2}(\mathbb{R}^{n-1};\mathbb{R})\) is well-defined and smooth, the map \(\eta \in B_{X^{s+5/2}(\mathbb{R}^{n-1};\mathbb{R})}(0,\delta)\mapsto \nabla_{\mathcal{A}}' \cdot (\frac{\nabla'\eta}{\sqrt{1+\abs{\nabla'\eta}^{2}}}) \in H^{s+1/2}(\mathbb{R}^{n-1};\mathbb{R})\) is well-defined and smooth. It follows that the map \(\eta \in  B_{X^{s+5/2}(\mathbb{R}^{n-1};\mathbb{R})}(0,\delta) \mapsto \nabla_{\mathcal{A}}'\cdot(\frac{\nabla'\eta}{\sqrt{1+\abs{\nabla'\eta}^{2}}})(-\nabla'\eta,1) \in H^{s+1/2}(\mathbb{R}^{n-1};\mathbb{R}^{n})\) is well-defined and smooth. In Lemma \ref{c2cont}, we set \(d=n-1\), \(p=q=1\), \(\mathfrak{D}=\mathbb{R}^{n-1}\cong \Sigma_{b}\), \(k=s+1>\frac{n}{2}+\frac{1}{2}=\frac{n-1}{2}+1\), and \(f=\sigma\in C_{b}^{s+4}(\mathbb{R};\mathbb{R})\subseteq C_{b}^{s+3}(\mathbb{R};\mathbb{R})\) to see that the map
    \begin{align}
        \psi \in \prescript{}{0}{H}^{s+2}(\Omega;\mathbb{R}) \xmapsto{tr} \psi \mid_{\Sigma_{b}} \in H^{s+3/2}(\Sigma_{b};\mathbb{R}) \xmapsto{\Lambda_{\sigma(\cdot)-\sigma(0)}} \sigma(\psi \mid_{\Sigma_{b}})-\sigma(0) \in H^{s+1}(\Sigma_{b};\mathbb{R})
    \end{align}
    is well-defined and \(C^{1}\). Since the map \(\eta \in B_{X^{s+5/2}(\mathbb{R}^{n-1};\mathbb{R})}(0,\delta) \mapsto \mathcal{H}(\eta)\mathcal{N}(\eta) \in H^{s+1/2}(\Sigma_{b};\mathbb{R}^{n})\) is well-defined and smooth, the map \((\psi,\eta) \in \prescript{}{0}{H}^{s+2}(\Omega;\mathbb{R}) \times B_{X^{s+5/2}(\mathbb{R}^{n-1};\mathbb{R})}(0,\delta) \mapsto \sigma(\psi \mid_{\Sigma_{b}})\mathcal{H}(\eta)\mathcal{N}(\eta)=(\sigma(\psi \mid_{\Sigma_{b}})-\sigma(0))\mathcal{H}(\eta)\mathcal{N}(\eta)+\sigma(0)\mathcal{H}(\eta)\mathcal{N}(\eta) \in H^{s+1/2}(\Sigma_{b};\mathbb{R}^{n})\) is well-defined and \(C^{1}\).

    For \(i \in \{1, \dots, n-1\}\), the \(i\)th component of \(\nabla_{\Sigma}^{\mathcal{A}}\sigma(\psi)\sqrt{1+\abs{\nabla'\eta}^{2}}\) is
    \begin{align} \label{ithcomp}
        \mathcal{A}_{ij}\sigma'(\psi)\partial_{j}\psi\sqrt{1+\abs{\nabla'\eta}^{2}}-\frac{-\partial_{i}\eta}{\sqrt{1+\abs{\nabla'\eta}^{2}}}(-\partial_{k'}\eta)\mathcal{A}_{k' l}\sigma'(\psi)\partial_{l}\psi, \quad k' \in \{1, \dots, n-1\},
    \end{align}
    in Einstein notation. Since \(\delta \leq \delta_{3}(n,s+3/2)/c_{3}(n,s+5/2)\), the map \(\eta \in B_{X^{s+5/2}(\mathbb{R}^{n-1};\mathbb{R})}(0,\delta) \mapsto \sqrt{1+\abs{\nabla'\eta}^{2}}-1 \in H^{s+3/2}(\mathbb{R}^{n-1};\mathbb{R})\) is well-defined and smooth. In Lemma \ref{c2cont}, we set \(d=n-1\), \(p=q=1\), \(\mathfrak{D}=\mathbb{R}^{n-1}\cong\Sigma_{b}\), \(k=s+1>\frac{n}{2}+\frac{1}{2}=\frac{n-1}{2}+1\), and \(f=\sigma' \in C_{b}^{s+3}(\mathbb{R};\mathbb{R})\) to see that the map
    \begin{align}
        \psi \in \prescript{}{0}{H}^{s+2}(\Omega;\mathbb{R}) \xmapsto{tr} \psi \mid_{\Sigma_{b}} \in H^{s+3/2}(\Sigma_{b};\mathbb{R}) \xmapsto{\Lambda_{\sigma'(\cdot)-\sigma'(0)}} \sigma'(\psi \mid_{\Sigma_{b}})-\sigma'(0) \in H^{s+1}(\Sigma_{b};\mathbb{R})
    \end{align}
    is well-defined and \(C^{1}\). Then (\ref{ithcomp}) can be written as
    \begin{align}
        &\mathcal{A}_{ij}(\sigma'(\psi)-\sigma'(0))\partial_{j}\psi(\sqrt{1+\abs{\nabla'\eta}^{2}}-1)+\mathcal{A}_{ij}(\sigma'(\psi)-\sigma'(0))\partial_{j}\psi  \nonumber\\
        &+\mathcal{A}_{ij}\sigma'(0)\partial_{j}\psi(\sqrt{1+\abs{\nabla'\eta}^{2}}-1)+\mathcal{A}_{ij}\sigma'(0)\partial_{j}\psi-\frac{-\partial_{i}\eta}{\sqrt{1+\abs{\nabla'\eta}^{2}}}(-\partial_{k'}\eta)\mathcal{A}_{k' l}\sigma'(0)\partial_{l}\psi \nonumber \\
        &\in H^{s+1/2}(\Sigma_{b};\mathbb{R})
    \end{align}
    and the map sending \((\psi,\eta) \in \prescript{}{0}{H}^{s+2}(\Omega;\mathbb{R}) \times B_{X^{s+5/2}(\mathbb{R}^{n-1};\mathbb{R})}(0,\delta)\) to this element of \(H^{s+1/2}(\Sigma_{b};\mathbb{R})\) is \(C^{1}\). The \(n\)th component of \(\nabla_{\Sigma}^{\mathcal{A}}\sigma(\psi)\sqrt{1+\abs{\nabla'\eta}^{2}}\) is \(\frac{\partial_{i'}\eta}{\sqrt{1+\abs{\nabla'\eta}^{2}}}\mathcal{A}_{i'j}(\sigma'(\psi)-\sigma'(0))\partial_{j}\psi+\frac{\partial_{i'}\eta}{\sqrt{1+\abs{\nabla'\eta}^{2}}}\mathcal{A}_{i'j}\sigma'(0)\partial_{j}\psi \in H^{s+1/2}(\Sigma_{b};\mathbb{R})\) in Einstein notation with \(i' \in \{1, \dots, n-1\}\) and \(j \in \{1, \dots, n\}\), and the map sending \((\psi,\eta) \in \prescript{}{0}{H}^{s+2}(\Omega;\mathbb{R}) \times B_{X^{s+5/2}(\mathbb{R}^{n-1};\mathbb{R})}(0,\delta)\) to this element of \(H^{s+1/2}(\Sigma_{b};\mathbb{R})\) is \(C^{1}\).

    Next, let us show that the map
    \begin{align}
        &(u,\psi,\eta) \in \prescript{}{0}{H}^{s+2}(\Omega;\mathbb{R}^{n}) \times \prescript{}{0}{H}^{s+2}(\Omega;\mathbb{R}) \times B_{X^{s+5/2}(\mathbb{R}^{n-1};\mathbb{R})}(0,\delta) \mapsto \nonumber \\
        &(-\nabla_{\mathcal{A}}\cdot \Gamma(\psi,\mathbb{D}_{\mathcal{A}}u), \nabla_{\mathcal{A}} \cdot \Phi(\psi,\nabla_{\mathcal{A}}\psi),-\Gamma(\psi,\mathbb{D}_{\mathcal{A}}u)\mathcal{N}(\eta),\Phi(\psi,\nabla_{\mathcal{A}}\psi)\cdot \frac{\mathcal{N}(\eta)}{\sqrt{1+\abs{\nabla'\eta}^{2}}}) \nonumber\\
        \in&H^{s}(\Omega;\mathbb{R}^{n}) \times H^{s}(\Omega;\mathbb{R}) \times H^{s+1/2}(\Sigma_{b};\mathbb{R}^{n}) \times H^{s+1/2}(\Sigma_{b};\mathbb{R})
    \end{align}
    is well-defined and \(C^{1}\). In Lemma \ref{c2cont}, we set \(d=n\), \(p=1+n^{2}\), \(q=n^{2}\), \(\mathfrak{D}=\Omega\), \(k=s+1> \frac{n}{2}+1\), and \(f=\Gamma(\cdot,\cdot) \in C_{b}^{s+3}(\mathbb{R}^{1+n^{2}},\mathbb{R}^{n^{2}})\) to see that the map
    \begin{align}
        (u,\psi) \in \prescript{}{0}{H}^{s+2}(\Omega;\mathbb{R}^{n}) \times \prescript{}{0}{H}^{s+2}(\Omega;\mathbb{R}) &\mapsto (\psi,\mathbb{D}_{\mathcal{A}}u) \in \prescript{}{0}{H}^{s+2}(\Omega;\mathbb{R}) \times H^{s+1}(\Omega;\mathbb{R}_{sym}^{n \times n}) \nonumber\\
        &\mapsto \Gamma(\psi,\mathbb{D}_{\mathcal{A}}u) \in H^{s+1}(\Omega;\mathbb{R}^{n^{2}}) \nonumber\\
        &\mapsto -\nabla_{\mathcal{A}} \cdot \Gamma(\psi,\mathbb{D}_{\mathcal{A}}u) \in H^{s}(\Omega;\mathbb{R}^{n})
    \end{align}
    is well-defined and \(C^{1}\). In Lemma \ref{c2cont}, we set \(d=n\), \(p=1+n\), \(q=n\), \(\mathfrak{D}=\Omega\), \(k=s+1>\frac{n}{2}+1\), and \(f=\Phi(\cdot,\cdot) \in C_{b}^{s+3}(\mathbb{R}^{1+n};\mathbb{R}^{n})\) to see that the map
    \begin{align}
        \psi \in \prescript{}{0}{H}^{s+2}(\Omega;\mathbb{R}) &\mapsto (\psi,\nabla_{\mathcal{A}}\psi) \in \prescript{}{0}{H}^{s+2}(\Omega;\mathbb{R}) \times H^{s+1}(\Omega;\mathbb{R}^{n}) \nonumber\\
        &\mapsto \Phi(\psi,\nabla_{\mathcal{A}}\psi) \in H^{s+1}(\Omega;\mathbb{R}^{n}) \nonumber\\
        &\mapsto \nabla_{\mathcal{A}} \cdot \Phi(\psi,\nabla_{\mathcal{A}}\psi) \in H^{s}(\Omega;\mathbb{R})
    \end{align}
    is well-defined and \(C^{1}\). In Lemma \ref{c2cont}, we set \(d=n\), \(p=1+n^{2}\), \(q=n^{2}\), \(\mathfrak{D}=\Omega\), \(k=s+1>\frac{n}{2}+1\), and \(f=\Gamma(\cdot,\cdot) \in C_{b}^{s+3}(\mathbb{R}^{1+n^{2}};\mathbb{R}^{n^{2}})\) to see that the map
    \begin{align}
        (u,\psi) \in \prescript{}{0}{H}^{s+2}(\Omega;\mathbb{R}^{n}) \times \prescript{}{0}{H}^{s+2}(\Omega;\mathbb{R}) &\mapsto (\psi,\mathbb{D}_{\mathcal{A}}u) \in \prescript{}{0}{H}^{s+2}(\Omega;\mathbb{R}) \times H^{s+1}(\Omega;\mathbb{R}_{sym}^{n \times n}) \nonumber\\
        &\mapsto \Gamma(\psi,\mathbb{D}_{\mathcal{A}}u) \in H^{s+1}(\Omega;\mathbb{R}^{n^{2}}) \nonumber\\
        &\xmapsto{tr} \Gamma(\psi \mid_{\Sigma_{b}},\mathbb{D}_{\mathcal{A}}u \mid_{\Sigma_{b}}) \in H^{s+1/2}(\Sigma_{b};\mathbb{R}^{n^{2}}) \nonumber\\
        &\mapsto -\Gamma(\psi \mid_{\Sigma_{b}},\mathbb{D}_{\mathcal{A}}u \mid_{\Sigma_{b}})\mathcal{N}(\eta) \in H^{s+1/2}(\Sigma_{b};\mathbb{R}^{n})
    \end{align}
    is well-defined and \(C^{1}\). In Lemma \ref{c2cont}, we set \(d=n\), \(p=1+n\), \(q=n\), and \(\mathfrak{D}=\Omega\), \(k=s+1>\frac{n}{2}+1\), and \(f=\Phi(\cdot,\cdot) \in C_{b}^{s+3}(\mathbb{R}^{1+n};\mathbb{R}^{n})\) to see that the map
    \begin{align}
        \psi \in \prescript{}{0}{H}^{s+2}(\Omega;\mathbb{R}) &\mapsto (\psi,\nabla_{\mathcal{A}}\psi) \in \prescript{}{0}{H}^{s+2}(\Omega;\mathbb{R}) \times H^{s+1}(\Omega;\mathbb{R}^{n}) \nonumber\\
        &\mapsto \Phi(\psi,\nabla_{\mathcal{A}}\psi) \in H^{s+1}(\Omega;\mathbb{R}^{n}) \nonumber\\
        &\xmapsto{tr} \Phi(\psi \mid_{\Sigma_{b}},\nabla_{\mathcal{A}}\psi \mid_{\Sigma_{b}}) \in H^{s+1/2}(\Sigma_{b};\mathbb{R}^{n}) \nonumber\\
        &\mapsto \Phi(\psi \mid_{\Sigma_{b}},\nabla_{\mathcal{A}}\psi\mid_{\Sigma_{b}}) \cdot \frac{\mathcal{N}(\eta)}{\sqrt{1+\abs{\nabla'\eta}^{2}}} \in H^{s+1/2}(\Sigma_{b};\mathbb{R}) 
    \end{align}
    is well-defined and \(C^{1}\). Lastly, the divergence-trace compatibility condition required for membership in \(\mathcal{Y}^{s}\) is satisfied. The proof is omitted since it is identical to that of the analogous claim in the proof of Theorem \(7.3\) of \cite{leoni2023traveling}.
\end{proof}
\begin{lemma} \label{omegalemma}
    Let \(n \geq 2\) and \(\mathbb{N} \ni s>n/2\). Then there exists \(\delta>0\) such that the map \((\mathcal{T},\eta) \in H^{s+2}(\mathbb{R}^{n};\mathbb{R}_{sym}^{n \times n}) \times B_{X^{s+5/2}(\mathbb{R}^{n-1};\mathbb{R})}(0,\delta) \mapsto (\mathcal{T} \circ \mathfrak{F}_{\eta} \mid_{\Sigma_{b}})\mathcal{N}(\eta) \in H^{s+1/2}(\Sigma_{b};\mathbb{R}^{n})\) is well-defined and \(C^{1}\).
\end{lemma}
\begin{proof}
    Let \(n \geq 2\) and \(\mathbb{N} \ni s>n/2\). Let \(\delta>0\) be the \(\delta_{*} \in (0,1)\) which appears in the third item of Proposition \(5.18\) of \cite{leoni2023traveling}. The lemma follows from the fact that by Corollary \(5.21\) of \cite{leoni2023traveling}, the map \((\mathcal{T},\eta) \in H^{s+2}(\mathbb{R}^{n};\mathbb{R}_{sym}^{n \times n}) \times B_{X^{s+5/2}(\mathbb{R}^{n-1};\mathbb{R})}(0,\delta) \mapsto (\mathcal{T} \circ \mathfrak{F}_{\eta} \mid_{\Sigma_{b}}) \in H^{s+1/2}(\Sigma_{b};\mathbb{R}_{sym}^{n \times n})\) is well-defined and \(C^{1}\), the seventh item of Theorem \(5.6\) of \cite{leoni2023traveling}, and the fact that for any \(1 \leq d \in \mathbb{N}\), \(H^{r}(\mathbb{R}^{d};\mathbb{R})\) is an algebra for \(r>d/2\). 
\end{proof}
\begin{lemma} \label{ol2}
    Let \(n \geq 2\) and \(\mathbb{N} \ni s >n/2\). Then there exists \(\delta>0\) such that the map \((\mathfrak{H},\eta) \in H^{s+2}(\mathbb{R}^{n};\mathbb{R}) \times B_{X^{s+5/2}(\mathbb{R}^{n-1};\mathbb{R})}(0,\delta) \mapsto \mathfrak{H} \circ \mathfrak{F}_{\eta} \mid_{\Sigma_{b}} \in H^{s+1/2}(\Sigma_{b};\mathbb{R})\) is well-defined and \(C^{1}\).
\end{lemma}
\begin{proof}
    This is an immediate consequence of the second item of Corollary \(5.21\) of \cite{leoni2023traveling}.
\end{proof}
\begin{lemma} \label{ol3}
    Let \(n \geq 2\) and \(\mathbb{N} \ni s >n/2\). Then there exists \(\delta>0\) such that the map \((\mathfrak{f},\eta) \in H^{s+1}(\mathbb{R}^{n};\mathbb{R}^{n}) \times B_{X^{s+5/2}(\mathbb{R}^{n-1};\mathbb{R})}(0,\delta) \mapsto \mathfrak{f} \circ \mathfrak{F}_{\eta} \in H^{s}(\Omega;\mathbb{R}^{n})\) is well-defined and \(C^{1}\).
\end{lemma}
\begin{proof}
    This is an immediate consequence of the first item of Corollary \(5.21\) of \cite{leoni2023traveling}.
\end{proof}

\begin{lemma} \label{ol4}
    Let \(n \geq 2\), \(\mathbb{N} \ni s >n/2\), and \(1 \leq d \in \mathbb{N}\). The maps \(f \in H^{s}(\mathbb{R}^{n-1};\mathbb{R}^{n}) \mapsto L_{b}f \in H^{s}(\Omega;\mathbb{R}^{n})\), \(H \in H^{s+1/2}(\mathbb{R}^{n-1};\mathbb{R}^{d})\mapsto S_{b}H \in H^{s+1/2}(\Sigma_{b};\mathbb{R}^{d})\), and \((T,\eta) \in H^{s+1/2}(\mathbb{R}^{n-1};\mathbb{R}_{sym}^{n \times n})\times X^{s+5/2}(\mathbb{R}^{n-1};\mathbb{R}) \mapsto (S_{b}T)\mathcal{N}(\eta) \in H^{s+1/2}(\Sigma_{b};\mathbb{R}^{n})\) are well-defined and smooth.
\end{lemma}
\begin{proof}
    That the map \(f \in H^{s}(\mathbb{R}^{n-1};\mathbb{R}^{n}) \mapsto L_{b}f \in H^{s}(\Omega;\mathbb{R}^{n})\) is well-defined and smooth is a consequence of Lemma \(A.10\) of \cite{leoni2023traveling}. The map \(H \in H^{s+1/2}(\mathbb{R}^{n-1};\mathbb{R}^{d})\mapsto S_{b}H \in H^{s+1/2}(\Sigma_{b};\mathbb{R}^{d})\) is well-defined and smooth because \((x',b) \in \Sigma_{b} \mapsto x' \in \mathbb{R}^{n-1}\) is a smooth diffeomorphism. The map \((T,\eta) \in H^{s+1/2}(\mathbb{R}^{n-1};\mathbb{R}_{sym}^{n \times n})\times X^{s+5/2}(\mathbb{R}^{n-1};\mathbb{R}) \mapsto (S_{b}T)\mathcal{N}(\eta) \in H^{s+1/2}(\Sigma_{b};\mathbb{R}^{n})\) is well-defined and smooth due to Lemma \(A.11\) of \cite{leoni2023traveling}, the seventh item of Theorem \(5.6\) of \cite{leoni2023traveling}, and the fact that for any \(1 \leq d \in \mathbb{N}\), \(H^{r}(\mathbb{R}^{d};\mathbb{R})\) is an algebra for \(r>d/2\).
\end{proof}
\begin{proposition}
    Suppose that \(\mathbb{N} \ni r \geq 1+ \lfloor n/2\rfloor\). Then we have the continuous inclusion
    \begin{align}
        \mathcal{X}^{r} \hookrightarrow C_{0}^{r+1-\lfloor n/2\rfloor}(\Omega;\mathbb{R}^{n}) \times C_{0}^{r+1-\lfloor n/2\rfloor}(\Omega;\mathbb{R}) \times C_{0}^{r-\lfloor n/2 \rfloor}(\Omega;\mathbb{R}) \times C_{0}^{r+1-\lfloor (n-1)/2\rfloor}(\mathbb{R}^{n-1};\mathbb{R}).
    \end{align}
    Moreover, if \((u,\psi,p,\eta) \in \mathcal{X}^{r}\), then $\partial^{\alpha}A(x) \to 0$ as $\abs{x'} \to \infty$, where  $\abs{\alpha} \le r+1-\lfloor n/2\rfloor$ for $A \in \{u,\psi\}$,  $\abs{\alpha}\le  r-\lfloor n/2\rfloor$ for $A=p$, and $\abs{\alpha}\le  r+1-\lfloor (n-1)/2\rfloor$ for $A=\eta$.
\end{proposition}
\begin{proof}
    Let \(\mathbb{N} \ni r \geq 1+ \lfloor n/2\rfloor\). By the fifth item of Theorem \(5.6\) of \cite{leoni2023traveling} with \(s=r+5/2\), \(k=r+1-\lfloor (n-1)/2\rfloor\), and \(d=n-1\), we have the continuous inclusion \(X^{r+5/2}(\mathbb{R}^{n-1};\mathbb{R}) \hookrightarrow C_{0}^{r+1-\lfloor (n-1)/2\rfloor}(\mathbb{R}^{n-1};\mathbb{R})\). Moreover, by the standard Sobolev embedding, \(H^{s+2}(\Omega;\mathbb{R}) \hookrightarrow C_{0}^{r+1-\lfloor n/2\rfloor}(\Omega;\mathbb{R})\) and \(H^{s+1}(\Omega;\mathbb{R})\hookrightarrow C_{0}^{r-\lfloor n/2 \rfloor}(\Omega;\mathbb{R})\).
\end{proof}

\begin{proof}[Proof of Theorem \ref{1.1analog}]
    Suppose that \(n \geq 2\) and \(\mathbb{N} \ni s \geq 1+ \lfloor n/2 \rfloor\). Letting \(\psi \in C_{c}^{\infty}(\mathbb{R};\mathbb{R})\) such that \(0 \leq \psi \leq 1\), \(\psi=1\) on \([-2b,2b]\), and \(\supp(\psi) \subseteq (-3b,3b)\), we define \(\varphi \in C_{c}^{\infty}(\mathbb{R};\mathbb{R})\) by \(\varphi(t)=t\psi(t)\). For any \(\eta \in X^{s+1/2}(\mathbb{R}^{n-1};\mathbb{R})\), we define \(\mathfrak{G}_{\eta}:\mathbb{R}^{n} \to \mathbb{R}^{n}\) by \(\mathfrak{G}_{\eta}(x)=(x',x_{n}+\varphi(x_{n})\eta(x')/b)\). Let \(\delta>0\) be the smaller of the constant \(\delta^{*} \in (0,1)\) which appears in the third item of Proposition \(5.18\) of \cite{leoni2023traveling} and the constant \(\delta=\delta(n,s,b)>0\) that appears in Proposition \ref{8.3analog}. If \(\norm{\eta}_{X^{s+1/2}}<\delta\), then \(\mathfrak{G}_{\eta}\) is a bi-Lipschitz homeomorphism and a \(C^{1}\) diffeomorphism. Letting \(U_{\delta}^{s}=\{(u,\psi,q,\eta) \in \mathcal{X}^{s} \mid \norm{\eta}_{X^{s+5/2}}<\delta\} \subseteq \mathcal{X}^{s}\) and
\begin{align}
    \mathcal{E}^{s}=& \mathbb{R} \setminus \{0\} \times H^{s+1}(\mathbb{R}^{n};\mathbb{R}^{n}) \times H^{s}(\mathbb{R}^{n-1};\mathbb{R}^{n}) \times H^{s+2}(\mathbb{R}^{n};\mathbb{R}_{sym}^{n \times n}) \nonumber\\
        &\times H^{s+1/2}(\mathbb{R}^{n-1};\mathbb{R}_{sym}^{n \times n}) \times H^{s+2}(\mathbb{R}^{n};\mathbb{R}) \times H^{s+1/2}(\mathbb{R}^{n-1};\mathbb{R}),
\end{align}
we define the map \(\widetilde{\Xi}:\mathcal{E}^{s} \times U_{\delta}^{s} \to \mathcal{Y}^{s}\) by
 \begin{align}
     &\widetilde{\Xi}(\gamma,\mathfrak{f},f, \mathcal{T},T,\mathfrak{H},H, u, \psi,p,\eta) \nonumber\\
     =&-(\mathfrak{f}\circ \mathfrak{F}_{\eta}+L_{b}f,0,0,(\mathcal{T} \circ \mathfrak{F}_{\eta} \mid_{\Sigma_{b}}+S_{b}T)\mathcal{N}(\eta),0,-\mathfrak{H} \circ \mathfrak{F}_{\eta}\mid_{\Sigma_{b}}-S_{b}H) \nonumber \\
     &+(-\gamma \partial_{1}^{\mathcal{A}}u + u \cdot \nabla_{\mathcal{A}}u+\mathfrak{g}(\nabla'\eta,0) + \nabla_{\mathcal{A}}p-\nabla_{\mathcal{A}}\cdot \Gamma(\psi,\mathbb{D}_{\mathcal{A}}u),J(\nabla_{\mathcal{A}}\cdot u), \nonumber \\
     &-\gamma \partial_{1}^{\mathcal{A}}\psi + u \cdot \nabla_{\mathcal{A}}\psi +\nabla_{\mathcal{A}} \cdot \Phi(\psi,\nabla_{\mathcal{A}}\psi),-p\mathcal{N}(\eta)+\Gamma(\psi,\mathbb{D}_{\mathcal{A}}u)\mathcal{N}(\eta)-\sigma(\psi)\mathcal{H}(\eta)\mathcal{N}(\eta)\nonumber\\
     &-\nabla_{\Sigma}^{\mathcal{A}}[\sigma(\psi)]\sqrt{1+\abs{\nabla'\eta}^{2}},u \cdot \mathcal{N}(\eta)+\gamma \partial_{1}\eta,\Phi(\psi,\nabla_{\mathcal{A}}\psi)\cdot\frac{\mathcal{N}(\eta)}{\abs{\mathcal{N}(\eta)}}).
 \end{align}
 By Proposition \ref{8.3analog} and Lemmas \ref{omegalemma}, \ref{ol2}, \ref{ol3}, and \ref{ol4}, the map \(\widetilde{\Xi}\) is well-defined and \(C^{1}\). Moreover, \(\widetilde{\Xi}(\gamma,0,0,0,0,0,0,0,0,0,0)=(0,0,0,0,0,0)\). If we denote by \(D_{2}\widetilde{\Xi}:\mathcal{E}^{s} \times U_{\delta}^{s} \to \mathcal{L}(\mathcal{X}^{s};\mathcal{Y}^{s})\) the partial Fr\'echet derivative of \(\widetilde{\Xi}\) with respect to the variable in the second argument, then
 \begin{align}
     D_{2}\widetilde{\Xi}(\gamma,0,0,0,0,0,0,0,0,0,0)(u,\psi,p,\eta)=\Upsilon(u,\psi,p,\eta),
 \end{align}
 where \(\Upsilon:\mathcal{X}^{s} \to \mathcal{Y}^{s}\) is the bijective bounded linear map defined in \eqref{upsilondef}. Therefore,
 \begin{align}
     D_{2}\widetilde{\Xi}(\gamma,0,0,0,0,0,0,0,0,0,0)=\Upsilon\in \mathcal{L}(\mathcal{X}^{s};\mathcal{Y}^{s})
 \end{align}
 is a linear homeomorphism by the open mapping theorem. Due to the implicit function theorem proven in \cite{abraham2012manifolds}, there exist open sets \(\mathcal{U}(\gamma) \subseteq \mathcal{E}^{s}\), \(\mathcal{O}(\gamma) \subseteq U_{\delta}^{s}\), and a \(C^{1}\) Lipschitz map \(\overline{\omega_{\gamma}}:\mathcal{U}(\gamma) \to \mathcal{O}(\gamma) \subseteq U_{\delta}^{s}\) such that \((\gamma,0,0,0,0,0,0) \in \mathcal{U}(\gamma)\); \((0,0,0,0) \in \mathcal{O}(\gamma)\); and for all \((\widetilde{\gamma},\mathfrak{f},f, \mathcal{T},T,\mathfrak{H},H) \in \mathcal{U}(\gamma)\),
 \begin{align}
     \widetilde{\Xi}(\widetilde{\gamma},\mathfrak{f},f, \mathcal{T},T,\mathfrak{H},H,\overline{\omega_{\gamma}}(\widetilde{\gamma},\mathfrak{f},f, \mathcal{T},T,\mathfrak{H},H))=(0,0,0,0,0,0).
 \end{align}
 Moreover, for any \((\widetilde{\gamma},\mathfrak{f},f, \mathcal{T},T,\mathfrak{H},H) \in \mathcal{U}(\gamma)\), \((u,\psi,p,\eta)=\overline{\omega_{\gamma}}(\widetilde{\gamma},\mathfrak{f},f, \mathcal{T},T,\mathfrak{H},H) \in \mathcal{O}(\gamma)\) is the unique solution to
 \begin{align}
     \widetilde{\Xi}(\widetilde{\gamma},\mathfrak{f},f, \mathcal{T},T,\mathfrak{H},H,u,\psi,p,\eta)=(0,0,0,0,0,0)
 \end{align}
 in \(\mathcal{O}(\gamma)\). Let \(\mathcal{U}^{s}=\cup_{\gamma \in \mathbb{R} \setminus \{0\}}\mathcal{U}(\gamma)\subseteq \mathcal{E}^{s}\) and \(\mathcal{O}^{s}=\cup_{\gamma \in \mathbb{R} \setminus \{0\}}\mathcal{O}(\gamma) \subseteq U_{\delta}^{s}\). By construction, \(\mathbb{R} \setminus \{0\} \times \{0\} \times\{0\} \times\{0\} \times\{0\} \times\{0\} \times\{0\} \subseteq \mathcal{U}^{s}\). Let us define \(\overline{\omega}:\mathcal{U}^{s} \to \mathcal{O}^{s}\) by \(\overline{\omega}(\widetilde{\gamma},\mathfrak{f},f, \mathcal{T},T,\mathfrak{H},H)=\overline{\omega_{\gamma}}(\widetilde{\gamma},\mathfrak{f},f, \mathcal{T},T,\mathfrak{H},H)\), where \((\widetilde{\gamma},\mathfrak{f},f, \mathcal{T},T,\mathfrak{H},H) \in \mathcal{U}(\gamma)\) for some \(\gamma \in \mathbb{R} \setminus \{0\}\). This map is well-defined, \(C^{1}\), and locally Lipschitz.
\end{proof}

\begin{proof}[Proof of Theorem \ref{1.3analog}]
    Suppose that \(n \geq 2\) and \(\mathbb{N} \ni s \geq 1+ \lfloor n/2 \rfloor\). For each  \(\gamma \in \mathbb{R} \setminus \{0\}\), let \(\mathcal{U}(\gamma) \subseteq \mathcal{E}^{s}\) and \(\mathcal{O}(\gamma) \subseteq U_{\delta}^{s}\) be the open sets obtained in the proof of Theorem \ref{1.1analog}. Let \(\overline{\omega}:\mathcal{U}^{s} \to \mathcal{O}^{s}\) be the map defined in that proof. If \((\gamma,\mathfrak{f},f, \mathcal{T},T,\mathfrak{H},H) \in \mathcal{U}^{s}\), then \((u,\psi,p,\eta) =\overline{\omega}(\gamma,\mathfrak{f},f, \mathcal{T},T,\mathfrak{H},H)\) solves (\ref{sysflattened}). Let \(v=u \circ \mathfrak{F}_{\eta}^{-1}\), \(\theta=\psi \circ \mathfrak{F}_{\eta}^{-1}\), and \(q=p \circ \mathfrak{F}_{\eta}^{-1}\). By Theorem \(5.17\) of \cite{leoni2023traveling}, we have \(v\in \prescript{}{0}{H}^{s+2}(\Omega_{b+\eta};\mathbb{R}^{n})\), \(\theta\in \prescript{}{0}{H}^{s+2}(\Omega_{b+\eta};\mathbb{R})\), and \(q \in H^{s+1}(\Omega_{b+\eta};\mathbb{R})\). By the usual Sobolev embedding, we have \(v \in C_{b}^{s+1-\lfloor n/2\rfloor}(\Omega_{b+\eta};\mathbb{R}^{n})\), \(\theta \in C_{b}^{s+1-\lfloor n/2\rfloor}(\Omega_{b+\eta};\mathbb{R})\), and \(q \in C_{b}^{s-\lfloor n/2 \rfloor}(\Omega_{b+\eta};\mathbb{R}^{n})\). Since \((u,\psi,q,\eta)\) solves (\ref{sysflattened}), \((v,\theta,q,\eta)\) solves (\ref{sysori}).
\end{proof}

%\printbibliography

\bibliographystyle{abbrv}
\bibliography{bibliography}

@article{leoni2023traveling,
  title={Traveling Wave Solutions to the Free Boundary Incompressible Navier-Stokes Equations},
  author={Leoni, Giovanni and Tice, Ian},
  journal={Comm. Pure Appl. Math.},
  volume={76},
  number={10},
  pages={2474--2576},
  year={2023},
  publisher={Wiley Online Library}
}

@article {leoni2021travelingwavesolutionsfree,
    AUTHOR = {Giovanni Leoni and Ian Tice},
     TITLE = {Traveling wave solutions to the free boundary incompressible Navier-Stokes equations},
      NOTE  = {Preprint, \href{https://arxiv.org/abs/1912.10091}{arXiv:1912.10091}},
      YEAR  = {2021},
    EPRINT = {1912.10091},
    ARCHIVEPREFIX = {arXiv},
    PRIMARYCLASS = {math.AP}
}

@book{abraham2012manifolds,
  title={Manifolds, tensor analysis, and applications},
  author={Abraham, Ralph and Marsden, Jerrold E and Ratiu, Tudor},
  volume={75},
  year={2012},
  publisher={Springer Science \& Business Media}
}

@article{agmon1964estimates,
  title={Estimates near the boundary for solutions of elliptic partial differential equations satisfying general boundary conditions II},
  author={Agmon, Shmuel and Douglis, Avron and Nirenberg, Louis},
  journal={Comm. Pure Appl. Math.},
  volume={17},
  number={1},
  pages={35--92},
  year={1964},
  publisher={Wiley Online Library}
}

@book{Gurtin_Fried_Anand_2010, 
place={Cambridge}, 
title={The Mechanics and Thermodynamics of Continua}, 
publisher={Cambridge University Press}, 
author={Gurtin, Morton E. and Fried, Eliot and Anand, Lallit}, 
year={2010}
}

@article{Toland_1996,
	author = "Toland, J. F.",
	fjournal = "Topological Methods in Nonlinear Analysis",
	issn = "1230-3429",
	journal = "Topol. Methods Nonlinear Anal.",
	mrclass = "35Q35 (35R35 45G10 47H15 47N20 76B15)",
	mrnumber = "1422004",
	mrreviewer = "J{\"u}rgen Socolowsky",
	number = "1",
	pages = "1--48",
	title = "{Stokes waves}",
	url = "https://doi.org/10.12775/TMNA.1996.001",
	volume = "7",
	year = "1996"
}

@article{Groves_2004,
	author = "Groves, M. D.",
	doi = "10.2991/jnmp.2004.11.4.2",
	fjournal = "Journal of Nonlinear Mathematical Physics",
	issn = "1402-9251",
	journal = "J. Nonlinear Math. Phys.",
	mrclass = "76B15 (35J25 35Q35 47J15 76B25)",
	mrnumber = "2097656",
	mrreviewer = "Mariana H\u{a}r\u{a}gu\c{s}",
	number = "4",
	pages = "435--460",
	title = "{Steady water waves}",
	url = "https://doi.org/10.2991/jnmp.2004.11.4.2",
	volume = "11",
	year = "2004"
}

@article{Strauss_2010,
	author = "Strauss, Walter A.",
	doi = "10.1090/S0273-0979-2010-01302-1",
	fjournal = "American Mathematical Society. Bulletin. New Series",
	issn = "0273-0979",
	journal = "Bull. Amer. Math. Soc. (N.S.)",
	mrclass = "76B15 (35Q35 76B03 76B25)",
	mrnumber = "2721042",
	mrreviewer = "Eugen Varvaruca",
	number = "4",
	pages = "671--694",
	title = "{Steady water waves}",
	url = "https://doi.org/10.1090/S0273-0979-2010-01302-1",
	volume = "47",
	year = "2010"
}

@article {HHSTWWW_2022,
    AUTHOR = {Haziot, Susanna V. and Hur, Vera Mikyoung and Strauss, Walter
              A. and Toland, J. F. and Wahl\'{e}n, Erik and Walsh, Samuel and
              Wheeler, Miles H.},
     TITLE = {Traveling water waves---the ebb and flow of two centuries},
   JOURNAL = {Quart. Appl. Math.},
  FJOURNAL = {Quarterly of Applied Mathematics},
    VOLUME = {80},
      YEAR = {2022},
    NUMBER = {2},
     PAGES = {317--401},
      ISSN = {0033-569X},
   MRCLASS = {35C07 (76B15 76B25 76B47)},
  MRNUMBER = {4406719},
       DOI = {10.1090/qam/1614},
       URL = {https://doi.org/10.1090/qam/1614},
}

@article {MR4337506,
    AUTHOR = {Stevenson, Noah and Tice, Ian},
     TITLE = {Traveling wave solutions to the multilayer free boundary
              incompressible {N}avier-{S}tokes equations},
   JOURNAL = {SIAM J. Math. Anal.},
  FJOURNAL = {SIAM Journal on Mathematical Analysis},
    VOLUME = {53},
      YEAR = {2021},
    NUMBER = {6},
     PAGES = {6370--6423},
      ISSN = {0036-1410},
   MRCLASS = {35Q30 (35C07 35N25 35R35 76D33 76D45)},
  MRNUMBER = {4337506},
       DOI = {10.1137/20M1360670},
       URL = {https://doi.org/10.1137/20M1360670},
}

@article {MR4787851,
    AUTHOR = {Stevenson, Noah and Tice, Ian},
     TITLE = {Well-posedness of the stationary and slowly traveling wave
              problems for the free boundary incompressible
              {N}avier-{S}tokes equations},
   JOURNAL = {J. Funct. Anal.},
  FJOURNAL = {Journal of Functional Analysis},
    VOLUME = {287},
      YEAR = {2024},
    NUMBER = {11},
     PAGES = {Paper No. 110617, 85},
      ISSN = {0022-1236,1096-0783},
   MRCLASS = {35Q30 (35C07 35R35 35S05 76D03 76D33)},
  MRNUMBER = {4787851},
       DOI = {10.1016/j.jfa.2024.110617},
       URL = {https://doi.org/10.1016/j.jfa.2024.110617},
}

@article {MR4785303,
    AUTHOR = {Koganemaru, Junichi and Tice, Ian},
     TITLE = {Traveling wave solutions to the free boundary incompressible
              {N}avier-{S}tokes equations with {N}avier boundary conditions},
   JOURNAL = {J. Differential Equations},
  FJOURNAL = {Journal of Differential Equations},
    VOLUME = {411},
      YEAR = {2024},
     PAGES = {381--437},
      ISSN = {0022-0396,1090-2732},
   MRCLASS = {35Q30 (35B30 35C07 35R35 76D03 76D33)},
  MRNUMBER = {4785303},
       DOI = {10.1016/j.jde.2024.07.045},
       URL = {https://doi.org/10.1016/j.jde.2024.07.045},
}

@article {MR4609068,
    AUTHOR = {Koganemaru, Junichi and Tice, Ian},
     TITLE = {Traveling wave solutions to the inclined or periodic free
              boundary incompressible {N}avier-{S}tokes equations},
   JOURNAL = {J. Funct. Anal.},
  FJOURNAL = {Journal of Functional Analysis},
    VOLUME = {285},
      YEAR = {2023},
    NUMBER = {7},
     PAGES = {Paper No. 110057, 75},
      ISSN = {0022-1236,1096-0783},
   MRCLASS = {35Q30 (35C07 35R35 46J10 76D45 76E05)},
  MRNUMBER = {4609068},
       DOI = {10.1016/j.jfa.2023.110057},
       URL = {https://doi.org/10.1016/j.jfa.2023.110057},
}

@article {MR4797733,
    AUTHOR = {Brownfield, John and Nguyen, Huy Q.},
     TITLE = {Slowly traveling gravity waves for {D}arcy flow: existence and
              stability of large waves},
   JOURNAL = {Comm. Math. Phys.},
  FJOURNAL = {Communications in Mathematical Physics},
    VOLUME = {405},
      YEAR = {2024},
    NUMBER = {10},
     PAGES = {Paper No. 222, 25},
      ISSN = {0010-3616,1432-0916},
   MRCLASS = {76D33 (35C07 35Q35 76S05)},
  MRNUMBER = {4797733},
       DOI = {10.1007/s00220-024-05103-6},
       URL = {https://doi.org/10.1007/s00220-024-05103-6},
}

@article{Nguyen_2026,
doi = {10.1088/1361-6544/ae4afd},
url = {https://doi.org/10.1088/1361-6544/ae4afd},
year = {2026},
publisher = {IOP Publishing},
volume = {39},
number = {3},
pages = {035008},
author = {Nguyen, Huy Q},
title = {Large travelling capillary-gravity waves for {D}arcy flow},
journal = {Nonlinearity}
}

@article {nguyen_stevenson_2025,
    AUTHOR = {Nguyen, Huy Q. and Stevenson, Noah},
     TITLE = {On large periodic traveling surface waves in porous media},
      NOTE  = {Preprint, \href{https://arxiv.org/abs/2601.11800}{arXiv:2601.11800}},
      YEAR  = {2026},
    EPRINT = {2601.11800},
    ARCHIVEPREFIX = {arXiv},
    PRIMARYCLASS = {math.AP},
}

@article {nguyen_Banihashemi_2025,
    AUTHOR = {Nguyen, Huy Q. and Banihashemi, Seyed Abdolhamid},
     TITLE = {On large periodic traveling wave solutions to the free boundary Stokes and Navier-Stokes equations },
      NOTE  = {Preprint, \href{https://arxiv.org/abs/2601.14085}{arXiv:2601.14085}},
      YEAR  = {2026},
    EPRINT = {2601.14085},
    ARCHIVEPREFIX = {arXiv},
    PRIMARYCLASS = {math.AP},
}

@article {MR4690615,
    AUTHOR = {Nguyen, Huy Q. and Tice, Ian},
     TITLE = {Traveling wave solutions to the one-phase {M}uskat problem:
              existence and stability},
   JOURNAL = {Arch. Ration. Mech. Anal.},
  FJOURNAL = {Archive for Rational Mechanics and Analysis},
    VOLUME = {248},
      YEAR = {2024},
    NUMBER = {1},
     PAGES = {Paper No. 5, 58},
      ISSN = {0003-9527,1432-0673},
   MRCLASS = {76D27 (35C07 35Q35 35R35 76S05)},
  MRNUMBER = {4690615},
       DOI = {10.1007/s00205-023-01951-z},
       URL = {https://doi.org/10.1007/s00205-023-01951-z},
}

@article {stevenson2023wellposedness,
    AUTHOR = {Stevenson, Noah and Tice, Ian},
     TITLE = {Well-posedness of the traveling wave problem for the free boundary compressible {N}avier-{S}tokes equations},
      NOTE  = {Preprint, \href{https://arxiv.org/abs/2301.00773}{arXiv:2301.00773}},
      YEAR  = {2023},
    EPRINT = {2301.00773},
    ARCHIVEPREFIX = {arXiv},
    PRIMARYCLASS = {math.AP},
}

@article {stevenson2023shallow,
    AUTHOR = {Stevenson, Noah and Tice, Ian},
     TITLE = {The traveling wave problem for the shallow water equations: well-posedness and the limits of vanishing viscosity and surface tension},
      NOTE  = {To appear},
    JOURNAL = {Comm. Math. Phys.},
    EPRINT = {2311.00160},
      ARCHIVEPREFIX = {arXiv},
      PRIMARYCLASS = {math.AP},
}

@article {stev_tice_big_waves,
    AUTHOR = {Stevenson, Noah and Tice, Ian},
     TITLE = {Stationary wave solutions to two dimensional viscous shallow water equations: theory of small and large solutions},
      NOTE  = {Preprint, \href{https://arxiv.org/abs/2502.11899}{arXiv:2502.11899}},
      YEAR  = {2025},
    EPRINT = {2502.11899},
    ARCHIVEPREFIX = {arXiv},
    PRIMARYCLASS = {math.AP},
}

@article {stev_tice_bore_waves,
    AUTHOR = {Stevenson, Noah and Tice, Ian},
     TITLE = {Gravity driven traveling bore wave solutions to the free boundary incompressible Navier-Stokes equations},
      NOTE  = {Preprint, \href{https://arxiv.org/abs/2505.24562}{arXiv:2505.24562}},
      YEAR  = {2025},
    EPRINT = {2505.24562},
    ARCHIVEPREFIX = {arXiv},
    PRIMARYCLASS = {math.AP},
}

@incollection {MR1373218,
    AUTHOR = {Amann, Herbert},
     TITLE = {Heat-conducting incompressible viscous fluids},
 BOOKTITLE = {Navier-{S}tokes equations and related nonlinear problems
              ({F}unchal, 1994)},
     PAGES = {231--243},
 PUBLISHER = {Plenum, New York},
      YEAR = {1995},
      ISBN = {0-306-45118-2},
   MRCLASS = {35Q35 (76A05 76D05)},
  MRNUMBER = {1373218},
MRREVIEWER = {Vladimir\ V.\ Shelukhin},
}

@article {MR2410236,
    AUTHOR = {Cho, Yonggeun and Kim, Hyunseok},
     TITLE = {Existence result for heat-conducting viscous incompressible
              fluids with vacuum},
   JOURNAL = {J. Korean Math. Soc.},
  FJOURNAL = {Journal of the Korean Mathematical Society},
    VOLUME = {45},
      YEAR = {2008},
    NUMBER = {3},
     PAGES = {645--681},
      ISSN = {0304-9914,2234-3008},
   MRCLASS = {35Q35 (35D05 76N10 80A20)},
  MRNUMBER = {2410236},
       DOI = {10.4134/JKMS.2008.45.3.645},
       URL = {https://doi.org/10.4134/JKMS.2008.45.3.645},
}

@article {MR3280824,
    AUTHOR = {Zhang, Xu and Tan, Zhong},
     TITLE = {The global wellposedness of the 3{D} heat-conducting viscous
              incompressible fluids with bounded density},
   JOURNAL = {Nonlinear Anal. Real World Appl.},
  FJOURNAL = {Nonlinear Analysis. Real World Applications. An International
              Multidisciplinary Journal},
    VOLUME = {22},
      YEAR = {2015},
     PAGES = {129--147},
      ISSN = {1468-1218,1878-5719},
   MRCLASS = {35Q35 (35B30)},
  MRNUMBER = {3280824},
       DOI = {10.1016/j.nonrwa.2014.08.001},
       URL = {https://doi.org/10.1016/j.nonrwa.2014.08.001},
}

@article {MR3680944,
    AUTHOR = {Zhong, Xin},
     TITLE = {Global strong solution for 3{D} viscous incompressible heat
              conducting {N}avier-{S}tokes flows with non-negative density},
   JOURNAL = {J. Differential Equations},
  FJOURNAL = {Journal of Differential Equations},
    VOLUME = {263},
      YEAR = {2017},
    NUMBER = {8},
     PAGES = {4978--4996},
      ISSN = {0022-0396,1090-2732},
   MRCLASS = {35Q35 (35B65 35D35 76D05)},
  MRNUMBER = {3680944},
       DOI = {10.1016/j.jde.2017.06.004},
       URL = {https://doi.org/10.1016/j.jde.2017.06.004},
}

@article {MR3828345,
    AUTHOR = {Wang, Wan and Yu, Haibo and Zhang, Peixin},
     TITLE = {Global strong solutions for 3{D} viscous incompressible heat
              conducting {N}avier-{S}tokes flows with the general external
              force},
   JOURNAL = {Math. Methods Appl. Sci.},
  FJOURNAL = {Mathematical Methods in the Applied Sciences},
    VOLUME = {41},
      YEAR = {2018},
    NUMBER = {12},
     PAGES = {4589--4601},
      ISSN = {0170-4214,1099-1476},
   MRCLASS = {35Q30 (35B65 35D35 76D05)},
  MRNUMBER = {3828345},
       DOI = {10.1002/mma.4915},
       URL = {https://doi.org/10.1002/mma.4915},
}

@article {MR3794117,
    AUTHOR = {Xu, Hao and Yu, Haibo},
     TITLE = {Global regularity to the {C}auchy problem of the 3{D} heat
              conducting incompressible {N}avier-{S}tokes equations},
   JOURNAL = {J. Math. Anal. Appl.},
  FJOURNAL = {Journal of Mathematical Analysis and Applications},
    VOLUME = {464},
      YEAR = {2018},
    NUMBER = {1},
     PAGES = {823--837},
      ISSN = {0022-247X,1096-0813},
   MRCLASS = {35Q30 (35B65)},
  MRNUMBER = {3794117},
       DOI = {10.1016/j.jmaa.2018.04.037},
       URL = {https://doi.org/10.1016/j.jmaa.2018.04.037},
}

@article {MR3850090,
    AUTHOR = {Zhong, Xin},
     TITLE = {Global strong solution for viscous incompressible heat
              conducting {N}avier-{S}tokes flows with density-dependent
              viscosity},
   JOURNAL = {Anal. Appl. (Singap.)},
  FJOURNAL = {Analysis and Applications},
    VOLUME = {16},
      YEAR = {2018},
    NUMBER = {5},
     PAGES = {623--647},
      ISSN = {0219-5305,1793-6861},
   MRCLASS = {35Q35 (35B65 35D35 76D05)},
  MRNUMBER = {3850090},
       DOI = {10.1142/S0219530518500069},
       URL = {https://doi.org/10.1142/S0219530518500069},
}

@article {MR3904564,
    AUTHOR = {Xu, Hao and Yu, Haibo},
     TITLE = {Global strong solutions to the 3{D} inhomogeneous
              heat-conducting incompressible fluids},
   JOURNAL = {Appl. Anal.},
  FJOURNAL = {Applicable Analysis. An International Journal},
    VOLUME = {98},
      YEAR = {2019},
    NUMBER = {3},
     PAGES = {622--637},
      ISSN = {0003-6811,1563-504X},
   MRCLASS = {76D03 (35B40 35B45 35Q35 76D05)},
  MRNUMBER = {3904564},
MRREVIEWER = {Luigi\ Carlo\ Berselli},
       DOI = {10.1080/00036811.2017.1399362},
       URL = {https://doi.org/10.1080/00036811.2017.1399362},
}

@article {MR4172578,
    AUTHOR = {Zhong, Xin},
     TITLE = {Global existence and large time behavior of strong solutions
              for 3{D} nonhomogeneous heat conducting {N}avier-{S}tokes
              equations},
   JOURNAL = {J. Math. Phys.},
  FJOURNAL = {Journal of Mathematical Physics},
    VOLUME = {61},
      YEAR = {2020},
    NUMBER = {11},
     PAGES = {111503, 18},
      ISSN = {0022-2488,1089-7658},
   MRCLASS = {35Q30 (76D03 76D05)},
  MRNUMBER = {4172578},
       DOI = {10.1063/5.0012871},
       URL = {https://doi.org/10.1063/5.0012871},
}

@article {MR4423144,
    AUTHOR = {Zhong, Xin},
     TITLE = {Global well-posedness to the {C}auchy problem of
              two-dimensional nonhomogeneous heat conducting
              {N}avier-{S}tokes equations},
   JOURNAL = {J. Geom. Anal.},
  FJOURNAL = {Journal of Geometric Analysis},
    VOLUME = {32},
      YEAR = {2022},
    NUMBER = {7},
     PAGES = {Paper No. 200, 22},
      ISSN = {1050-6926,1559-002X},
   MRCLASS = {76D05 (76D03)},
  MRNUMBER = {4423144},
MRREVIEWER = {Xiaoping\ Zhai},
       DOI = {10.1007/s12220-022-00947-7},
       URL = {https://doi.org/10.1007/s12220-022-00947-7},
}

@article {MR4357119,
    AUTHOR = {Zhong, Xin},
     TITLE = {Global well-posedness to the 3{D} {C}auchy problem of
              nonhomogeneous heat conducting {N}avier-{S}tokes equations
              with vacuum and large oscillations},
   JOURNAL = {J. Math. Fluid Mech.},
  FJOURNAL = {Journal of Mathematical Fluid Mechanics},
    VOLUME = {24},
      YEAR = {2022},
    NUMBER = {1},
     PAGES = {Paper No. 14, 17},
      ISSN = {1422-6928,1422-6952},
   MRCLASS = {35Q30 (76D03 76D05)},
  MRNUMBER = {4357119},
MRREVIEWER = {Tatsu-Hiko\ Miura},
       DOI = {10.1007/s00021-021-00649-0},
       URL = {https://doi.org/10.1007/s00021-021-00649-0},
}

@article {MR4896668,
    AUTHOR = {Dong, Wenchao and Li, Qingyan},
     TITLE = {Global well-posedness for the two-dimensional incompressible
              heat conducting {N}avier-{S}tokes equations with
              temperature-dependent coefficients and vacuum},
   JOURNAL = {J. Math. Phys.},
  FJOURNAL = {Journal of Mathematical Physics},
    VOLUME = {66},
      YEAR = {2025},
    NUMBER = {4},
     PAGES = {Paper No. 041511, 24},
      ISSN = {0022-2488,1089-7658},
   MRCLASS = {35Q30 (35D35 35G61 76D03)},
  MRNUMBER = {4896668},
       DOI = {10.1063/5.0175069},
       URL = {https://doi.org/10.1063/5.0175069},
}

\end{document}